\documentclass[a4paper,12pt]{amsart}

\usepackage[
  margin=30mm,
  marginparwidth=25mm,     
  marginparsep=2mm,       
  bottom=25mm,
  ]{geometry}

\usepackage[latin1]{inputenc}
\usepackage[T1]{fontenc}
\usepackage[english]{babel}
\usepackage{amsmath,amssymb,amsthm,mathtools}

\expandafter\let\expandafter\xbf\csname bfseries \endcsname
\expandafter\let\expandafter\xmd\csname mdseries \endcsname
\let\xbar\bar
\usepackage{allrunes}
\let\bar\xbar
\expandafter\let\csname bfseries \endcsname\xbf
\expandafter\let\csname mdseries \endcsname\xmd

\usepackage{latexsym}
\usepackage{delarray}
\usepackage{bbm}
\usepackage{scrextend}
\usepackage{hyperref}
\usepackage[pdftex,usenames,dvipsnames]{xcolor}
\usepackage{datetime}
\usepackage{mathrsfs}
\usepackage{enumitem}  
\usepackage{tikz}

\newtheorem{Th}{Theorem}[section]
\newtheorem{Prop}[Th]{Proposition}
\newtheorem{Lem}[Th]{Lemma}
\newtheorem{Cor}[Th]{Corollary}
\newtheorem{Rem}[Th]{Remark}
\newtheorem{Def}[Th]{Definition}

\newcommand{\xxi}{\langle \xi\rangle}
\newcommand{\Rl}{\mathbb{R}}

\newcommand{\Rn}{\mathbb{R}^{n}}
\newcommand{\Z}{\mathbb{Z}}

\renewcommand{\d}{\partial}

\def\la{\langle}
\def\ra{\rangle}
\newcommand{\BMO}{\mathrm{BMO}}
\def\bra#1{{\langle{#1}\rangle}}

\newcommand{\jap}[1]{\left\langle #1\right\rangle}
\newcommand{\bmo}{\mathrm{bmo}}
\newcommand{\phase}{\varphi}

\newcommand{\dd}{\, \mathrm{d}}
\newcommand{\ddd}{\,\text{\rm{\mbox{\dj}}}}
\newcommand{\bessel}[1]{(1-\Delta)^{#1/2}}
\renewcommand{\SS}{{\mathscr{S}}}

\newcommand{\at}{\mathfrak{a}}

\newcommand{\brkt}[1]{\left({#1}\right)}
\newcommand{\set}[1]{\left\{{#1}\right\}}
\newcommand{\norm}[1]{\left\Vert#1\right\Vert}
\newcommand{\abs}[1]{\left |#1\right |}

\newcommand{\eps}{\varepsilon}

\newcommand{\eq}[1]{\begin{equation*}
\begin{split}
#1
\end{split}
\end{equation*}}

\newcommand{\nm}[2]{\begin{equation}\label{#1}
\begin{split}
#2
\end{split}
\end{equation}}

\DeclareMathOperator{\supp}{supp}

\allowdisplaybreaks 

\title
[Regularity of oscillatory integrals operators]
{Regularity properties of Schr\"odinger integral operators and general oscillatory integrals}

\author[A. J. Castro]{Alejandro J. Castro}
\author[A. Israelsson]{Anders Israelsson}
\author[W. Staubach]{Wolfgang Staubach}
\author[M. Yerlanov]{Madi Yerlanov}

\address{\newline
       Alejandro J. Castro \newline
       Department of  Mathematics, Nazarbayev University, \newline
		010000 Nur-Sultan, Kazakhstan}
\email{alejandro.castilla@nu.edu.kz}

\address{\newline
       Anders Israelsson, Wolfgang Staubach \newline
       Department of  Mathematics, Uppsala University, \newline
       S-751 06 Uppsala, Sweden}
       \email{anders.israelsson@math.uu.se, wulf@math.uu.se}
       
\address{\newline
       Madi Yerlanov \newline
       Department of  Mathematics, Simon Fraser University, \newline
	   V5A 1S6 Burnaby, B.C., Canada }
\email{myerlano@sfu.ca}       

 \thanks{
A. J. Castro is supported by the Nazarbayev University Faculty Development Competitive Research Grants Program, grant number 110119FD4544.
{W. Staubach is partially supported by a grant from the Crafoord foundation and by a grant from G. S. Magnusons fond, grant number MG2015-0077}}



\begin{document}
\vspace*{-1cm}
\maketitle
\vspace*{-0.6cm}
 \begin{abstract}
We introduce the notion of Schr\"odinger integral operators and prove sharp local and global regularity results for these (including propagators for the quantum mechanical harmonic oscillator). Furthermore we introduce general classes of oscillatory integral operators with inhomogeneous phase functions, whose local and global regularity are also established in classical function spaces (both in the Banach and quasi-Banach scales). The results are then applied to obtain optimal (local in time) estimates for the solution to the Cauchy problem for variable-coefficient Schr\"odinger equations as well as other evolutionary partial differential equations.
 \end{abstract}
\vspace*{-0.2cm}
\tableofcontents
\newpage



\section{Introduction}\label{Sect:Intro}

The goal of this paper is to prove sharp estimates (in classical function spaces) for a wide class of oscillatory integral operators that appear in the theory of partial differential equations and mathematical physics. As an upshot of these, one also obtains optimal regularity for the Cauchy problem for various evolutionary partial differential equations, such as Schr\"odinger-type equations, just to mention one important example. We start by giving an overview of the previously known regularity results for oscillatory integral operators, which is to a large extent biased by their relevance to our current paper.\\

For simplicity, we confine ourselves to oscillatory integral operators of the form
\begin{equation}\label{intro:eq:OIO}
T_a^\varphi f(x) := \int_{\Rl^n} e^{i\varphi(x,\xi)}\, a(x,\xi)\, \widehat f (\xi) \ddd\xi,
\end{equation}

with amplitude $a(x,\xi)$ and phase function $\varphi(x,\xi)$. It was shown by L. H\"ormander {\cite{Hor2}}
and G. I. Eskin {\cite{Esk2}} in the 1970's
that; if $a(x,\xi)$ is smooth and compactly supported in $x$ and belongs to the class $S^{0}_{1,0}(\Rn)$ (see Definition \ref{symbol class Sm}), and if $\varphi\in \mathcal{C}^{\infty}(\Rl^n \times \Rl^n \setminus \{0\})$ is positively homogeneous of degree 1 in $\xi$ and the mixed Hessian matrix of $\varphi(x,\xi)$ has a non-zero determinant on the support of $a(x,\xi)$ (the non-degeneracy condition), then the operator $T_a^\varphi$ is $L^2$-bounded. Later D. Fujiwara
{\cite{Fuj}} showed that if $\varphi\in \mathcal{C}^{\infty}(\Rl^n \times \Rl^n)$ satisfies the condition \begin{equation}\label{fujiwaras condition}
    \abs{\partial^\alpha_\xi \partial^\beta_x \varphi (x,\xi)}
\leq c_{\alpha,\beta} 
, \quad 
|\alpha+\beta|\geq 2,
\end{equation}
and the determinant of the mixed Hessian of $\varphi$ is globally bounded from below by a non-zero constant (the strong non-degeneracy condition), and the amplitude merely belongs to the class $S^{0}_{0,0}(\Rn)$, then $T_a^\varphi$ is $L^2$-bounded. The Fujiwara condition on the phase function may seem rather strong, however its role is to deal with the lack of decay of the amplitude. Of course if one is in the case that was considered by Eskin and H\"ormander, where the amplitude has some decay, then the assumptions on the phase function could be relaxed. Compared to the result of Eskin and H\"ormander, the boundedness result of Fujiwara has the advantage of avoiding the homogeneity assumption on the phase function, and it also avoids the assumption of compact support in the spatial variables and therefore provides a global $L^2$-boundedness result. Inspired by this result of Fujiwara's, D. Dos Santos Ferreira and W. Staubach {\cite{DS}} showed that; if the phase function $\varphi$ is positively homogeneous of degree 1 in $\xi$, is strongly non-degenerate, and satisfies 
\begin{equation*}
	\sup_{(x,\,\xi) \in \Rl^n \times\Rl^n \setminus\{0\}}  |\xi| ^{-1+\vert \alpha\vert}\left | \partial_{\xi}^{\alpha}\partial_{x}^{\beta}\varphi(x,\xi)\right |
	\leq C_{\alpha , \beta},\qquad |\alpha+\beta|\geq 2,
\end{equation*}
	 and if the amplitude $a(x,\xi)\in \mathcal{C}^{\infty}(\Rl^n \times \Rl^n)$ satisfies 
	 \begin{equation*}
	\left|\partial_{\xi}^{\alpha}\partial_{x}^{\beta} a(x,\xi)\right |
	\leq C_{\alpha , \beta} (1+|\xi|)^{m-\rho|\alpha|+\delta|\beta|},
\end{equation*}	  
then one has that $T^\varphi_a$ is globally $L^2$-bounded, provided that $\rho, \delta\in [0,1]$, $\delta\neq 1$ and $m=\min(0,n(\rho-\delta)/2)$, or $\rho\in [0,1],$ $\delta=1$ and $m<n(\rho-1)/2$. This result is sharp.\\

In 2010, E. Cordero, F. Nicola and L. Rodino {\cite{CNR2}} gave an elegant proof of the result of Fujiwara, which completely avoids many of the technicalities (e.g. the use of the Cotlar-Stein lemma) involved in the previous proofs and instead relies on techniques from the theory of modulation spaces. Therefore, one could assert that the $L^2$-regularity of operators $T_a^\varphi$ with smooth amplitudes and smooth non-degenerate phase functions, has been brought to completion. However, the extent of the impact of  {\cite{CNR2}} was not confined to the aforementioned $L^2$-result, and indeed the investigations of Cordero--Nicola--Rodino also paved the way and inspired much activity in the field, and not least, some of the results of this paper. \\

Turning to the problem of $L^p$-regularity for $p\neq 2$, in the 1980's 
J. Peral {\cite{Per}}, 
A. Miyachi {\cite{Miyachi2}}, and
M. Beals {\cite{Bea}}
studied the problem of $L^p$-boundedness of operators of the form \eqref{intro:eq:OIO},  when the phase function is non-degenerate and positively homogeneous of degree one. It was realised that for an $a\in S^m_{1,0}(\Rn)$ the $L^p$-boundedness can (in general) not hold if
$m>-(n-1) |1/p-1/2|$. \\

As a matter of fact, the influential paper by Miyachi {\cite{Miyachi1}} had a decisive impact on the development of the regularity theory of  oscillatory integral operators, as we shall briefly explain. Miyachi essentially considered operators $T_a^\varphi$ with phase functions $\varphi(x,\xi)= x\cdot\xi +|\xi|^k$ for $k>0$ and amplitudes $a(x,\xi)=\sigma(\xi)\in S^{m}_{1,0}(\Rn)$ i.e. satisfying $|\partial^{\alpha}_{\xi} \sigma(\xi)|\leq c_{\alpha}(1+|\xi|)^{m-|\alpha|}.$ Then he showed that 
for $k=1$ the operator $T_a^\varphi$ is $L^p$-bounded for $1<p<\infty$ if and only if  
$m\leq -(n-1)|1/p-1/2|$, and for $k>0$ (but $k\neq 1$) $T_a^\varphi$ is $L^p$-bounded if and only if   
$m\leq -kn |1/p-1/2|.$
Moreover, Miyachi goes beyond these and also proves similar results on the scales of quasi-Banach Hardy spaces and also in Lipschitz (or H\"older) spaces.\\

For $k=1,$ which is intimately connected to the wave equation and yields a Fourier integral operator in the sense of H\"ormander {\cite{Hor2}}, A. Seeger, C. Sogge and E. Stein \cite{SSS}  generalised Miyachi's result to the variable-coefficient setting. More specifically, Miyachi's result for the ordinary wave operator was generalised to the case of $x$-dependent amplitudes (with compact spatial support) and more general homogeneous of degree one non-degenerate phase functions $\varphi(x,\xi)$. Using a novel method which was also partly inspired by C. Fefferman's paper \cite{Feff}, Seeger--Sogge--Stein proved the optimal local $L^p$-boundedness of Fourier integral operators. Extensions to global estimates in more general function spaces were carried out by A. Israelsson, S. Rodr\'iguez-L\'opez, and W. Staubach in \cite{IRS}.
Thus, the investigations mentioned above complete the picture regarding the regularity of Fourier integral operators of the form \eqref{intro:eq:OIO} with non-degenerate homogeneous of degree one phase functions, and amplitudes in $S^m_{1,0}(\Rn)$.\\

Therefore, the natural remaining problem is the extension of the results of Miyachi to the variable-coefficient setting in the case of $k\neq 1$. This amounts to the investigation of the regularity properties of oscillatory integral operators that fall beyond the scope of the theory of Fourier integral operators. As we mentioned earlier, Fujiwara's work on one hand and the work of Cordero--Nicola--Rodino on the other, have clarified the $L^2$-regularity of certain class of oscillatory integral operators. Indeed, the earlier investigations of B. Helffer and  D. Robert {\cite{Helf-Rob}} and those of Helffer {\cite{Helffer}} in connection to the study of propagators for Schr\"odinger equations (for example the harmonic oscillator) have demonstrated the importance of oscillatory integral operators of the form considered in \cite{Fuj}, see e.g. the paper of Cordero--Nicola--Rodino \cite{CoNiRo}. Furthermore, in the remarkable paper {\cite{DN}}, P. D'Ancona and F. Nicola established some sharp estimates between $L^p$ and Sobolev spaces for operators $ T_a^\varphi$ that satisfy the aforementioned Fujiwara condition \eqref{fujiwaras condition}, without using any $L^p$-boundedness results for $ T_a^\varphi$. As a matter of fact, they also admit in their paper that, so far, one doesn't have any $L^p$-regularity theory (much the same as the one available for the Fourier integral operators) for those oscillatory integral operators that are associated to Schr\"odinger equations.\\

One of the goals of this paper is to fill this gap by not only providing an $L^p$-regularity (and indeed even $L^p$--$L^q$, $1< p\leq q<\infty$) theory for the so called {\it{Schr\"odinger integral operators}} which are the oscillatory integrals with phase functions satisfying \eqref{fujiwaras condition} (see Definition \ref{def:SIO}), but also to go beyond these classes of operators and investigate the regularity of general oscillatory integral operators (see Definition \ref{def:OIO}) in classical function spaces. Furthermore, the results that are obtained here also extend the range of validity of the estimates obtained in \cite{DN}.\\

Here it is important to highlight the fact that, just as in the case of Fourier integral operators which had their origin in  P. Lax's construction of parametrices for hyperbolic partial differential equations in his pioneering 1957-paper
\cite{Lax}, the theory of Schr\"odinger integral operators has its origin in the seminal 1978-paper of K. Asada and D. Fujiwara \cite{AF}. In that paper, the authors considered operators intimately related to those considered here and also established their $L^2$-regularity. \\

One of the main features of the results that are obtained in this paper is the abolition of the usual homogeneity assumption in the phase functions. 
 For these phase functions {we will} also allow the amplitudes to merely belong to the class $S^{m}_{0,0}(\Rn),$ as opposed to the usual $S^m_{1,0}(\Rn)$. Finally, we will show that the regularity results that are obtained here are optimal for the specific order of decay $m$ that we choose. It is important to note that, for those oscillatory integral operators that are not Fourier integral operators, choosing an amplitude in the better class $S^m_{1,0}(\Rn)$ would not yield any improvement in the order of decay $m$, which is required for the regularity in various functions spaces. Therefore, our results are not only sharp regarding the order of the amplitudes, but also optimal regarding their type, which is measured by the lower-case indices of the classes of amplitudes.\\
 
Regarding calculi for Schr\"odinger integral operators, Asada and Fujiwara \cite{AF} studied the action of pseudodifferential operators on their class of oscillatory integrals and showed that their class is closed under composition with pseudodifferential operators of order zero. In 2013, E. Cordero, K--H. Gr\"ochenig, F. Nicola and L. Rodino showed in \cite{CGNR} that the class of operators that we here refer to as {Schr\"odinger integral operators (of order zero)} is actually closed under composition.
In this paper, we prove a basic composition theorem, similar to that of H\"ormander's in the Fourier integral operator-setting, for the composition of a pseudodifferential operator and a general oscillatory integral operator. The usefulness of this composition theorem is also clearly demonstrated in the applications to the regularity results in Besov-Lipschitz and Triebel-Lizorkin spaces.  \\

The paper is organised as follows{: In} Section \ref{subsection:definitions} we recall some of the basic definitions and facts from the theory of function spaces. We also define our classes of phase functions, amplitudes and the corresponding operators which will be treated in this paper. In Section \ref{sec:maintheorems} we state the main results of the paper and outline the proofs of the theorems. This includes {both local and global regularity} results in  Besov-Lipschitz and Triebel-Lizorkin spaces, $L^p-L^q$ estimates, and also our parameter-dependent composition theorem. Furthermore in the same section we also provide some examples regarding applications of our results within harmonic analysis and partial differential equations. For example we show the regularity of operators with phase functions of the form  $x\cdot\xi+ t(x)|\xi|^k$ with $0<k\leq 1$ and $x\cdot\xi+ t(x)\jap{\xi}$, with $t(x)$ being smooth and bounded together with all of its derivatives. The former is significant in the study of water-wave equation ($k=\frac{1}{2}$), and the latter example is of significance in the study of Klein-Gordon equations. Since our regularity results are also valid for phase functions of the form $x\cdot\xi+|\xi|^k$ with $0<k<\infty$, this enables us to prove sharp basic estimates (in both Banach and quasi-Banach scales) for the solutions of a large class of dispersive PDEs. Thereafter, we turn to variable-coefficient Schr\"odinger equations and show sharp  estimates in classical function spaces for the solutions of Schr\"odinger equations with quadratic potentials (including the case of the harmonic oscillator).\\

Section \ref{Sect:kernel estimates} is devoted to the basic kernel estimates for oscillatory integral operators. In Section \ref{sec:L2-boundedness} we discuss the $L^2$-regularity of the operators and in Section \ref{sect:lowfreq} we deal with the boundedness of the low frequency portion of the operators in Besov-Lipschitz and Triebel-Lizorkin spaces. Since we will sometimes divide the operators in question into  low frequency,  middle frequency and  high frequency portions, in Section \ref{sect:midfreq} we treat the boundedness of the middle frequency portion of the operators. In Section \ref{sec:local Lp-boundedness} we prove a local $h^p-L^p$ result for the oscillatory integral operators and Section \ref{sect:highfreq} is devoted to the study of the  $h^p-L^p$ boundedness of the high-frequency portion of the operators. The same problem for the Schr\"odinger integral operators is treated in Section \ref{sect:hpLpScho}. In Section \ref{section:left comp of OIO with pseudo} we prove a parameter-dependent composition formula and an expansion for the action of a pseudodifferential operator on an oscillatory integral operator. Section \ref{sec:Besov} and Section \ref{sec:Triebel} are devoted to regularity results in Besov-Lipschitz and Triebel-Lizorkin spaces respectively. The sharpness of the results is discussed in Section \ref{sect:Sharp}.\\

{\bf{Acknowledgements.}}
The authors are grateful to Jorge Betancor for reading through the first draft of the manuscript and for his comments that have led to an overall improvement of the presentation. 

\section{Definitions and preliminaries}\label{subsection:definitions}

\noindent In this section, we will collect all the definitions that will be used throughout this paper. We also state some useful results from both harmonic and microlocal analysis which will be used in the proofs.
\subsection{Notations}
We will denote constants which can be determined by known parameters in a given situation, but whose values are not crucial to the problem at hand, by $C$, or $c$ or $c_\alpha$ and so on. Such parameters in this paper would be, for example, $m$, $p$, $n$,  and the constants connected to the seminorms of various amplitudes or phase functions. The value of the constants may differ
from line to line, but in each instance could be estimated if necessary. We also write $a\lesssim b$ as shorthand for $a\leq Cb$ and $a\sim b$ when $a\lesssim b$ and $b\lesssim a$.\\

Also, we shall denote the normalised Lebesgue measure $\dd \xi / (2\pi)^n$ by $\ddd \xi$, \linebreak$\langle\xi\rangle:= (1+|\xi|^2)^{1/2},$ the space of smooth functions with compact support by $\mathcal{C}_c^\infty(\Rn)$,  the space of smooth functions with bounded derivatives of all orders by $\mathcal{C}_b^\infty(\Rn),$ the Schwartz class of rapidly decreasing smooth functions by $\mathscr{S}(\Rn)$ and the set of non-negative integers $\{0,\,1,\,2,\,\dots\}$ by $\mathbb{Z}_+$. In what follows, we use the notation
$$ \widehat{f}(\xi):= \int_{\mathbb R^n} f(x)\, e^{-ix\cdot\xi}\dd x,$$
for the Fourier transform of the function $f$ and $\xi$ and $\eta$ will denote frequency variables.

\subsection{Function spaces}
We start this section by defining the standard \textit{Littlewood-Paley decomposition} which  is a basic ingredient in our proofs and is also used to define the function spaces that we are concerned with here.

\begin{Def}\label{def:LP}
Let $\psi_0 \in\mathcal C_c^\infty(\Rl^n)$ be equal to $1$ on $B(0,1)$ and have its support in $B(0,2)$. Then let 
$$\psi_j(\xi) := \psi_0 \left (2^{-j}\xi \right )-\psi_0 \left (2^{-(j-1)}\xi \right ),$$
where $j\geq 1$ is an integer and $\psi(\xi) := \psi_1(\xi)$. Then $\psi_j(\xi) = \psi\left (2^{-(j-1)}\xi \right )$ and one has the following Littlewood-Paley partition of unity
\eq{
\sum_{j=0}^\infty \psi_j(\xi) = 1 \quad \text{\emph{for all }}\xi\in\Rl^n .
}
\end{Def}
Observe that $\psi_j$ is supported inside the annulus $\set{\xi\in\Rn:2^{j-1}\leq \abs\xi \leq 2^{j+1}}$. It is sometimes also useful to define a sequence of smooth and compactly supported functions $\Psi_j$ with $\Psi_j=1$ on the support of $\psi_j$ and $\Psi_j=0$ outside a slightly larger compact set. Explicitly, one could set
\begin{equation*}\label{eq:PSI}
\Psi_j 
:= \psi_{j+1}+\psi_j+\psi_{j-1},
\end{equation*}
with $\psi_{-1}:=\psi_{0}$.\\

Next we proceed with the definition of \emph{local Hardy space}, $h^p(\Rl^n)$ due to D. Goldberg, see \protect{\cite{Goldberg}}. This space plays an important role in the paper, since many of the subsequent results will be obtained by means of interpolation with the local Hardy spaces.

\begin{Def}\label{nonatomic hardy}
For $0<p\leq 1$, the local Hardy space $h^p(\Rl^n)$ is the set of distributions $f \in {\mathscr{S}}'(\Rl^n)$ such that
\begin{equation*}\label{eq:hp basic}
	\|f\|_{{h}^p(\Rl^n)} := \brkt{\int_{\Rl^n} \sup_{0<t<1} \abs{\psi_0(tD) f(x)}^p\dd x}^{1/p}<\infty,
\end{equation*}
where $\psi_0$ is given in \emph{Definition \ref{def:LP}}, and for $t\in \Rl$
$$\psi_0 (tD) f(x)
:= \int_{\Rl^n} e^{ix\cdot\xi} \, \psi_{0} (t\xi)\, \widehat{f}(\xi)  \ddd \xi.$$ 
\end{Def}
Another useful definition of the Hardy spaces is based on the so called \emph{atoms} and is given as follows:
\begin{Def}\label{def:hardyspace}
For $0<p\leq 1$, a function $\at$ is called an $h^p$-atom $($or a p-atom for short$)$ if for some $x_0\in \Rl^n$ and $r>0$ the following three conditions are satisfied:
\begin{enumerate}
\item[$i)$] $\supp \at\subset B(x_{0}, r)$,
\item[$ii)$] $ |\at(x)|\leq|B(x_{0}, r)|^{-1/p},$
\item[$iii)$] if $r\leq 1$ 
then 
\smallskip 
$ \int_{\Rl^n} x^{\alpha}\,\at(x)\dd x=0,$
for any multi-index $\alpha$  with 
$|\alpha|\leq [n(1/p-1)]$, and no further condition if $r>1.$ Here $[x]$ denotes the integer part of $x$.
\end{enumerate}

Then one has that a distribution $f\in h^p (\Rl^n)$ has an atomic decomposition
$$
f=\sum_{j}\lambda_{j}\at_{j},
$$
where the $\lambda_{j}$ are constants such that $$ \sum_{j}|\lambda_{j}|^{p}
\sim \Vert f\Vert_{h^p(\mathbb{R}^{n})}^{p},$$ and the $a_{j}$ are $p$-atoms. 
\end{Def}
For $1<p<\infty$ we identify $h^p (\Rl^n)$ with $L^p(\Rl^n).$
The dual of the local Hardy space $h^1 (\Rl^n)$ is the \textit{local} $\BMO(\Rn)$, and is denoted by $\bmo(\Rn)$, which consists of locally integrable functions that verify
$$\Vert f\Vert_{\bmo(\Rl^n)}:= \Vert f\Vert_{\BMO(\Rl^n)}+ \Vert \psi_0 (D) f\Vert_{L^{\infty} (\Rl^n)}<\infty,$$
where $\BMO(\Rl^n)$ is the usual John-Nirenberg space of functions of bounded mean oscillation and $\psi_0$ is the cut-off function introduced in Definition \ref{def:LP}.

Using the Littlewood-Paley decomposition above, we define the \textit{Besov-Lipschitz spaces}. 

\begin{Def}\label{def:Besov}
	Let $0 < p,q \le \infty$ and $s \in {\mathbb R}$. The Besov-Lipschitz spaces are defined by
	\[
	{B}^s_{p,q}(\Rl^n)
	:=
	\Big\{
	f \in {{\SS}'(\Rl^n)} \,:\,
	\|f\|_{{B}^s_{p,q}(\Rl^n)}
	:=
	\Big(
	\sum_{j=0}^\infty
	2^{jq s}\|\psi_j(D)f\|^{q}_{L^p(\Rl^n)}
	\Big)^{1/q}<\infty
	\Big\}.
	\]
\end{Def}

\noindent 
It is worth to mention that for $p=q=\infty$ and $0<s\leq 1$ we obtain the familiar Lipschitz (or H\"older) space $\Lambda^s(\Rl^n)$, i.e. $$B^s_{\infty,\infty}(\Rl^n)= \Lambda^s(\Rl^n).$$

We will also produce boundedness results in the realm of \textit{Triebel-Lizorkin} spaces which can be defined using Littlewood-Paley theory, as follows:

\begin{Def}\label{def:Triebel}
Let $0 < p<\infty$, $0<q \le \infty$ and $s \in {\mathbb R}$. The Triebel-Lizorkin spaces are given by
	\[
	{F}^s_{p,q}(\Rl^n)
	:=
	\Big\{
	f \in {\SS}'(\Rl^n) \,:\,
	\|f\|_{{F}^s_{p,q}(\Rl^n)}
	:=
	\Big\|\Big(
	\sum_{j=0}^\infty
	2^{jq s}|\psi_j(D)f|^{q}
	\Big)^{1/q} \Big\|_{L^p(\Rl^n)} <\infty
	\Big\}.
	\]

\end{Def}
\noindent It is well-known, see e.g. \cite[p. 51]{Trie83}
that
\begin{equation}\label{table of tl spaces}
F^s_{p,q}(\Rl^n)=
\left\{
\begin{array}{llll}
L^{p}(\Rl^n), 
& s=0, 
& 1< p<\infty, 
& q=2,\\
h^p(\Rl^n), 
& s=0, 
& 0<p\leq 1,
& q=2,\\
\mathrm{bmo}(\mathbb{R}^n), 
& s=0, 
& p=\infty, 
& q=2,\\
H^{s,p}(\Rl^n), 
& -\infty <s<\infty, 
& 1< p<\infty, 
& q=2,\\
\end{array}
\right.
\end{equation}
where $H^{s,p}(\Rl^n)$ are various Sobolev and Slobodecskij spaces.

\begin{Rem}
Different choices of the basis $\{\psi_j\}_{j=0}^\infty$ give equivalent
{\emph{(}quasi\emph{)}-norms}
of $B_{p,q}^s(\Rl^n)$ and $F_{p,q}^s(\Rl^n)$  in \emph{Definition \ref{def:Besov}} and \emph{\ref{def:Triebel},} see e.g. \cite[p. 41]{Trie83}. We will use either $\{\psi_j\}_{j=0}^\infty$ or  $\{\Psi_j\}_{j=0}^\infty$ to define the norm of $B_{p,q}^s(\Rl^n)$ and $F_{p,q}^s(\Rl^n)$.
\end{Rem}
Another fact which will be useful to us is that for $-\infty <s<\infty$ and $0<p\leq \infty$
\begin{equation}\label{equality of TL and BL}
B^s_{p,p}(\Rl^n)= F^s_{p,p}(\Rl^n),
\end{equation}
and that one has the two continuous embeddings
\begin{equation}\label{embedding of TL}
F^{s+\varepsilon}_{p,q_0}(\Rl^n)\xhookrightarrow{} F^s_{p ,q_1}(\Rl^n)\quad\text{and}\quad B^{s+\varepsilon}_{p,q_0}(\Rl^n)\xhookrightarrow{} B^s_{p ,q_1}(\Rl^n) 
\end{equation}
for $-\infty <s<\infty$, $0<p< \infty$, $0<q_0,q_1 \leq \infty$ and all $\varepsilon>0$.
Furthermore, for $s'\in \Rl$, the operator $ \brkt{1-\Delta}^{s'/2}$ maps ${F}^s_{p,q}(\Rl^n)$ isomorphically into ${F}^{s-s'}_{p,q}(\Rl^n)$ and ${B}^s_{p,q}(\Rl^n)$ isomorphically into ${B}^{s-s'}_{p,q}(\Rl^n),$ see \cite[p. 58]{Trie83}.\\

We will also make repeated use of the estimate; for and all $s,\, p,\, q$
\begin{equation}\label{poitwise multiiplier}
    \Vert fu\Vert_{A^s_{p,q} (\Rl^n)} \lesssim \Big(\sum_{|\alpha|\leq M} \sup_{x\in \Rl^n} |\partial^\alpha f(x)|\Big)\,\Vert u\Vert_{A^s_{p,q}(\Rl ^n)},
\end{equation}
which is valid for $A=B$ (Besov-Lipschitz spaces) or $A=F$ (Triebel-Lizorkin spaces), and $M\in \Z_+$ large enough, see \cite[p. 229, eq. (9), (10)]{RuSi}. For all the other facts about function spaces that are used in this paper we refer the reader to \cite{Trie83}.\\

We will state the following lemma which is a consequence of a Lemma originally due to J. Peetre \cite{peetre}, which turns out to be useful in proving quasi-Banach Besov-Lipschitz/Triebel-Lizorkin boundedness of the low frequency portions of oscillatory integral operators studied in forthcoming sections. 
\begin{Lem}\label{grafakos lemma 1}
Let $f\in \mathcal C^1(\Rl^n)$ with Fourier support inside the unit ball. Then for every $\rho>n$, and $r \in ( n/\rho,1]$ one has
\eq{
\left (\langle \cdot\rangle^{-\rho} \ast |f|\right )(x)\lesssim \Big (M(|f|^r)(x)\Big ) ^{1/r}, \quad x \in \mathbb{R}^n,
}
where $M$ denotes the Hardy-Littlewood maximal function on $\Rl^n$. 

\end{Lem}
\begin{proof}
As was shown by Peetre, see e.g. \cite[Proposition 2, p. 22]{RuSi}, one has for 
$r\geq n/\rho$ that 
\begin{equation}\label{Peetres inequality}
    \sup_{y\in\Rl^n} \frac{|f(x-y)|}{\langle y\rangle^{\rho}}
\lesssim \Big( M(|f|^r)(x)\Big)^{1/r}.
\end{equation}
Now taking $r\in (n/\rho, 1]$, and using \eqref {Peetres inequality} we obtain
\begin{align*}
|\langle \cdot\rangle^{-\rho} \ast f(x)|
& \lesssim \int_{\Rl^n} \frac{|f(x-y)|}{\langle y\rangle^{\rho} } \dd y
\leq \Big(\sup_{y\in\Rl^n} \frac{|f(x-y)|}{\langle y\rangle^{\rho} }\Big)^{1-r} 
\int_{\Rl^n} \frac{|f(x-y)|^r}{\langle y\rangle^{\rho r} } \dd y
\\ & \lesssim \Big(M(|f|^r)(x)\Big)^{1/r-1} \Big(M(|f|^r)(x)\Big)
 =\Big(M(|f|^r)(x)\Big)^{1/r}. \qedhere
\end{align*}
\end{proof}

In establishing the local boundedness of oscillatory integral operators in the  range $0<p<1$, the following Bernstein-type estimate will be useful. The proof can be found in \cite[p. 22]{Trie83}. 

\begin{Lem}\label{lem:bernstein}
Let $\mathcal{K} \subset \Rl^n$ be a compact set and let $0<p\leq r \leq \infty$. Then
\eq{
\norm {\partial^\alpha f }_{L^r(\Rl^n)} \lesssim \norm{f}_{L^p(\Rl^n)},
}
for all multi-indices $\alpha$ and all $f\in L^p_\mathcal{K} (\Rl^n)$, where 
\begin{equation*}
L_\mathcal{K}^p(\Rl^n) := \left \{ f\in {\SS}' (\Rl^n) \, :\, \norm{f}_{L^p(\Rl^n)} <\infty,  \,\supp \widehat f \subset \mathcal{K}  \right \}.
\end{equation*}
\end{Lem}

\subsection{Oscillatory integral operators}

The class of amplitudes which are the basic building blocks of the oscillatory integral operators, were first introduced by L. H\"ormander in \cite{Hor1}.
\begin{Def}\label{symbol class Sm}
Let $m\in \Rl$ and $0\leq\rho, \delta\leq 1$. An \textit{amplitude} \emph{(}symbol\emph{)} $a(x,\xi)$ in the class $S^m_{\rho,\delta}(\Rl^n)$ is a function $a\in \mathcal{C}^\infty (\Rl^n\times \Rl^n)$ that verifies the estimate
\begin{equation*}
\left |\partial_\xi^\alpha \partial_x^\beta a(x,\xi) \right |\leq C_{\alpha,\beta} \langle\xi\rangle ^{m-\rho|\alpha|+\delta|\beta|},
\end{equation*}
for all multi-indices $\alpha$ and $\beta$ and $(x,\xi)\in \Rl^n\times \Rl^n$.
We shall henceforth refer to $m$ as the order of the amplitude, and $\rho, \delta$ as its type.
\end{Def}

\noindent Given the symbol classes defined above, one associates to the symbol its \textit{Kohn-Nirenberg quantisation} as follows:
\begin{Def}
Let $a$ be a symbol. Define a pseudodifferential operator \emph{(}$\Psi\mathrm{DO}$ for short\emph{)} as the operator
\begin{equation*}
a(x,D)f(x) := \int_{\Rl^n}e^{ix\cdot\xi}\,a(x,\xi)\,\widehat{f}(\xi) \ddd \xi,
\end{equation*}
a priori defined on the Schwartz class $\mathscr{S}(\Rl^n).$
\end{Def}


In order to define the oscillatory integral operators that are studied in this paper, we also define classes of phase functions, which together with the amplitudes of Definition \ref{symbol class Sm} are useful and natural conditions to assume in the study of oscillatory integral operators.

\begin{Def}\label{def:fk}
For $0<k<\infty$, we say that a real-valued \textit{phase function} $\varphi(x,\xi)$ belongs to the class $\textart F^k$, if
$\varphi(x,\xi)\in \mathcal{C}^{\infty}(\Rl^n \times\Rl^n \setminus\{0\})$ and satisfies the following estimates $($depending on the range of $k)$\emph:
\begin{itemize}
    \item for $k \geq 1$, 
        $$ 
\abs{\partial^\alpha_\xi  (\varphi (x,\xi)-x\cdot\xi) }\leq
c_{\alpha} |\xi|^{k-1}, \quad  |\alpha| \geq 1 ,
$$
    \item for $0<k<1$, 
$$ 
\abs{\partial^\alpha_\xi \partial^{\beta}_x  (\varphi (x,\xi)-x\cdot\xi) }\leq
c_{\alpha,\beta} |\xi|^{k-|\alpha|}, \quad |\alpha + \beta | \geq 1 ,
$$
\end{itemize}
for all $x\in \Rl^n$ and $|\xi|\geq 1$.

\end{Def}






 

\begin{Rem}
Allowing a singularity $($in the frequency$)$ at the origin in the phase functions is a natural assumption for both $k\geq 1$ and $k<1$ and is motivated by the \emph{PDE}--applications. Indeed the phase function associated to the wave equation is $x\cdot \xi+ |\xi|$ $($$k=1 $$)$, the phase associated to the water-wave equation is $x\cdot \xi+ |\xi|^{1/2}$ $($$k<1 $$)$, and the phase associated to the capillary waves is $x\cdot \xi+ |\xi|^{3/2}$ $($$k>1 $$)$, all of which are non--smooth at $\xi=0$.
\end{Rem}
 
We will also need to consider phase functions that satisfy a certain {\em non-degeneracy condition}. To this end we have

\begin{Def}
One says that the phase function $\varphi(x,\xi)\in \mathcal{C}^{\infty}(\Rl^n \times\Rl^n \setminus\{0\})$ satisfies the strong non-degeneracy condition \emph{(}or $\varphi$ is $\mathrm{SND}$ for short\emph{)} if
\begin{equation*}\label{eq:SND}
	\Big|\det \Big(\partial^{2}_{x_{j}\xi_{k}}\varphi(x,\xi)\Big) \Big|\geq \delta,
\end{equation*}
for  some $\delta>0$,  all $x\in \mathbb{R}^{n}$ and all $\abs\xi\geq 1 $.\\

In case $\varphi(x,\xi)\in \mathcal{C}^{\infty}(\Rl^n \times\Rl^n)$, then we require that the condition above is  satisfied for all $(x,\xi)\in \Rl^n \times\Rl^n$.
\end{Def}
In order the guarantee that our operators are globally $L^2$-bounded we should also put yet another condition of the phase which we shall henceforth simply refer to as the $L^2$-condition (motivated by D. Fujiwara's result in \cite{Fuj}).
\begin{Def}
One says that the phase function $\varphi(x,\xi)\in \mathcal{C}^{\infty}(\Rl^n \times\Rl^n \setminus\{0\})$ satisfies the $L^2$-condition if
\begin{equation}\label{eq:L2 condition}
	\big |\partial^{\alpha}_\xi \partial^{\beta}_x \varphi (x,\xi)\big |
\leq c_{\alpha,\beta},\qquad  |\alpha| \geq 1, \,  |\beta| \geq 1,
\end{equation}
for all $x\in \mathbb{R}^{n}$ and all $|\xi|\geq 1$.\\

In case $\varphi(x,\xi)\in \mathcal{C}^{\infty}(\Rl^n \times\Rl^n)$ then we require that the condition above is  satisfied for all $(x,\xi)\in \Rl^n \times\Rl^n$.
\end{Def}




\noindent Having the definitions of the amplitudes and the phase functions at hand, one has
\begin{Def}\label{def:OIO}
An oscillatory integral operator $T_a^\varphi$ with amplitude $a\in S^{m}_{\rho, \delta}(\Rl^n)$ and a real valued phase function $\varphi$, is defined \emph{(}once again a-priori on $\mathscr{S}(\Rl^n)$\emph{)} by
\begin{equation}\label{eq:OIO}
T_a^\varphi f(x) := \int_{\Rl^n} e^{i\varphi(x,\xi)}\, a(x,\xi)\, \widehat f (\xi) \ddd\xi.
\end{equation}
If $\varphi\in \textart F^k$ and is $\mathrm{SND}$, then these operators will be referred to as oscillatory integral operators of order $k$.
\end{Def}
The formal adjoint of $T_a^\varphi$ is denoted by   $\brkt{T_a^\varphi}^*$ and is given by 
\nm{eq:adjointOIO}{
\brkt{T_a^\varphi}^* f(x) = \int_{\Rl^n} \int_{\Rl^n} e^{ix\cdot\xi -i\varphi(y,\xi)}\,\overline{a(y,\xi)}\, f (y) \dd y\ddd\xi.
}

To deal with the low frequency portion of an oscillatory integral, which is frequency supported in a neighborhood of the origin, one would need a separate analysis because the phase function might be, and it usually is singular at the origin. This typically doesn't affect the Banach space results so much, but as we shall see, it certainly restricts the ranges of parameters in the quasi-Banach spaces. Therefore, to be able to prove regularity results for the low frequency portions of the operators, one should put a mild condition on the phase functions. From the point of view of the applications into PDE's, this condition will always be satisfied and would not cause any loss of generality. 
\begin{Def}\label{Def:LFmu}
Assume that $\varphi(x,\xi)\in \mathcal{C}^{\infty}(\Rl^n \times \Rl^n \setminus \{0\}) $ is real-valued and $0<\mu\leq 1$. 
 We say that $\varphi$ satisfies the low frequency phase condition of order $\mu$, 
\linebreak\emph($\varphi$ satisfies \emph{LF}$(\mu)$-condition for short\emph),
if one has 
\begin{equation}\label{eq:LFmu}
\abs{\partial^{\alpha}_{\xi}\partial_{x}^{\beta} (\varphi(x,\xi)-x\cdot \xi) }\leq c_{\alpha} |\xi|^{\mu-|\alpha|}, 
\end{equation}
for all $x\in \Rl^n$, $0<|\xi| \leq 2$ and all multi-indices 
$\alpha,\beta$.
\end{Def}

\subsection{Schr\"odinger integral operators}
Another important class of oscillatory integrals is the following:
\begin{Def}\label{def:SIO}
An operator of the form \eqref{eq:OIO} with a real-valued phase function $\varphi (x,\xi)\in \mathcal{C}^\infty (\Rl^n \times \Rl^n)$  that verifies
\begin{equation}\label{Schrodinger phase}
\abs{\partial^\alpha_\xi \partial^\beta_x \varphi (x,\xi)}
\leq c_{\alpha,\beta}, \quad   |\alpha+\beta|\geq 2,
\end{equation}
for all $(x,\xi)\in\Rl^n \times \Rl^n,$ will be referred to as a \textit{Schr\"odinger integral operator}.
\end{Def}

\begin{Rem}
Observe that in one dimension $\sin x\,\sin \xi + \xi^2 +(2\xi +1)x$ is an example of an \emph{SND} phase function satisfying \eqref{Schrodinger phase} which is not in $\textart F^2$.
\end{Rem}

Our motivation for such a name stems from the fact that the solution to the Cauchy problem with initial data $f$, for the free Schr\"odinger equation is given by the operator $e^{it\Delta}f$. Observe that for a fixed time (say $t=1$), the phase function of the oscillatory integral defining the Schr\"odinger semigroup is given by $x\cdot \xi +|\xi|^2$ which satisfies \eqref{Schrodinger phase} and is also SND, and its amplitude is identically equal to one which is trivially in the class $S^0_{1,0}(\Rl^n)$. A less naive example, which once again motivates our choice of designation above, stems for the Cauchy problem for the Schr\"odinger equation associated to the quantum mechanical harmonic oscillator $-\Delta+ |x|^2$. In this case, the solution is given by $e^{it(-\Delta+|\cdot|^2)}f$, which is also a Schr\"odinger integral operator according to Definition \ref{def:SIO} with a phase function which is once again SND and verifies \eqref{Schrodinger phase}, see \cite{Helffer}.\\

For the purpose of proving boundedness results for oscillatory integral operators, it turns out that, in most of the cases, the following order of the amplitude is the critical one, namely
\begin{equation*}\label{eq:criticaldecay}
m_k(p) := -kn \Big|\frac 1p -\frac 12 \Big|,
\end{equation*}
where $0<p\leq\infty$ and $k>0$ stems from the so called $\textart F^k$-condition, which is given in Definition \ref{def:fk}. The corresponding critical order for the Schr\"odinger integral operators will be $m_2(p).$ This means that, we will be able to establish various boundedness results for the oscillatory integral operators (and Schr\"odinger integral operators) when the order of the amplitude is less than or equal to $m_k(p)$ (or $m_2(p)$), respectively.\\ 

A common method throughout the paper will be to split the amplitude $a(x,\xi)$ into several pieces with respect to $\xi$. This is used when there is a singularity at the origin $\xi=0$ that needs to be treated separately. In some cases we divide the amplitude into a low and a high frequency part $$a(x,\xi) = \psi_0(\xi)\,a(x,\xi)+(1-\psi_0(\xi))\,a(x,\xi)=:a_L(x,\xi)+a_H(x,\xi),$$ 
where 
$\psi_0$ is given in Definition \ref{def:LP}.
In other cases we divide the amplitude into three different pieces, a low, middle and high frequency part 
\eq{a(x,\xi) &= \psi_0(\xi)\,a(x,\xi)+(\psi_0(\xi/R)-\psi_0(\xi))\,a(x,\xi)+(1-\psi_0(\xi/R))\,a(x,\xi)\\&=:a_L(x,\xi)+a_M(x,\xi)+a_H(x,\xi),}
where $R$ is some large constant that typically depends only on the dimension and the upper and lower bound of the mixed Hessian of $\varphi$.
\begin{Rem}
We should emphasise here that the conditions that are put on the phases of the oscillatory- and the Schr\"odinger integral operators are quite natural and indeed without the \emph{SND}-condition and boundedness of the mixed derivatives \eqref{eq:L2 condition}, the operators under consideration $($i.e. with inhomogeneous phase functions$)$ are not $($in general$)$ even $L^2$-bounded.  Assuming, say homogeneity of degree one in the frequency variable of the phase function, which is the case of Fourier integral operators, enables one to improve on the order of decay of the amplitude. This is however not possible for the Schr\"odinger- and general oscillatory integral operators. The other conditions on the phase functions are there to guarantee $L^p$-boundedness, and the ability to develop a calculus in order to be able to establish boundedness in Besov-Lipschitz and Triebel-Lizorkin spaces.
\end{Rem}

\section{Main regularity results and applications}\label{sec:maintheorems}
In this section, we gather the main regularity results of this paper and briefly outline the proofs, or rather refer to the relevant sections where the various proofs could be found. At the end of this section, we shall discuss the application of our results to regularity problems in harmonic analysis and theory of partial differential equations.

\subsection{Local regularity results}
This subsection deals with local regularity of both Schr\"odinger integral operators and oscillatory integral operators on Besov-Lipschitz and Triebel-Lizorkin spaces. This, as usual, amounts to study the operators whose amplitude $a(x,\xi)$ is compactly supported in the spatial variables.\\

First, we start by the following basic theorem which is the counterpart of the available local $L^p$-boundedness result in the more familiar context of Fourier integral operators.\\

In what follows, the operator $T^\varphi_a$ will denote an oscillatory integral of the form \eqref{eq:OIO}.

\begin{Th}\label{thm:main1}
Let $k \geq 1$, $p\in(1,\infty)$
and $a(x,\xi) \in S^{m_k(p)}_{0,0}(\Rl^n)$  with compact support in the $x$-variable.
Assume that one of the following assumptions hold true:
\begin{enumerate}
    \item[$i)$] $\varphi(x,\xi) \in \textart F^k\cap \mathcal{C}^\infty(\Rn \times \Rn)$ is \emph{SND} and satisfies the $L^2$-condition \eqref{eq:L2 condition}.
\item[$ii)$]  $\varphi(x,\xi) \in \textart F^k$ is \emph{SND}, satisfies \eqref{eq:L2 condition},  and additionally satisfies the estimate
\begin{equation*}
|\partial^{\alpha}_{\xi}\partial^{\beta}_{x} \varphi(x,\xi)|\leq c_{\alpha,\beta} |\xi|^{-|\alpha|},
\qquad |\alpha+\beta| \geq 1, \,\quad\, 0<|\xi|\leq 2\sqrt n.\end{equation*}
\end{enumerate}
 Then in either case, the operator $T_a^\varphi$  maps $L^p(\Rl^n)$ into $L^p(\Rl^n)$ continuously. In the case $0<k<1$, all the results above are true provided that $a(x,\xi) \in S^{m_k(p)}_{1,0}(\Rl^n)$.
\end{Th}

\textbf{Outline of the proof}.
\begin{itemize}

\item[] For the high frequency portion of the operator, we use Propositions \ref{Th:new} and
\ref{Tj:Tad global} to show that for 
$a(x,\xi)\in S^{m_{k}(p_0)}_{0,0}(\Rl^n)$ (when $k \geq 1$), and for 
$a(x,\xi)\in S^{m_{k}(p_0)}_{1,0}(\Rl^n)$ (when $0<k<1$), the operators $T^{\varphi}_a$ and $\brkt{T^{\varphi}_a}^*$ are bounded from $h^{p_0}(\mathbb{R}^n)$ to $L^{p_0}(\mathbb{R}^n)$ for all $0<p_0 <1$. Now using analytic interpolation to the analytic family of operators in the Hardy space setting due to R. Mac\'ias (see e.g. \cite[Theorem E, p. 597]{coifweiss}), one considers $T^{\varphi}_{a_z}$ and ${(T^{\varphi}_{a_z})}^*$,  with  $0\leq \operatorname{Re} z\leq 1$ and
$$a_z (x,\xi):=|\xi|^{m_k(p_0)\varepsilon -(1+\varepsilon)m_k(p_0) z }\,a (x,\xi),$$
with $\varepsilon = (1/q-1/2)/(1/p_0-1/2)-1$ and $q<1$ chosen such that $[n(1/q-1/2)]=[n/2]$. Now the method of proof of Propositions \ref{Th:new} and
\ref{Tj:Tad global} reveals that $T^{\varphi}_{a_{i\!\operatorname{Im}z}}$ and ${(T^{\varphi}_{a_{i\!\operatorname{Im}z}})}^*$ are bounded from $h^{q}(\mathbb{R}^n)$ to $L^{q}(\mathbb{R}^n)$ with bounds that depend on a positive power of $(1+|\operatorname{Im}z|)$ while Theorem \ref{thm:fujiwara} yields that $T^{\varphi}_{a_{1+i\!\operatorname{Im}z}}$ and ${(T^{\varphi}_{a_{1+i\!\operatorname{Im}z}})}^*$ are bounded from $L^{2}(\mathbb{R}^n)$ to $L^{2}(\mathbb{R}^n)$ with constant bounds independent of $z$. This enables one to interpolate these results in accordance with \cite[Theorem E, p. 597]{coifweiss} to show that $T^{\varphi}_a$ and ${(T^{\varphi}_a)}^*$ are bounded from $h^{p_0}(\mathbb{R}^n)$ to $L^{p_0}(\mathbb{R}^n)$ for all $0<p_0 \leq 2.$ Hence, the claimed $L^p$-boundedness follows by duality.\\

\item[] For the low frequency portion of the operator, when the phase function is smooth we use Proposition \ref{Prop:ThpLp}, Lemma \ref{Lem:smoothlowfreq} and interpolation with Fujiwara's $L^2$-boundedness result in \cite{Fuj} (see the proof of Theorem \ref{thm:fujiwara} for details). For the low frequency portion in the non-smooth case, we just use Lemma \ref{low freq lemma llf 2}.\qed \\ 
\end{itemize}
The next theorem deals with the local regularity of oscillatory integral operators on Besov-Lipschitz and Triebel-Lizorkin spaces.

\begin{Th}\label{thm:main4}
Let $m, s\in\Rl$ and $a(x,\xi) \in S^{m}_{0,0}(\Rl^n)$ with compact support in the $x$-variable.  Assume that $k \geq 1$, $\varphi\in \textart F^k$ \emph{is SND}, satisfies the $L^2$-condition \eqref{eq:L2 condition} and the \emph{LF}$(\mu)$-condition \eqref{eq:LFmu} for some $0<\mu\leq 1$.
Then the following statements hold true:
\begin{enumerate}
\item[$i)$] If 
$p\in (0,\infty]$, 
$q\in (0,\infty]$, then $T_a^\varphi :B_{p,q}^{s+m-m_k(p)}(\Rl^n)\to B_{p,q}^{s}(\Rl^n)$.
\item[$ii)$] If 
$p\in (0,\infty )$, $q\in (0,\infty]$ and $\varepsilon>0$, then $T_a^\varphi :F_{p,q}^{s+m-m_k(p)+\varepsilon}(\Rl^n)\to F_{p,q}^{s}(\Rl^n)$.
\item[$iii)$] If 
$p\in (0,\infty )$, $\min\, (2,p)\leq q\leq \max\,( 2,p)$, then $T_a^\varphi :F_{p,q}^{s+m-m_k(p)}(\Rl^n)\to F_{p,q}^{s}(\Rl^n)$.
\item[$iv)$] $T_a^\varphi :F_{\infty,2}^{s+m-m_k(\infty)}(\Rl^n)\to F_{\infty,2}^{s}(\Rl^n)$.
\end{enumerate}
In the case $0<k<1$, all the results above are true provided that 
$a(x,\xi) \in S^{m}_{1,0}(\Rl^n)$.
\end{Th}

\textbf{Outline of the proof}.
\begin{itemize}
 \item[] $i)$ See Section \ref{sec:Besov}. 
 \item[] $ii)$--$iv)$  See Section \ref{sec:Triebel}. \qed 
\end{itemize}

The following theorem deals with the question of the local regularity of the Schr\"odinger integral operators in Besov-Lipschitz and Triebel-Lizorkin spaces.

\begin{Th}\label{thm:main5}
Let $m, s\in\Rl$ and $a(x,\xi)\in S^{m}_{0,0}(\Rl^n)$ with compact support in the $x$-variable. Assume that
$\varphi$ satisfies \eqref{Schrodinger phase} and is \emph{SND}.
Then the following statements hold true:
\begin{enumerate}
\item[$i)$] If $p\in (0,\infty]$, $q\in (0,\infty]$, then $T_a^\varphi :B_{p,q}^{s+m-m_2(p)}(\Rl^n)\to B_{p,q}^{s}(\Rl^n)$.
\item[$ii)$] If $p\in (0,\infty)$, $q\in (0,\infty]$ and $\varepsilon>0$, then $T_a^\varphi :F_{p,q}^{s+m-m_2(p)+\varepsilon}(\Rl^n)\to F_{p,q}^{s}(\Rl^n)$.
\item[$iii)$] If $p\in (0,\infty)$, $\min\, (2,p)\leq q\leq \max\,( 2,p)$, then $T_a^\varphi :F_{p,q}^{s+m-m_2(p)}(\Rl^n)\to F_{p,q}^{s}(\Rl^n)$.
\item[$iv)$] $T_a^\varphi :F_{\infty,2}^{s+m-m_2(\infty)}(\Rl^n)\to F_{\infty,2}^{s}(\Rl^n)$.
\end{enumerate}
\end{Th}

\textbf{Outline of the proof}.
\begin{itemize}
\item[] $i)$ See Section \ref{sec:Besov}.
\item[] $ii)$--$iv)$ See Section \ref{sec:Triebel}. \qed
\end{itemize}

\subsection{Global regularity results}
In this subsection, we deal with global regularity of both Schr\"odinger integral and oscillatory integral operators on Besov-Lipschitz and Triebel-Lizorkin spaces. We shall see that the global results concerning oscillatory integral operators (but not Schr\"odinger integrals) 
{also require a further restriction of the range of $p$ in case of operators with phase functions that are non-smooth (at the origin) in the frequency variables .}\\

We start with a global $L^p$-boundedness theorem.
\begin{Th}\label{thm:main1globalLp}
Let $k \geq 1$, $p\in(1,\infty)$
and $a(x,\xi) \in S^{m_k(p)}_{0,0}(\Rl^n)$.
Assume that $\varphi(x,\xi) \in \textart F^k$ is \emph{SND} and satisfies the $L^2$-condition \eqref{eq:L2 condition},
 and for some $\mu>0$ and some $R>n$ verifies the estimate
\begin{equation*}\label{extra phase condition 1}
|\partial^{\alpha}_{\xi} (\varphi(x,\xi)-x\cdot\xi)|\leq c_{\alpha} |\xi|^{\mu-|\alpha|},
\qquad |\alpha| \geq 0, \,\quad\, 0<|\xi|\leq 2R.
\end{equation*}
Then the operator $T_a^\varphi$ as defined in \eqref{eq:OIO} maps $L^p(\Rl^n)$ into $L^p(\Rl^n)$ continuously. In the case $0<k<1$, all the results above are true provided that $a(x,\xi) \in S^{m_k(p)}_{1,0}(\Rl^n)$.
\end{Th}


\textbf{Outline of the proof}.
\begin{itemize}
\item[] For the low frequency part of the operator, using $a_L(x,\xi)=\psi_0(\xi/R)\,a(x,\xi)$, one applies 
Lemma \ref{low freq lemma llf 1} with condition \eqref{eq:lf2}.\\

\item[] For the high frequency part, using $a_H(x,\xi)=(1-\psi_0(\xi/R))\,a(x,\xi)$, we shall use Propositions \ref{Th:new} and \ref{Tj:Tad global}. To this end, we break up the operator $T_a^\varphi$ into pieces $T_{a_\ell}^{\tilde\varphi}$ that satisfy $\partial_x^\beta\tilde\varphi(x,e_\ell)\in L^\infty(\Rn)$. To do this we make the following construction: define the set of unit vectors $\set{e_{\ell}}_{\ell=1}^{2n}$ by letting $\set{e_{2\gamma-1}}_{\gamma=1}^n$ be the standard basis in $\Rn$ and $e_{2\gamma} := -e_{2\gamma-1}$ for $1\leq \gamma \leq n$. 

Next let $\chi$ be a non-negative function in $\mathcal{C}^\infty_b(\Rl)$ with 
\eq{
\supp\chi =\set {t\in \Rl;\,{t} \geq 1},\quad \chi({t})\geq 1\text{ if } {t} \geq {R/\sqrt{n}}, 
}
with $R$ as in the statement of the theorem, and let $\chi_
\ell$ be the functions in $\mathcal{C}^\infty_b(\Rl^n)$ defined by  \eq{
\chi_{2\gamma-1}(\xi) =\chi (\xi_\gamma), \quad
\chi_{2\gamma}(\xi) =\chi (-\xi_\gamma), \quad 1\leq\gamma\leq n,
}
where $\xi=(\xi_1,\dots,\xi_n)$.

Furthermore define $\lambda_\ell(\xi) := \chi_\ell(\xi)/\sum_{\ell=1}^{2n}\chi_\ell(\xi)$ for $1\leq\ell\leq 2n$, so that $\lambda_\ell\in \mathcal{C}^\infty(\Rl ^n)$,  and $\sum_{\ell=1}^{2n}\lambda_\ell(\xi) = 1$ for every $\xi\in\Rn\setminus B(0,R)$. Observe that on the $\xi$-support of $a_H(x,\xi)$, the sum $\sum_{\ell=1}^{2n}\chi_\ell(\xi)$ is bounded from below by $1$. This is because of the fact that if $\abs\xi\geq R$, then at least one coordinate $\xi_\gamma$ must satisfy $\abs{\xi_\gamma}\geq {R/\sqrt{n}}$ and hence one of the $\chi_\ell$'s is bounded from below by $1$. This yields that for all multi-indices $\alpha,$ one has $\abs{\partial^\alpha \lambda_\ell(\xi)}\lesssim 1$. Now split 
\eq{
T_{a_H}^\varphi f(x) 
= \sum_{\ell=1}^{2n}\int_{\Rn} e^{i\varphi(x,\xi)}\,a_H(x,\xi)\,\lambda_\ell(\xi)\, \widehat f(\xi) \ddd \xi 
=: \sum_{\ell=1}^{2n} T_{a_\ell}^\varphi f(x).
}
The proof reduces to showing the $L^p$-boundedness of each $T_{a_\ell}^\varphi$. By letting $\tilde\varphi(x,\xi) :=\varphi(x,\xi)-\varphi(x,e_\ell) $, we can write $T_{a_\ell}^\varphi f(x) = e^{i\varphi(x,e_\ell)}\,T_{a_\ell}^{\tilde\varphi} f(x)$ with $\tilde\varphi(x,e_\ell) = 0 $ and since $L^p$-norms are invariant under multiplications by factors of the form $e^{i\varphi(x,e_\ell)}$, the results are unchanged. Now the rest of the argument goes exactly as in the proof of Theorem \ref{thm:main1}, however this does not require compact support in the $x$-variable. In particular Proposition \ref{Tj:Tad global} goes through since the new phase function $\tilde\varphi$ trivially satisfies $\partial_x^\beta\tilde\varphi(x,e_\ell)\in L^\infty(\Rn)$ for every integer $\ell\in[1,2n]$.
\end{itemize}

Next we prove the global Besov-Lipschitz and Triebel-Lizorkin regularity of oscillatory integral operators.

\begin{Th}\label{thm:main4glob}
Let $m, s\in\Rl$ and $a(x,\xi) \in S^{m}_{0,0}(\Rl^n)$.  Assume that $k \geq 1$, $\varphi\in \textart F^k$ \emph{is SND}, satisfies the $L^2$-condition \eqref{eq:L2 condition} and the \emph{LF}$(\mu)$-condition \eqref{eq:LFmu} for some $0<\mu\leq1$.
Then the following statements hold true:
\begin{enumerate}
\item[$i)$] If 
$p\in (n/(n+\mu),\infty]$, 
$q\in (0,\infty]$, then $T_a^\varphi :B_{p,q}^{s+m-m_k(p)}(\Rl^n)\to B_{p,q}^{s}(\Rl^n)$.
\item[$ii)$] If 
$p\in (n/(n+\mu),\infty )$, $q\in (0,\infty]$ and $\varepsilon>0$, then $T_a^\varphi :F_{p,q}^{s+m-m_k(p)+\varepsilon}(\Rl^n)\to F_{p,q}^{s}(\Rl^n)$.
\item[$iii)$] If 
$p\in (n/(n+\mu),\infty )$, $\min\, (2,p)\leq q\leq \max\,( 2,p)$, then $T_a^\varphi :F_{p,q}^{s+m-m_k(p)}(\Rl^n)$ $\to F_{p,q}^{s}(\Rl^n)$.
\item[$iv)$]$T_a^\varphi :F_{\infty,2}^{s+m-m_k(\infty)}(\Rl^n)\to F_{\infty,2}^{s}(\Rl^n)$. 
\end{enumerate}

In the case $0<k<1$, all the results above are true provided that 
$a(x,\xi) \in S^{m}_{1,0}(\Rl^n)$. \\

If one deals with smooth phase functions, i.e. if we assume that $\varphi \in \textart F^k\cap \mathcal{C}^{\infty}(\Rl^n \times \Rl^n)$, is \emph{SND} and verifies the $L^2$-condition \eqref{eq:L2 condition} \emph(both conditions for all $(x,\xi) \in\Rl^n \times \Rl^n)$, and $\abs{\nabla_x \varphi(x,0)} \in L^\infty(\mathbb{R}^n),$ then the range of validity of the results above can be extended to $p>0.$
\end{Th}
\textbf{Outline of the proof}.
\begin{itemize}
 \item[] $i)$ See Section \ref{sec:Besov}. 
 \item[] $ii)$--$iv)$  See Section \ref{sec:Triebel}. \qed 
\end{itemize}

For the case of Schr\"odinger integral operators, as in the case of smooth phase functions treated above, {we only need to have control on $|\nabla_x \varphi(x,0)|$}, instead of the LF($\mu$)-condition \eqref{eq:LFmu} thanks to the smoothness of the phase and assumption \eqref{Schrodinger phase}.

\begin{Th}\label{thm:main5glob}
Let $m, s\in\Rl$ and $a(x,\xi)\in S^{m}_{0,0}(\Rl^n)$. Assume that
$\varphi$ satisfies \eqref{Schrodinger phase}, is \emph{SND} and $|\nabla_x \varphi(x,0)| \in L^\infty(\mathbb{R}^n)$.
Then the following statements hold true:
\begin{enumerate}
\item[$i)$] If $p\in (0,\infty]$, $q\in (0,\infty]$, then $T_a^\varphi :B_{p,q}^{s+m-m_2(p)}(\Rl^n)\to B_{p,q}^{s}(\Rl^n)$.
\item[$ii)$] If $p\in (0,\infty)$, $q\in (0,\infty]$ and $\varepsilon>0$, then $T_a^\varphi :F_{p,q}^{s+m-m_2(p)+\varepsilon}(\Rl^n)\to F_{p,q}^{s}(\Rl^n)$.
\item[$iii)$] If $p\in (0,\infty)$, $\min\, (2,p)\leq q\leq \max\,( 2,p)$, then $T_a^\varphi :F_{p,q}^{s+m-m_2(p)}(\Rl^n)\to F_{p,q}^{s}(\Rl^n)$.
\item[$iv)$] $T_a^\varphi :F_{\infty,2}^{s+m-m_2(\infty)}(\Rl^n)\to F_{\infty,2}^{s}(\Rl^n)$.
\end{enumerate}
\end{Th}

\textbf{Outline of the proof}.
\begin{itemize}
\item[] $i)$ See Section \ref{sec:Besov}.
\item[] $ii)$--$iv)$ See Section \ref{sec:Triebel}. \qed
\end{itemize}

\begin{Rem}\label{Rem:cases}
Using the function space table \eqref{table of tl spaces}, one immediately sees that the above regularity results yield in particular the local and global boundedness of the Schr\"odinger and oscillatory integral operators on 
$L^p(\Rl^n)$, 
$h^p(\Rl^n)$, 
$\bmo(\Rl^n)$, and
$\Lambda^s(\Rl^n)$.
\end{Rem}

\begin{Rem}\label{Rem:varphi0}
In dealing with the $L^p$-boundedness in the smooth case of \emph{Theorem \ref{thm:main4glob}} and {\emph{Theorem \ref{thm:main5glob}} }
the assumption on the boundedness of the derivatives of $\varphi(x,0)$  is superfluous. Indeed, if this is not the case, then we can simply replace $\varphi(x,\xi)$ by $\varphi(x,\xi)-\varphi(x,0)+\varphi(x,0)$. Now the new phase function $\tilde{\varphi}(x,\xi):=\varphi(x,\xi)-\varphi(x,0)$ is also \emph{SND},
verifies \eqref{Schrodinger phase} and last but not least $\tilde{\varphi}(x,0)=0$. 
Since $L^p$-norms are invariant under multiplications by factors of the form $e^{i\varphi(x,0)}$, the results are unchanged.
\end{Rem}

An interesting question here is, whether one can prove global regularity results for Schr\"odinger integral operators  when $|\nabla_x\varphi(x,0)|\notin L^\infty(\Rl^n)$. This case already appears for the phase function associated to the propagator of the harmonic oscillator where $\varphi(x,0)$ exhibits quadratic behaviour. The following theorem provides an answer to this question.

\begin{Th}\label{Th:bonus}
 Assume that
$\varphi$ satisfies \eqref{Schrodinger phase} and is \emph{SND}. Then the following statements hold true:
\begin{itemize}
    \item[$i)$] If $p \in (0,\infty)$ and $a(x,\xi)\in S^{m_2(p)}_{0,0}(\Rl^n)$, then 
    $T_a^\varphi : h^p(\Rl^n) \to h^p(\Rl^n).$ 
 \item[$ii)$] If $m\in\Rl$, $a(x,\xi)\in S^{m}_{0,0}(\Rl^n)$, $p \in [2,\infty)$ and $s \in [m_2(p),0]$, then \linebreak$T_a^\varphi : {H^{s+m-m_2(p),p}}(\Rl^n) \to H^{s,p}(\Rl^n).$ 
    Furthermore, this estimate is sharp with respect to $s$.
\end{itemize}
\end{Th}

\textbf{Outline of the proof}.
\begin{itemize}
\item[] $i)$  See Section \ref{sec:Triebel}.
\item[] $ii)$ Assume that $T_a^\varphi$ is any Schr\"odinger integral operator with $a\in S^m_{0,0}(\Rn)$ $m_2(p)\leq s\leq 0$ for $2\leq p<\infty$. By \cite[Theorem 5.3]{DN}, one has $T_a^\varphi:L^p(\Rn)\to H^{m_2(p),p}(\Rn) $ if $m=0$, which directly generalises to $T_a^\varphi:H^{m,p}(\Rn)\to H^{m_2(p),p}(\Rn) $ for any $m\in\Rl$. It follows from Theorem \ref{thm:schrodinger} that \linebreak$T_a^\varphi:H^{m-m_2(p),p}(\Rn)\to L^{p}(\Rn)$ and hence complex interpolation $H^{s , p}(\Rn)=\left (H^{0 , p}(\Rn), H^{m_2(p) , p}(\Rn)\right )_\theta$ (taking $\theta=s/m_2(p)$) yields the desired estimate.\\
To prove the sharpness in $s$, define the operator $Tf(x) := e^{i|x|^2}f(x)$ and let $a(x,\xi) := \jap\xi^{m}$ and $\varphi(x,\xi) := x\cdot\xi+ \abs x^2$. Using the fact that $T_a^\varphi = T \bessel{m}$, we note that the estimate
 $$\norm{T_a^\varphi f}_{H^{s,p}(\Rl^n)} 
 \lesssim \norm{f}_{H^{s+m-m_2(p),p}(\Rl^n)},$$
 is equivalent to  
 $$\norm{Tf}_{H^{s,p}(\Rl^n)} 
 \lesssim \norm{f}_{H^{s-m_2(p),p}(\Rl^n)}.$$
 Hence, from now on, we can take $m=0$.\\
We start by assuming that $s>0$. 
If $\mu := -s-n/p$ and $s'>0$, then $\bessel{s'}\jap x^\mu \sim \jap x^\mu \in L^p(\Rn) $, but $\bessel{s}T\jap x^\mu \sim \jap x^{\mu+s} = \jap x^{-n/p}\notin L^p(\Rn) $. 
This shows that $T$ does not map $H^{s',p}(\Rn)$ into $H^{s,p}(\Rn)$ continuously for any $s,s'>0$ and in particular, if we choose $s'=s-m_2(p)$, then this is true.\\
We now assume that $s<m_2(p)$. Since $T^*f(x) = e^{-i|x|^2}f(x)$ we see by a duality argument that for any $s,s'<0$, $T$ does not map $H^{s',p}(\Rn)$ into $H^{s,p}(\Rn)$ continuously for any $s,s'<0$. 
If one takes $s'=s -m_2(p)$, then $s<m_2(p)$ implies $s,s'<0$ and this concludes the proof. \qed\\

\end{itemize}
One can also show the sharpness of the results in Theorem \ref{Th:bonus} in a much larger scale, as the following corollary shows: 
\begin{Cor}

If $s<m_2(p)$ or $s>0$, then there is a Schr\"odinger integral operator $T_a^\varphi$ of order $m_2(p)$ that is not $F^s_{p,q}$-- or $B^s_{p,q}$--bounded for $1<p<\infty$ and $0<q\leq \infty$.
\end{Cor}
\begin{proof}
The proofs for $F^s_{p,q}(\Rn)$ and $B^s_{p,q}(\Rn)$ can be done in one single step, so let $A$ denote either $F$ or $B$. We proceed using a proof by contradiction. Assume that $s<m_2(p)$ or $s>0$ and that all $T_a^\varphi$ are $A^s_{p,q}$-bounded. Take $0<\eps < \abs{s/4}$. Then according to the boundedness assumption one has
\eq{
\norm{T_a^\varphi f}_{H^{s-4\eps, p}(\Rn)} 
&\lesssim \norm{T_a^\varphi f}_{F^{s-3\eps}_{p,p}(\Rn)}
= \norm{T_a^\varphi f}_{A^{s-3\eps}_{p,p}(\Rn)} 
\lesssim \norm{T_a^\varphi f}_{A^{s-2\eps}_{p,q}(\Rn)} \\
& \lesssim \norm{f}_{A^{s-2\eps}_{p,q}(\Rn)} 
\lesssim \norm{ f}_{A^{s-\eps}_{p,p}(\Rn)} 
=\norm{f}_{F^{s-\eps}_{p,p}(\Rn)}
\lesssim \norm{ f}_{H^{s,p}(\Rn)}, 
}
using the embeddings in \eqref{equality of TL and BL} and \eqref{embedding of TL}. Now this is a contradiction since $s-4\eps<m_2(p)$ or $s-4\eps>0$ and the Schr\"odinger integral operator  $e^{i|x|^2}\bessel{m_2(p)}$ is not bounded from $H^{s,p}(\Rn)$ to $H^{s-4\eps,p}(\Rn)$ as was shown in the proof of Theorem \ref{Th:bonus} \emph{ii}).
\end{proof}

\subsection{A parameter--dependent composition formula}

The next result describes the action of a parameter-dependent pseudodifferential operator on a general oscillatory integral operator. Its significance is two-fold. On one hand, it provides a step towards a calculus for the oscillatory integral operators. On the  other, it enables one to prove regularity results for the operators on classical functions paces, in both Banach and quasi-Banach scales. The result also generalises the asymptotic expansion that was obtained in \cite{RRS1}.

 \begin{Th}\label{thm:left composition with pseudo}
Let $m,s\in \Rl$ and $\rho\in[0,1]$. Suppose that $ a(x, \xi)\in S_{\rho,0}^m (\Rl^n)$, $b(x,\xi)\in S^s_{1,0}(\Rl^n)$ and $\varphi$ is a phase function that is smooth on $\supp a$ and verifies the conditions
\begin{enumerate}
\item[$i)$]$|\xi| \lesssim |\nabla_x \varphi(x, \xi)| \lesssim |\xi|$ and
\item[$ii)$]for all $|\alpha|, |\beta| \geq 1$, $\abs{\partial_x^\alpha \varphi(x, \xi)}\lesssim  \la \xi \ra$ and $\abs{\partial_\xi^\alpha \partial _x^\beta \varphi (x, \xi)} \lesssim  1$, 
\end{enumerate}
for all $(x, \xi) \in \supp a$.
\noindent For  $0<t\leq 1$ consider the parameter-dependent pseudodifferential operator 
\begin{equation*}
b(x, tD)f(x) := \int_{\Rl^n} e^{ix\cdot \xi}\,b(x,t\xi)\,\widehat f(\xi) \ddd \xi,
\end{equation*}
and the oscillatory integral operator
$$T_{a}^\varphi f(x) := \int_{\Rl^n} e^{i\varphi(x,\xi)}\,a(x, \xi)\,\widehat f(\xi) \ddd\xi.$$
Let $\sigma_t$ be the amplitude of the composition operator 
$T_{\sigma_t}^\varphi
:=b(x, tD)T_{a}^\varphi$ 
given by
\begin{equation*}
\sigma_t(x, \xi) := \iint_{\Rl^n\times \Rl^n} a(y, \xi)\, b (x,t\eta)\,e^{i(x-y)\cdot \eta+i\varphi(y,\xi)-i\varphi(x,\xi)} \ddd\eta \dd y.
\end{equation*}
Then for any $M\geq 1$ and all $0<\eps<1/2$, one can write $\sigma_t$ as

\begin{equation}\label{asymptotic expansion}
\sigma_t(x, \xi) =b(x,t\nabla_{x}\phase(x,\xi))\,a(x,\xi) + \sum_{0<|\alpha| < M}\frac{t^{\eps|\alpha|}}{\alpha!}\, \sigma_{\alpha}(t,x,\xi)+t^{\eps M} r(t,x,\xi),
\end{equation}

where, for all multi-indices $\beta, \gamma$ one has
\begin{align*}
\abs{\partial^{\beta}_{\xi} \partial^{\gamma}_{x}\sigma_{\alpha}(t,x,\xi)} & \lesssim  t^{\min(s,0)}  \bra{\xi}^{s+m-(1/2-\varepsilon)|\alpha|-\rho\abs\beta},\quad \text{for}\,\, 0<|\alpha|<M,  \\\abs{\partial^{\beta}_{\xi} \partial^{\gamma}_{x} r(t,x,\xi)} &  \lesssim t^{\min(s,0)} \bra{\xi}^{s+m-\brkt{1/2- \varepsilon}M-\rho\abs\beta}.
\end{align*}
\end{Th}

\textbf{Outline of the proof}.
\begin{itemize}
\item[] See Section \ref{section:left comp of OIO with pseudo}. \qed
\end{itemize}
\begin{Rem}
We shall frequently use the previous theorem when $t$ is replaced by $2^{-j}$ and 
$b(x, tD)= {\psi}(2^{-j} D)$ with $\psi$ as in \emph{Definition \ref{def:LP}}. This yields the following formula for the composition $\psi(2^{-j} D)T^\varphi_a$.
For any integer $M\geq 1$ and $0<\eps<1/2$ one can write
\begin{equation}\label{eq:highfreq1}
\psi\left (2^{-j}D \right ) T_a ^\varphi 
= \sum_{{|\alpha|< M}} \frac{2^{-j\eps|\alpha|}}{\alpha!}T^{\varphi}_{\sigma_{\alpha,j}} + 2^{-j\eps M}T^\varphi_{r_j},
\end{equation}
 with $\sigma_{0,j}(x,\xi):=\psi\left (2^{-j}\nabla_{x}\phase(x,\xi) \right )a(x,\xi)$ and
\begin{equation}\label{eq:estimate1}
\left |\partial^{\beta}_{\xi} \partial^{\gamma}_{x} \sigma_{\alpha,j}(x,\xi) \right |
\lesssim \bra{\xi}^{m-(1/2-\eps)|\alpha|-\rho\abs\beta}, \qquad |\alpha|\geq 0
\end{equation}
\begin{equation*}\label{eq:estimate2}
\supp_\xi \sigma_{\alpha,j}(x,\xi) 
= \left \{ \xi\in\Rl^n: \ C_1 2^j\leq |\xi| \leq C_2 2^j \right \},
\end{equation*}
\begin{equation}\label{eq:estimate3}
\abs{\partial^{\beta}_{\xi} \partial^{\gamma}_{x}  r_j (x,\xi)} 
\lesssim  \,\bra{\xi}^{m-(1/2-\eps)M-\rho\abs\beta}.
\end{equation}

Moreover, if $a(x,\xi)$ is supported outside the origin in the $\xi$-variable, then $r_j(x,\xi)$ also vanishes in a neighbourhood of $\xi=0$. See the proof of \emph{Theorem \ref{thm:left composition with pseudo}} for the details.
\end{Rem}

\subsection{Global \texorpdfstring{$L^p-L^q$}\ \ estimates}
In this section we state and prove basic global $L^p-L^q$ estimates for the oscillatory integral and the Schr\"odinger integral operators. The $L^p-L^q$ estimates for the oscillatory integral operators are as follows:
\begin{Th}\label{thm:p-q oscillatory}
Let $m, s\in\Rl$ and $a(x,\xi) \in S^{m}_{0,0}(\Rl^n)$.  Assume that $k \geq 1$, $\varphi\in \textart F^k$ \emph{is SND}, satisfies the $L^2$-condition \eqref{eq:L2 condition} and for some $\mu>0$ and some $R>n$ verifies the estimate
\begin{equation*}
|\partial^{\alpha}_{\xi} (\varphi(x,\xi)-x\cdot\xi)|\leq c_{\alpha} |\xi|^{\mu-|\alpha|},
\qquad |\alpha| \geq 0, \,\quad\, 0<|\xi|\leq 2R.\end{equation*}
Then for
$1< p\leq q<\infty$,  $T_a^\varphi :L^p (\Rl^n)\to L^q(\Rl^n)$, provided that $m\leq m_k(q)-n(1/p-1/q)$.
In the case $0<k<1$, all the results above are true provided that 
$a(x,\xi) \in S^{m}_{1,0}(\Rl^n)$.
\end{Th}
\begin{proof}
We write 
$$T_a^\varphi= T_a^\varphi (1-\Delta)^{(m_k(q)-m)/2}(1-\Delta)^{(m-m_k(q))/2}.$$
Then since $T_a^\varphi (1-\Delta)^{(m_k(q)-m)/2}$ is an oscillatory integral operator with an amplitude in the class  $S^{m_k(q)}_{0,0}(\Rl^n)$ for $k \geq 1$ and $S^{m_k(q)}_{1,0}(\Rl^n)$ when $0<k<1$, Theorem \ref{thm:main1globalLp} yields that $$\Vert T_a^\varphi (1-\Delta)^{(m_k(q)-m)/2} u \Vert_{L^q(\Rl^n)}\lesssim \Vert u \Vert_{L^q(\Rl^n)}.$$ 
Therefore applying the Sobolev embedding theorem and taking 
$u=(1-\Delta)^{(m-m_k(q))/2} f$, we obtain
$$\Vert T_a^\varphi f \Vert_{L^q(\Rl^n)}\lesssim \Vert f \Vert_{L^p(\Rl^n)},$$
provided that $1< p\leq q<\infty$ and 
$1/p-1/q \leq (m_k(q)-m)/n.$
\end{proof}
For the Schr\"odinger integral operators we have
\begin{Th}\label{Th:p-q schrodinger}
 Assume that $a(x,\xi) \in S^{m}_{0,0}(\Rl^n)$ and
$\varphi$ satisfies \eqref{Schrodinger phase} and is \emph{SND}. Then for
$1< p\leq q<\infty$,  $T_a^\varphi :L^p (\Rl^n)\to L^q(\Rl^n)$, provided that \linebreak 
$m\leq m_2(q)-n(1/p-1/q)$.
\end{Th}
\begin{proof}
The proof is similar to that of Theorem \ref{thm:p-q oscillatory}. The only difference  is that instead of using  Theorem \ref{thm:main1globalLp}, we use Theorem \ref{thm:main5glob}  part $\emph{iii}$), noting that due to the $L^p-L^q$ nature of our result, no requirement on the gradient of the phase function is needed.
\end{proof}
\subsection{Applications to harmonic analysis and PDEs}

In this subsection, we outline some of the applications of the main results of this paper. We start by giving a couple of basic examples to highlight how the results obtain here can provide boundedness results for operators whose regularity has (hitherto) remained elusive.\\

For any $t(x)\in \mathcal{C}^{\infty}_b(\Rl^n)$ the function 
$$\varphi_1(x,\xi)
:=x\cdot\xi+t(x) \langle \xi\rangle$$ 
 is in $\textart{F}^1$. This example in not covered by the theory of Fourier integral operators due to lack of homogeneity, and exhibits the simplest example of a phase function related to equations of Klein-Gordon type. Moreover if we also choose $t(x)$ such that $|\nabla t(x)|$ is small enough, then this phase function also satisfies the SND-condition. It is also easily checked that this phase verifies the $L^2$-condition and $\partial^\beta_x \varphi_1(x,0)=\partial^\beta t(x)\in L^\infty(\Rl^n).$\\
 
Now Theorem \ref{thm:main4glob} (the part for smooth phase functions) shows that if $m\in \Rl$ and $a(x,\xi)\in S^m_{0,0}(\Rl^n)$, then for the oscillatory integral operator $T^{\varphi_1}_a$, $m_1(p)=\linebreak-n |1/p-1/2|$ and $s\in \Rl,$ one has the following regularity results:
 
 \begin{enumerate}
\item[$i)$] For
$p\in (0,\infty]$, 
$q\in (0,\infty]$, $T_a^{\varphi_1} :B_{p,q}^{s+m-m_1(p)}(\Rl^n)\to B_{p,q}^{s}(\Rl^n)$ continuously.
\item[$ii)$] If 
$p\in (0,\infty )$, $q\in (0,\infty]$ and $\varepsilon>0$, then $T_a^{\varphi_1} :F_{p,q}^{s+m-m_1(p)+\varepsilon}(\Rl^n)\to F_{p,q}^{s}(\Rl^n)$ continuously.
\item[$iii)$] If 
$p\in (0,\infty )$, $\min\, (2,p)\leq q\leq \max\,( 2,p)$, then $T_a^{\varphi_1} :F_{p,q}^{s+m-m_1(p)}(\Rl^n)\to F_{p,q}^{s}(\Rl^n)$ continuously.
\item[$iv)$]$T_a^{\varphi_1} :F_{\infty,2}^{s+m+n/2}(\Rl^n)\to F_{\infty,2}^{s}(\Rl^n)$ continuously.\\
\end{enumerate}
 
Another example is that of
$$\varphi_2(x,\xi)
:=x\cdot\xi+t(x)| \xi|^k$$ 
with $0<k\leq 1$ which is in $\textart F^k.$ Once again, if we choose $t(x)$ such that $|\nabla t(x)|$ is small enough, then $\varphi_2$ is also satisfies the SND-condition. Furthermore we have that  
$$|\partial^{\alpha}_{\xi} \partial^{\beta}_{x}(x\cdot\xi+t(x) | \xi|^k -x\cdot\xi)|\leq c_{\alpha,\beta} |\xi|^{k-|\alpha|}, 
 \quad |\alpha +\beta|\geq 0,
 \quad |\xi|\leq 2,$$
 which yields that the LF($\mu$) condition is satisfied with $\mu=k$. Finally  $$|\partial^{\alpha}_{\xi} \partial^{\beta}_{x}(x\cdot\xi+t(x) | \xi|^k )|\leq c_{\alpha,\beta},
 \quad |\alpha|\geq 1,
 \quad |\beta|\geq 1,
 \quad |\xi|\geq 1,$$
 implies that the $L^2$--condition is also satisfied.
 Thus, once again Theorem \ref{thm:main4glob} shows that if $m\in \Rl$ and  $a(x,\xi)\in S^m_{0,0}(\Rl^n)$, then for the oscillatory integral operator $T^{\varphi_2}_a$, 
 $m_k(p)=-kn |1/p-1/2|$ and $s\in \Rl,$ one has similar regularity results in Besov-Lipschitz and Triebel-Lizorkin spaces, as above with the only difference that $m_1(p)$ is replaced by $m_k(p)$ and 
the range of validity of the results in $p$ has to be taken larger than $n/(n+k)$. \\







 

The applications to partial differential equations concern local and global Besov-Lipschitz and Triebel-Lizorkin estimates for solutions to dispersive partial differential equations.
\noindent First, let us consider the basic example of a dispersive equation in $\mathbb{R}^{n+1}$
\begin{equation}\label{dispers eq}
     \left\{ \begin{array}{lll}
        i \partial_t u (x,t)+\phi(D) u(x,t)=0, 
        & x\in\mathbb{R}^n, \, t\not=0,\\
         u(x,0)=f (x),
         & x\in\mathbb{R}^n, \\
      \end{array} \right.
  \end{equation}
where $\widehat{\phi(D) u}(\xi,t)= \phi(\xi)\, \widehat{u}(\xi,t).$ It is well-known that the solution to this Cauchy problem is given by
\begin{equation}\label{eq:morebetterer}
u(x,t)
= \int_{\mathbb{R}^n} e^{ix\cdot \xi + it\phi(\xi)} \,\widehat{f}(\xi)\, \ddd \xi.
\end{equation}

\begin{Th}\label{Th:dispersive equations}
Assume that $0<k<\infty$, $\phi(\xi)\in \mathcal{C}^{\infty}(\mathbb{R}^n \setminus \{0\})$ and 
\begin{equation}\label{eq:simplified phase}
\abs{\partial^\alpha\phi(\xi)} \leq c_\alpha \abs\xi^{k-\abs\alpha} \text{ for }\, \xi\neq 0 \,\text{and}\, \abs\alpha \geq 1,
\end{equation}
and $u(x,t)$ is the solution of the Cauchy problem \eqref{dispers eq} represented by the oscillatory integral above. Then for any $\tau>0$ and each $t\in[-\tau,\tau]$ and all $ p\in (n/(n+\min(1,k)) ,\infty ]$, 
$0<q\leq \infty$, $s\in \mathbb{R}$ and $ m_k (p)=-kn |1/p-1/2 |$, one has
\begin{equation}
\label{main global besov estimate for the wave equation}
\sup_{t\in [-\tau,\tau]}\Vert u\Vert_{B^{s}_{p,q}(\mathbb{R}^n)}\leq C_\tau  \Vert f\Vert_{B^{s-m_k (p)}_{p,q}(\mathbb{R}^n)}.
\end{equation}
Similarly, we have for any $s\in \mathbb{R}$,  
$ p\in (n/(n+\min(1,k)) ,\infty )$, $\min\, (2,p)\leq q\leq \max\,(2,p)$ that
 \begin{equation}\label{main global Triebel estimate for the wave equation}
\sup_{t\in [-\tau,\tau]}\Vert u\Vert_{F^{s}_{p,q}(\mathbb{R}^n)}\leq C_\tau \Vert f\Vert_{F^{s-m_k (p)}_{p,q}(\mathbb{R}^n)}.
\end{equation}
All the results are sharp when $k>1$.\\
Furthermore one also has for $1< p\leq q<\infty$ and $s\in \Rl$, the Sobolev space estimate
\begin{equation}\label{lp-lq sobolev}
\sup_{t\in [-\tau,\tau]}\Vert u\Vert_{H^{s-n(1/p-1/q),q}(\mathbb{R}^n)}\leq C_\tau \Vert f\Vert_{H^{s-m_k (q),p}(\mathbb{R}^n)}.
\end{equation}

\end{Th}

\begin{proof}
Observe that, the phase function in the integral representation \eqref{eq:morebetterer} is $x\cdot\xi+t\phi(\xi)$. Now for any $\tau>0$ and each $t\in[-\tau,\tau]$ the estimate \eqref{eq:simplified phase} yields that this phase function is SND and in $\textart F^k$ for all $k>0$ and also satisfies the LF$(\mu)$-condition \eqref{eq:LFmu} with $\mu=\min(1,k)$. Moreover, the amplitude of the oscillatory integral \eqref{eq:morebetterer} is identically equal to 1, which is trivially in 
$S^{0}_{1,0}(\mathbb{R}^n) \subset S^{0}_{0,0}(\mathbb{R}^n).$
Using \eqref{eq:morebetterer} and Theorem \ref{thm:main4glob}, it follows that the solution equation \eqref{dispers eq} verifies \eqref{main global besov estimate for the wave equation}. The proof of \eqref{main global Triebel estimate for the wave equation} is similar, and hence omitted. For the proof of the sharpness, see Section \ref{sect:Sharp}. Finally \eqref{lp-lq sobolev} follows from Theorem \ref{thm:p-q oscillatory}. 
\end{proof}

\begin{Rem}
If the function $\phi$ in \emph{Theorem \ref{Th:dispersive equations}} is assumed to be positively homogeneous of degree $1,$ then the relevant order $m_1(p)$ in the theorem above could be improved to $-(n-1)\left|1/p-1/2\right|,$ see \cite[Section 10]{IRS}.
\end{Rem}

\noindent Concerning Schr\"odinger equations, let us consider the Cauchy problem for a variable-coefficient Schr\"odinger equation
\begin{equation}\label{hyp Cauchy prob}
\left\{
  \begin{array}{ll}
     i\partial_{t} \Psi(x,t)+ \mathscr{H}( x,D)\Psi(x,t)=0, & \quad x \in \mathbb{R}^n, \, t\neq 0,\\
   \Psi(x,0)=\Psi_0 (x), & \quad x \in \mathbb{R}^n,
  \end{array} 
  \right.
\end{equation}
where $\mathscr{H}(x,D)$ is the Hamiltonian of the quantum mechanical system. For example, one can have $\mathscr{H}( x,D)= -\Delta + V(x),$ which corresponds to the Hamiltonian function $\mathscr{H}( x, \xi)= |\xi|^2  + V(x)$.  Now, if in general $\mathscr{H}$ is real-valued and $|\partial ^{\alpha}_{\xi}\partial^{\beta}_x \mathscr{H}( x,\xi)| \lesssim 1$ for $|\alpha +\beta|\geq 2$ (for example the harmonic oscillator $-\Delta +|x|^2$ yields such a Hamiltonian),  then the Cauchy problem above can be solved locally in time and modulo smoothing operators by
\begin{equation}\label{FIO representation of the solution}
  \Psi(x,t)=  \int_{\mathbb{R}^n} e^{i\varphi(x,\xi,t)} \,a (x,\xi,t) \, \widehat{\Psi_0}(\xi) \ddd \xi,
\end{equation}
where for $t\in (-\tau,\tau)$, $\tau$ sufficiently small, one has that $|\partial^{\alpha}_{\xi} \partial^\beta_x\varphi(x,\xi,t)|\lesssim 1$ for $|\alpha +\beta|\geq 2$, $\varphi$ is SND and $a (x,\xi,t)\in S^{0}_{0,0}(\Rn)$, see \cite[ Proposition 4.1]{CoNiRo}. This yields the following:
\begin{Th}\label{local besov estimate for hyperbolic pde}
\noindent Let $\Psi(x,t)$ be the solution of the Schr\"odinger Cauchy problem \eqref{hyp Cauchy prob} with initial data $\Psi_0$, where the Hamiltonian $\mathscr{H}$ is real-valued and satisfies the estimate $|\partial ^{\alpha}_{\xi}\partial^{\beta}_x \mathscr{H}( x,\xi)| \lesssim 1$ for $|\alpha +\beta|\geq 2$. Then there exists $\tau>0$ such that for all $p, q\in (0, \infty]$ and $s\in \mathbb{R}$, $m_2(p)=-2n\left |1/p-1/2\right |$, we have the local Besov-Lipschitz space estimate
  \begin{equation*}
    \sup_{t\in [-\tau,\tau]}\Vert \Psi\Vert_{B^{s, \mathrm{loc}}_{p,q}(\mathbb{R}^n)}\leq C_{\tau} \Vert \Psi_0\Vert_{B_{p,q}^{s-m_2(p)}(\mathbb{R}^n)}.
  \end{equation*}
Here the superscript \emph{"loc"} means that we first multiply the function $($distribution$)$ with a smooth cut-off function and then take the norm.\\

Similarly, for any $s\in \mathbb{R}$,  $0<p< \infty$, $\min\, (2,p)\leq q\leq \max\,( 2,p)$, one has the local Triebel-Lizorkin estimate
  \begin{equation*} 
    \sup_{t\in [-\tau,\tau]}\Vert \Psi\Vert_{F^{s, \mathrm{loc}}_{p,q}(\mathbb{R}^n)}\leq C_{\tau} \Vert \Psi_0 \Vert_{F_{p,q}^{s-m_2(p)}(\mathbb{R}^n)},
  \end{equation*}
  which also holds when $p=\infty$ and $q=2$.
Moreover, if $ m<m_2(p)$ then for all $s\in\mathbb{R}$ and $p, q\in (0, \infty]$ one has
  \begin{equation*}
    \sup_{t\in [-\tau,\tau]}\Vert \Psi\Vert_{F^{s, \mathrm{loc}}_{p,q}(\mathbb{R}^n)}\leq C_{\tau} \Vert \Psi_0 \Vert_{F_{p,q}^{s-m}(\mathbb{R}^n)}.
  \end{equation*}
  Furthermore, we also have the following global $($in space$)$ sharp estimates
  \begin{equation*} 
 \left\{
\begin{array}{lll}
    \displaystyle \sup_{t\in [-\tau,\tau]}\Vert \Psi\Vert_{F^{s}_{p,2}(\mathbb{R}^n)}\leq C_\tau\Vert \Psi_0\Vert_{F^{s-m_2(p)}_{p,2}(\mathbb{R}^n)}, &  2\leq p< \infty,\, \, s\in [m_2(p), 0], \\
    \displaystyle \sup_{t\in [-\tau,\tau]}\Vert \Psi\Vert_{F^{0}_{p,2}(\mathbb{R}^n)}\leq C_\tau\Vert \Psi_0 \Vert_{F^{-m_{2}(p)}_{p,2}(\mathbb{R}^n)}, & 0<p<\infty,
\end{array}  \right.
\end{equation*}
\end{Th}

\begin{proof}
The local results all follow from the oscillatory integral representation \eqref{FIO representation of the solution} and Theorem \ref{thm:main5}. The global estimates are all consequences of Theorem \ref{Th:bonus}
parts $ii)$ and $i)$ respectively.
\end{proof}

\section{Estimates for phases and kernels} 
\label{Sect:kernel estimates}

In this section, we prove some basic kernel estimates for oscillatory integral operators. \\

The following lemma will enable us to use a composition formula and an asymptotic expansion for the action of a pseudodifferential operator on an oscillatory integral operator. It is also helpful in the proof of Proposition \ref{Tj:Tad global} below. Once this is done, we shall then prove Theorem \ref{thm:left composition with pseudo} using only
\eqref{eq:L2 condition}, 
\eqref{asymptotic condition2}, and
\eqref{asymptotic condition3}.

\begin{Lem}\label{Lem:saviouroftoday}
Assume that $a(x,\xi)$ is an amplitude and let  $\varphi$ be a \emph{SND} phase function satisfying
\begin{equation}\label{asymptotic condition1}
\big |\partial_\xi \partial^{\beta}_x \varphi (x,\xi)\big |
\leq c_{\beta} , \quad\, |\beta|\geq 1 \text{ and } \abs\xi\geq 1.
\end{equation}

Then for all $|\beta|\geq 1$, the following estimates
\begin{align}
 |\xi| \lesssim |\nabla_{x} \phase(x,\xi)| & \lesssim  |\xi|,\label{asymptotic condition2} \\
|\partial^{\beta}_{x} \phase(x,\xi)| & \lesssim \langle \xi\rangle, \label{asymptotic condition3}
\end{align}
hold true for the phase function $\varphi$, on the support of $a(x,\xi)$, provided that either
\begin{enumerate}
\item[$i)$]  the $\xi$-support of $a(x,\xi)$ lies outside the ball $B(0,R)$ for some large enough $R\gg 1$ and $\partial^\beta_x \varphi(x,\xi) \in L^\infty(\mathbb{R}^n
\times\mathbb S^{n-1})$, for $|\beta| \geq 1$
\item[] or
\item[$ii)$] the amplitude $a(x,\xi)$ has compact $x$-support and has its $\xi$-support outside the ball $B(0,R)$ for some large enough $R\gg 1$
\item[] or 
\item[$iii)$] $\varphi(x,\xi)\in \mathcal{C}^{\infty}(\Rl^n \times \Rl^n)$, $\varphi(x,0)=0,$ and $\big |\partial_\xi \partial^{\beta}_x \varphi (x,\xi)\big |\leq c_{\beta}$, $|\beta|\geq 1,$ for $(x,\xi)\in \Rl^n \times \Rl^n$.
\end{enumerate}
\end{Lem}

\begin{proof}

We would like to compare $\partial_x^\beta\varphi(x,\xi) $ with some $\partial_x\beta\varphi(x,\xi_0) $ for $\abs\beta\geq 1$. In $i)$ and $ii)$ we choose $\xi_0 = \xi/\abs\xi$. Note that the line segment $\xi_0+t ( \xi-\xi_0)$, with $t\in (0,1)$ and $\abs\xi \geq R$, does not intersect $B(0,1)$ so we can use \eqref{asymptotic condition1} without problem.  In $iii)$ we choose $\xi_0=0$. Therefore on the support of $a(x,\xi)$, using \eqref{asymptotic condition1} and the mean-value theorem yield for $|\beta|\geq 1$ that
\eq{\abs{\partial_x^\beta \varphi(x,\xi)}&\leq \abs{\partial_x^\beta \varphi(x,\xi)-\partial_x^\beta \varphi(x,\xi_0)}+\abs{\partial_x^\beta \varphi(x,\xi_0)}\lesssim \abs{\xi-\xi_0}+\abs{\partial_x^\beta \varphi(x,\xi_0)} \\&\lesssim \abs\xi+\abs{\partial_x^\beta \varphi(x,\xi_0)}.}
Thus for both cases $i)$ and $ii)$ one has that $\abs{\partial_x^\beta \varphi(x,\xi_0)}\lesssim 1 \lesssim \abs\xi$, uniformly in $x$ on the support of $a(x,\xi)$, and the same is also true in case $iii)$ due to the vanishing of the derivatives. This proves \eqref{asymptotic condition3} and the second inequality of \eqref{asymptotic condition2}.\\

To prove the first inequality of \eqref{asymptotic condition2}, Schwartz's global inverse function theorem can be used just as in the proof of  in \cite[Proposition 1.11]{DS} to obtain
\begin{equation}\label{eq:lastest}
|\xi|-\abs{\xi_0} 
\leq \abs{\xi-\xi_0} 
\lesssim |\nabla_x \phase(x,\xi) - \nabla_x \phase(x,\xi_0)| 
\leq |\nabla_x \phase(x,\xi)|+\abs{\nabla_x \phase(x,\xi_0)}.
\end{equation}
Therefore, to prove the desired lower bound for $|\nabla_x\varphi(x,\xi)|$ in case $i)$  and $ii)$, let $\xi_0$ be defined above and insert it to  \eqref{eq:lastest}. Then for a certain constant $A=A( n, \delta, c_1)>0$ (where $n$ is the dimension, $\delta$ is the lower bound in the SND-condition and $c_1$ is the upper bound on the norm of the mixed Hessian of $\varphi(x,\xi)$  when $|\xi|\geq 1$) \eqref{eq:lastest} yields that
\eq{
\abs\xi \leq A \Big( |\nabla_x \phase(x,\xi)|+\abs{\nabla_x \phase(x,\xi_0)}  \Big) +1.
}
However since $|\xi|>R$, on the support of $a(x,\xi)$,  and $R$ can be chosen large enough, by taking 
$$R\geq 2A \Big(\max_{x\in\supp a(x,\xi_0)}\abs{(\nabla_x \phase(x,\xi_0)} \Big)+2,$$ 
we obtain
\eq{
\abs\xi \lesssim |\nabla_x \phase(x,\xi)|.}
In case $iii)$ the same inequality is once gain valid since we take $\xi_0=0$ and \linebreak$\nabla_x \phase(x,0)=0$ in \eqref{eq:lastest}.
\end{proof}

Next we turn to kernel estimates of the operators in various settings. A  simple case is when the amplitude is spatially localised.
\begin{Lem}\label{Lem:pointwiseKjsimplified}
Let $m\in \Rl$, $\varphi(x,\xi)$ be a real-valued function and $a(x,\xi) \in S^{m}_{0,0}(\Rl^n)$ has compact support in the spatial variable $x$. Define
\eq{
K_j(x,y) := \int_{\Rl^n} a_j(x,\xi)\,e^{i\varphi(x,\xi)} 
\,e^{-iy\cdot\xi}\ddd\xi,
}
where $a_j(x,\xi):= \psi_j(\xi)\,a(x,\xi)$ is a Littlewood-Paley piece of the amplitude $a$.
Then for each $j \in \Z_+$ and all multi-indices $\beta$ we have
\begin{equation*}\label{eq:KjL2}
\norm{ \partial_y^\beta K_j(x,y)}_{L_{x,y}^\infty(\Rl^n \times \Rl^n)}
\lesssim 2^{j(m+|\beta|+n)}.
\end{equation*}
\end{Lem}

\begin{proof}
Observe that
\begin{align*}
| \partial_y^\beta K_j(x,y)|
& = \Big|\partial_y^\beta \int_{\Rl^n} a_j(x,\xi)\,e^{i\varphi(x,\xi)} 
\,e^{-iy\cdot\xi}\ddd\xi \Big|
= \Big| \int_{\Rl^n} a_j(x,\xi)\,e^{i\varphi(x,\xi)} \,\xi^\beta \,e^{-iy\cdot\xi}\ddd\xi \Big| \\
& \lesssim \int_{\Rl^n}\vert a_j(x,\xi) \vert \,\vert \xi \vert ^{\vert \beta \vert}\ddd\xi 
\lesssim \Vert a_j \Vert_{L^\infty(\Rl^n \times \mathbb{R}^n)}\,2^{j(\vert\beta\vert+n)}\lesssim 2^{j(\vert\beta\vert+m+n)},
\end{align*}
for any $x,\,y\in \Rl^n$.
\end{proof}

Next we prove a kernel estimate for the low frequency portion of oscillatory integral operators.

\begin{Lem}\label{low freq lemma llf 1}
Let $\mu>0$, $a_L(x,\xi)
$ be a symbol that is compactly supported and smooth outside the origin in the $\xi$-variable and $\varphi(x,\xi)\in \mathcal{C}^{\infty}(\Rl^n \times \Rl^n \setminus \{0\}) $ be a phase function. Assume that one of the following conditions hold:
\begin{align}
&\begin{cases}\Vert \partial_{\xi}^{\alpha} a_L(\cdot,\xi)\Vert_{L^\infty(\Rl^n)}
\leq c_\alpha |\xi|^{\mu-|\alpha|}, 
& |\alpha|\geq 0,\\
\Vert \partial^{\alpha}_{\xi} (\varphi(x,\xi)-x\cdot \xi) \Vert_{L^\infty_x(\Rl^n)}\leq c_{\alpha} |\xi|^{-|\alpha|}, & \abs\alpha\geq 1,
\end{cases}\label{eq:lf1}\\
&\begin{cases}\Vert \partial_{\xi}^{\alpha} a_L(\cdot,\xi)\Vert_{L^\infty(\Rl^n)}
\leq c_\alpha , 
&|\alpha|\geq 0,\\
\Vert \partial^{\alpha}_{\xi} (\varphi(x,\xi)-x\cdot \xi) \Vert_{L^\infty_x(\Rl^n)}\leq c_{\alpha} |\xi|^{\mu-|\alpha|}, & \abs\alpha\geq 0,
\end{cases}\label{eq:lf2}
\end{align}
{for $\xi\neq0$} and on the support of $a_L(x,\xi).$ Then the modulus of the integral kernel 
$$K(x,y)
:= \int_{\Rn} a_L(x,\xi)\, e^{i\varphi(x,\xi)-iy\cdot \xi}\ddd\xi,$$
and that of $\overline{K(y,x)}$ are both bounded by  $\langle x-y\rangle ^{-n-\eps\mu}$ for any $0\leq \eps <1$. 
\end{Lem}

\begin{proof}
Since $|K(x,y)|\lesssim 1,$ it is enough to show that that $|K(x,y)|\lesssim|x-y|^{-n-\eps\mu}$. 

In order to prove the Lemma under assumptions \eqref{eq:lf1} we set
$\sigma(x,\xi):= a_L(x,\xi)$ $e^{i\varphi(x,\xi)-ix\cdot\xi}$
\eq{
K(x,y) := \int_{\Rn} e^{i(x-y) \cdot \xi}\,  \sigma(x,\xi) \ddd\xi.
}

Observe that $ \abs{\partial_\xi^\alpha \sigma(x,\xi)} \lesssim  
\abs\xi^{\mu-\abs\alpha}
$ for any $\abs\alpha\geq 0$ and $\xi\in\supp_\xi a_L(x,\xi)$.
Now one introduces a Littlewood-Paley partition of unity 
 $$\sum_{j=-\infty}^{\infty} \psi(2^{-j}\xi)=1, \text{ for } \xi\neq 0, \text{ with } \supp \psi(\xi)\subset \{1/2\leq |\xi|\leq 2\},$$ and defines
\eq{
K_j(x,y) := \int_{\Rn} e^{i(x-y) \cdot \xi}\,  \sigma(x,\xi)\, \psi(2^{-j}\xi) \ddd\xi.
}

Integration by parts $N$ times yields
\begin{equation}\label{bloodykernelestim}
\begin{split}
\abs{K_j(x,y)} &\lesssim \abs{x-y}^{-N} \sum_{\abs\alpha+\abs\beta = N} \int_{\Rn} \abs{\partial_\xi^\alpha \sigma(x,\xi)}|\partial_\xi^\beta \psi(2^{-j}\xi)| \ddd\xi \\&\lesssim  \abs{x-y}^{-N} 2^{j(\mu+n-N)},
\end{split}
\end{equation}

However, if $H$ is any positive real number, then one can write $H$ as the sum $N+\theta$ where $N$ is a positive integer and $\theta \in [0,1)$. Now since \eqref{bloodykernelestim} implies
that \begin{equation}\label{skit 1}
|K_j(x,y)|\lesssim 2^{j(n-N+\mu)} |x-y|^{-N},
\end{equation} and 
\begin{equation}\label{skit 2}
|K_j(x,y)|\lesssim 2^{j(n-(N+1)+\mu)} |x-y|^{-(N+1)},
\end{equation}
raising \eqref{skit 1} to the power $1-\theta$ and \eqref{skit 2} to the power $\theta$, and using the fact that $H =N+\theta$, yield that
$$|K_j(x,y)|=|K_j(x,y)|^{1-\theta} |K_j(x,y)|^\theta
\lesssim 2^{j(n-H+\mu)} |x-y|^{-H},$$
for all $H\geq0$. 

Observe that there exists $M>0$ such that $\supp_\xi a_L(x,\xi)\subseteq B(0,2^M)$. Therefore, we can write
\eq{
K(x,y) = \sum_{j=-\infty}^M K_j(x,y),
}
and hence setting $H:= n+\eps\mu$, we obtain
\eq{
\abs{K(x,y)} \lesssim \sum_{j=-\infty}^M \abs{x-y}^{-n-\eps\mu} 2^{j\mu(1-\eps)} \lesssim \abs{x-y}^{-n-\eps\mu}.
}\hspace*{1cm}

To prove the lemma under assumptions \eqref{eq:lf2}, split the kernel into 
$K(x,y)= K'(x, y)$ $+K''(x, y)$ where
$$K'(x,y):= \int_{\Rn} e^{i(x-y) \cdot \xi}\,  a_L(x,\xi) \ddd\xi, $$
and
$$K''(x,y):= \int_{\Rn} e^{i(x-y) \cdot \xi}\,  a_L (x,\xi)\,(e^{i\varphi(x,\xi)-ix\cdot\xi}-1) \ddd\xi.
$$
Integration by parts $N$ times yields that $\abs{K'(x,y)}\lesssim \abs{x-y}^{-N}$ for all $N$. 
To obtain the estimate for $K''(x,y)$ we set $\sigma(x,\xi) :=  a_L(x,\xi)\,(e^{i\varphi(x,\xi)-ix\cdot\xi}-1) $ and note that for all $\abs\alpha\geq 0$,
$|\partial_\xi^{\alpha}( e^{i\varphi(x,\xi)-ix\cdot\xi}-1)|\lesssim |\xi|^{\mu-|\alpha|}$,
so that 
$\abs{\partial_\xi^\alpha \sigma(x,\xi)} \lesssim  
\abs\xi^{\mu-\abs\alpha}$
for any $\abs\alpha\geq 0$.  Now the rest of the proof proceeds as in the previous case above.\\

The proof for $\overline{K(y,x)}$ is identical and hence omitted.
\end{proof}
\begin{Rem}\label{Rem:CRS}
Observe that for phase functions of the form $x\cdot\xi+|\xi|^k$ with $k>1$ and symbols $a(x,\xi)=\chi(\xi)\in\mathcal C^\infty_c(\Rn)$, a decay of the form $\langle x-y\rangle^{-n-1}$, was already proven in e.g. \cite[Lemma 2.3]{CRS}.
\end{Rem}

The next lemma yields a sufficient condition for the $h^p-L^p$ boundedness of linear operators and will be quite useful in what follows.
\begin{Lem}\label{lem:hpLpofT}
Assume that $0<p<1$ and $T_a^{\varphi}$ is an $L^2$--bounded  oscillatory-- or  Schr\"odinger integral operator. Let $T_j$ be either $T^\varphi_a \psi_{j} (D)$ or $\psi_{j} (D)(T^\varphi_a)^*$ with $\psi_j$ as in \emph{Definition \ref{def:LP}} \emph{(i.e. the familiar $j$-th Littlewood-Paley piece of $T^\varphi_a$ and its formal adjoint $(T^\varphi_a)^*$ )}. We also assume that for a $p$-atom $\at$  supported in a ball of radius $r$ one has
\begin{equation}\label{eq:keyestimate1}
\| T_j \at \|_{L^p(\Rl^n)} \lesssim r^{n-n/p}\,2^{j\brkt{n-n/p}}.
\end{equation}
Moreover, assume that whenever $r <1$, 
\begin{equation}\label{eq:keyestimate2}
\| T_j \at \|_{L^p(\Rl^n)} \lesssim r^{{N} +1+n-n/p }\,2^{j\brkt{{N}+1+n-n/p}},
\end{equation} 
for some $N > n/p-n-1$. 
Then $T_a^{\varphi}$ \emph{(or $(T^\varphi_a)^*$ when the $T_j$'s are associated to the adjoint)} is bounded from $h^p(\Rl^n)$ to $L^p(\Rl^n)$.
\end{Lem}
\begin{proof}
 Using the atomic characterisation of $h^p(\Rn)$, and following the strategy in  \cite[p. 402]{Stein} for $T_a^{\varphi}$ and the strategy in \cite[p. 237]{SSS} for $(T_a^{\varphi})^*$, it is enough to show that for any $p$-atom $\at$, one has the uniform estimates
$$\|T_a^{\varphi} \at\|_{L^p(\Rn)}
\lesssim 1,$$
or 
$$\|(T_a^{\varphi})^* \at\|_{L^p(\Rn)}
\lesssim 1,$$ in each case.
We only prove the result in the case of $T^\varphi_a$, since the case of the adjoint is similar.
We split the proof in two different cases, namely $r<1$ and $r \geq 1$. 
For $r\geq 1$, \eqref{eq:keyestimate1} yields that
\begin{equation*}\label{eq:largeballs}
\| T_a^{\varphi} \at \|_{L^p(\Rl^n)}^p 
\lesssim \sum_{j=0}^\infty \| T_j \at \|_{L^p(\Rl^n)}^p \lesssim \sum_{j=0}^\infty r^{np-n}\,2^{j\brkt{np-n}} \lesssim 1.
\end{equation*}
Assume now that $r < 1$. 
Choose $\ell \in \Z_+$ such that 
$2^{-\ell-1} \leq r < 2^{-\ell }$. Using the facts that $2^{-\ell}\sim r$, $N+1+n-n/p>0$, $n-n/p<0$, \eqref{eq:keyestimate1} and \eqref{eq:keyestimate2} we conclude that
\begin{align*}
\| T_a^{\varphi} \at\|_{L^p(\Rn)}^p
& \lesssim \sum_{j=0}^{\ell} \brkt{r^{N+1+n-n/p} \, 2^{j\brkt{N+1+n-n/p}}}^p 
	+ \sum_{j=\ell +1}^{\infty} \brkt{r^{n-n/p} \, 2^{j\brkt{n-n/p}}}^p \\
& \lesssim  \brkt{r^{N+1+n-n/p} \, 2^{\ell \brkt{N+1+n-n/p}}}^p 
	+ \brkt{r^{n-n/p} \, 2^{\ell\brkt{n-n/p}}}^p \\
& \sim  \brkt{r^{N+1+n-n/p} \, r^{-\brkt{N+1+n-n/p}}}^p 
	+ \brkt{r^{n-n/p} \, r^{-\brkt{n-n/p}}}^p = 1. \qedhere
\end{align*} 
\end{proof}

As an application of the previous lemma, we have the following $h^p-L^p$ boundedness result, based entirely on kernel estimates of the corresponding operators.

\begin{Lem}\label{lem:Miyachistuff}
{Let $0<p\leq 1$, $k>0$,  $a(x,\xi)\in S^{m_k(p)}_{0,0}(\Rl^n)$, $\varphi(x,\xi)\in\textart{F}^k$, and let the operator $T_j$ given in \emph{Lemma \ref{lem:hpLpofT}}} have either the representation
\begin{equation}\label{representationone}
\int_{\Rl^n} K_{1,j}(x,x-y)\, f(y) \dd y , 
\end{equation} 
or the representation
\begin{equation}\label{representationtwo}
\int_{\Rl^n} K_{2,j}(y,x-y)\, f(y) \dd y. 
\end{equation}
\begin{itemize}
    \item If 
    \begin{equation}\label{eq:new1}
        \norm{ (x-y)^\alpha \partial^\beta_y \left(K_{1,j} (x,x-y)\right)}_{L^2_x(\Rl^n)}
\lesssim 2^{j(\vert\alpha\vert(k-1) +\vert \beta\vert+m_k(p)+n/2)}, 
    \end{equation}
    uniformly in $y \in \mathbb{R}^n$, and $T^\varphi_a:=\sum_{j=0}^\infty T_j$ \emph{(in the case of \eqref{representationone})} is $L^2$-bounded, 
then $T^\varphi_a$ is bounded from $h^p(\Rl^n)$ to $L^p(\Rl^n)$.\\
    \item If 
    \begin{equation}\label{eq:new2}
       \norm{ (x-y)^\alpha \partial^\beta_y \left(K_{2,j} (y,x-y)\right)}_{L^2_x(\Rl^n)}
\lesssim 2^{j(\vert\alpha\vert(k-1) +\vert \beta\vert+m_k(p)+n/2)}
    \end{equation}
    uniformly in $y \in \mathbb{R}^n$, and $(T^\varphi_a)^*:=\sum_{j=0}^\infty T_j$ \emph{(in the case of \eqref{representationtwo})} is $L^2$-bounded, 
then $(T^\varphi_a)^*$ is bounded from $h^p(\Rl^n)$ to $L^p(\Rl^n)$.
\end{itemize}

\end{Lem}

\begin{proof}

Once again, we only treat the case of $T_a^{\varphi}$, since that of the adjoint is done in a similar manner. Let $\at$ be a $p$-atom supported in the ball $B:=B(\bar{y},r)$ and let $2B:=B(\bar{y},2r)$. By H\"older's inequality and the $L^2$-boundedness of $T_a^{\varphi}$, we have
\begin{align*}\label{eq:combining1}
\Vert T_a^{\varphi} \at \Vert_{L^{p}(2B)} 
&\lesssim \Vert T_a^{\varphi} \at \Vert_{L^{2}(2B)} \norm{1}_{L^{2p/(2-p)}(2B)} \lesssim \norm{\at}_{L^{2}(\Rl^n)} r^{n(2-p)/2p}  \\
&\lesssim r^{n(p-2)/2p}\, r^{n(2-p)/2p} = 1.
\end{align*}

We proceed to the boundedness of
$\Vert T_a^{\varphi} \at \Vert_{L^{p}(\Rl^n\setminus 2B)}$, 
which is more subtle. By Lemma \ref{lem:hpLpofT}, it is enough to show estimates \eqref{eq:keyestimate1} and \eqref{eq:keyestimate2} for $\norm{T_j\at }_{L^p(\Rl^n\setminus 2B)}$. For all multi-indices $\alpha$, \eqref{eq:new1} yields
\eq{
\norm{ (2^{-j(k-1)}(x-y))^\alpha K_{1,j}(x,x-y)}_{L^2_x(\Rl^n)}
&\lesssim 2^{j(n/2+m_k(p))}, 
}
so that for any integer $M$, if one sums over $\abs\alpha\leq M$,
\nm{eq:newone}{
\norm{ (1+2^{-j(k-1)}\vert x-y\vert)^M K_{1,j}(x,x-y)}_{L^2_x(\Rl^n)}&\lesssim 2^{j(n/2+m_k(p))}.
}
We now observe that for $t\in [0,1],$ $x\in \Rl^n\setminus 2B$ and $y\in B$, one has 
\begin{equation}\label{eq:estimateybar}
    \vert x-\bar{y}\vert \lesssim \vert x-\bar{y}-t(y-\bar{y})\vert.
\end{equation}

Next we introduce
$$g(x):=\brkt{1+2^{-j(k-1)}|x-\bar{y}|}^{-M},$$
where $M > n/q$ and $1/q=1/p-1/2$. The H\"older and the Minkowski inequalities together with \eqref{eq:newone} and \eqref{eq:estimateybar} (with $t=1$) yield
\begin{align}
 & \norm{T_j\at }_{L^p(\Rl^n\setminus 2B)}
= \Big\| \int_{B} K_{1,j}(x,x-y)\at(y)\dd y \Big\| _{L^p_x(\Rl^n\setminus 2B)} \nonumber \\ 
&\qquad \leq \Big\|  \frac{1}{g(x)} 
\int_{B}K_{1,j}(x,x-y)\at(y)\dd y \Big\| _{L^2_x(\Rl^n\setminus 2B)}
\Vert g \Vert _{L^q(\Rl^n)}  \nonumber \\ 
&\qquad\lesssim 2^{jn(k-1)/q}
  \int_{B} \Big\|\frac{1}{g(x)}K_{1,j}(x,x-y)\at(y)\Big\|_{L^2_x(\Rl^n\setminus 2B)}  \dd y \nonumber \\
 &\qquad\lesssim 2^{jn(k-1)(1/p-1/2)} \int_{B} \vert \at(y) \vert \Big\| \brkt{1+2^{-j(k-1)}|x-{y}|}^{M}    K_{1,j}(x,x-y)\Big\| _{L^2_x(\Rl^n\setminus 2B)} \!\dd y \nonumber \\
&\qquad\lesssim r^{n-n/p} 2^{jn(k-1)(1/p-1/2)} 2^{j(n/2+m_k(p))}. \nonumber 
\end{align}
Recalling that $m_k(p)=-kn(1/p-1/2)$ 
we get \eqref{eq:keyestimate1}.\\
We proceed to show estimate \eqref{eq:keyestimate2}. Taking $N:=[n(1/p-1)]$ (note that $N > n/p-n-1$), 
a Taylor expansion of the kernel at the point $y=\overline{y}$ yields that
\eq{
K_{1,j}(x,x-y)
& = \sum_{ |\beta| \leq N} \frac{(y-\bar y)^\beta}{\beta!}  \partial^\beta_y (K_{1,j}(x,x- y))_{|_{y=\bar y}} \\
&  + (N+1) \sum_{ |\beta| = N+1}  \frac{(y-\bar y)^\beta}{\beta!}  \int_0^1 (1-t)^N \partial^\beta_y (K_{1,j}(x,x-y))_{|_{y=\bar y+t(y-\bar y)}} \dd t,
}
and due to vanishing moments of the atom in Definition \ref{def:hardyspace}, $iii)$, we may express the operator as
\begin{align*}
T_j\at(x)   =(N+1)\!\!\! \sum_{ |\beta| = N+1} \int_{B}  \int_0 ^1 \frac{(y-\bar{y})^\beta}{\beta!} (1-t)^N
\partial^\beta_y (K_{1,j}(x,x-y))_{|_{y=\bar y+t(y-\bar y)}}
\at(y) \dd t\dd y.
\end{align*}

 Now noting that $\abs{ (y-\bar{y})^\beta }\lesssim r^{N+1}$ and applying the same procedure as above together with estimates \eqref{eq:estimateybar} and \eqref{eq:new1}, we obtain
$$\Vert T_j \at \Vert _{L^p(\Rn)} 
\lesssim r^{N+1-n/p+n}2^{j(N+1+m_k(p)+n/2+n(k-1)(1/p-1/2))},
$$
which yields \eqref{eq:keyestimate2}.\\
\end{proof}

\section{\texorpdfstring{$L^2$}\ -boundedness}
\label{sec:L2-boundedness}

In the forthcoming sections, we will also need the following important theorem about $L^2$-boundedness of an oscillatory integral operator.

\begin{Th}\label{thm:fujiwara}
Let $a(x,\xi) \in S_{0,0}^0(\Rl^n)$ and assume that $\varphi (x,\xi)$ fulfills the $L^2$-condition \eqref{eq:L2 condition}
and is \emph{SND}. Then the  oscillatory integral operator $T_a^\varphi$ given by \eqref{eq:OIO} is bounded from $L^2(\Rl^n)$ to itself, under either of the following circumstances:
\begin{enumerate}
\item[$i)$] The amplitude $a(x,\xi)$ is compactly supported in $x$.

 \item[$ii)$] {One of the assumptions} \eqref{eq:lf1} or \eqref{eq:lf2} holds true.
 
\end{enumerate}
\end{Th}

\begin{proof}
We divide the proof into low and high frequency cases, by writing $a(x,\xi)= \psi_{0}(\xi)\,a(x,\xi)+ (1-\psi_0(\xi)) \,a(x,\xi)=: a_{\mathrm{L}}(x,\xi)+ a_{\mathrm{H}}(x,\xi)$, with $\psi_0$ is in Definition \ref{def:LP}. For $T_{a_H}^{\varphi}$, the phase function is smooth and doesn't have any singularity. This enables one to use an $L^2$-boundedness result for oscillatory integrals proven by D. Fujiwara in \cite{Fuj}, since the assumptions of Theorem \ref{thm:fujiwara} fulfill conditions (A-I)--(A-IV) in  \cite{Fuj}, on the support of $a_H$.
\\

For $T_{a_L}^{\varphi},$ part of the case
$i)$, using the compact support in $\xi$, Cauchy-Schwarz's inequality, and Plancherel's theorem allow us to write
\begin{equation*}
|T_{a_L}^\varphi f(x)|=
\Big| \int_{\Rl^n} a_L(x,\xi)\, e^{i\varphi(x,\xi)}\, \widehat{f}(\xi) \ddd\xi \Big|
\lesssim \Vert f\Vert_{L^2(\Rl^n)}.
\end{equation*}
Now the fact that $T_{a_L}^\varphi f(x)$ is compactly supported yields that 
$$\Vert T_{a_L}^\varphi f(x)\Vert_{L^2(\Rl^n)}\lesssim \Vert f\Vert_{L^2(\Rl^n)}.$$
For the $T_{a_L}^{\varphi},$ part of the case $ii)$, we use Lemma \ref{low freq lemma llf 1} to conclude that the kernel satisfies
\eq{
|K(x,y)|\lesssim \langle x-y\rangle ^{-n-\varepsilon\mu},
}
for any $\varepsilon\in[0,1)$. Therefore, {Schur's lemma} applies in this case. 
\end{proof}

\begin{Rem}
In dimension one, for $k>0$, if we take the phase function $$\varphi(x,\xi)
:=x\xi-\frac{1}{2}\sin x \cos \xi+ |\xi|^k,$$ 
then one can verify that for $k\geq 1$ 
{the low frequency assumption of \eqref{eq:lf2} holds with $\mu=1$, and for $0<k<1$ with $\mu=k.$}
Moreover, 
$$\Big|\partial^{\alpha}_{\xi} \partial^{\beta}_{x}\Big(x\xi-\frac{1}{2}\sin x \cos \xi+ |\xi|^k \Big) \Big|\lesssim 1,
\quad |\alpha|, |\beta|\geq 1$$
and the \emph{SND}-condition is also satisfied thanks to $$\Big|\partial_{\xi}\partial_{x} \Big(x\xi-\frac{1}{2}\sin x \cos \xi+ |\xi|^k \Big)\Big|\geq 1/2.$$
This together with an amplitude in $S^0_{0,0}(\Rn)$ gives rise to an $L^2$-bounded operator. However, this is not entirely covered by the $L^2$-boundedness results of H\"ormander \cite{Hor2} $($because of lack of homogeneity and also lack of compact support in the $x$-variable$)$ or Fujiwara \cite{Fuj} $($due to lack of smoothness$)$. It is also important to note that the rather strong assumptions on the phase function are needed to deal with the lack of decay in the amplitude $($i.e. an amplitude in $S^0_{0,0}(\Rl^n))$. 
\end{Rem}

\section{Boundedness of low frequency portion}
\label{sect:lowfreq}

The kernel estimate obtained in Lemma \ref{low freq lemma llf 1} can be used to show that the corresponding oscillatory integral operators are bounded in various Banach, as well as quasi-Banach spaces.  Now, as far as the $L^p$-regularity is concerned, the Mikhlin multiplier theorem yields the following boundedness result for operators with amplitudes that are compactly supported in the spatial variables.

\begin{Lem}\label{low freq lemma llf 2}
Let $a_L(x,\xi)\in \mathcal{C}^{\infty}_{c} (\Rl^n \times \Rl^n )$ be an amplitude and assume that  $\varphi(x,\xi)\in \mathcal{C}^{\infty}(\Rl^n \times \Rl^n \setminus \{0\}) $ with 
\begin{equation*}\label{mikhlin condition}
|\partial^{\alpha}_{\xi}\partial^{\beta}_{x} \varphi(x,\xi)|\leq c_{\alpha, \beta} |\xi|^{-|\alpha|},
\qquad|\alpha+\beta| \geq 1, \,\quad\, (x,\xi)\in\supp a_L
\end{equation*}
Then the operator $T_{a_L}^\varphi$ of the form \eqref{eq:OIO} is bounded on $L^p(\Rl^n)$ for $1<p<\infty$.

\end{Lem}
\begin{proof}
Set $\sigma(x,\xi):=a_L(x,\xi)\, e^{i(\varphi(x,\xi)-x\cdot \xi)}$ and observe that the condition on the phase function implies that $|\partial^\alpha_{\xi} \partial^\beta_{x} \sigma(x,\xi)|\lesssim |\xi|^{-|\alpha|}.$ Now, we write 
$$T_{a_L}^\varphi f(x)
=\int_{\Rl^n} \sigma(x,\xi)\, e^{ix\cdot\xi}\, \widehat{f}(\xi) \ddd\xi,$$
and
using the fact that $\sigma$ is compactly supported in $x$ we have that for any integer $N>0$
\begin{equation}\label{Fourier rep of T}
T_{a_L}^\varphi f(x)= \int_{\Rl^n} \langle \eta\rangle^{-2N}\,\brkt{\int_{\Rl^n} \langle \eta\rangle^{2N}\,\widehat{\sigma}(\eta,\xi)\, e^{ix\cdot\xi}\,\widehat{f}(\xi) \ddd \xi}\, e^{ix\cdot\eta}  \ddd \eta.
\end{equation}
The compact support of $\sigma(x,\xi)$ also implies that
$$\langle \eta\rangle^{2N}\,\abs{\partial^\alpha_{\xi}\widehat{\sigma}(\eta,\xi)}=\abs{\int_{\Rl^n}e^{-ix\cdot\eta}\, (1-\Delta_{x})^{N} \partial^\alpha_{\xi} \sigma(x,\xi)\dd x}\lesssim |\xi|^{-|\alpha|},$$
uniformly in $\eta$. Now the boundedness of $\langle\eta\rangle^{2N}\,\widehat{\sigma}(\eta,\xi)$ and the estimate above show that the aforementioned function is a Mikhlin multiplier and therefore bounded on $L^p(\Rl^n)$ for $1<p<\infty.$ Therefore, using Minkowski's integral inequality to \eqref{Fourier rep of T}, which is valid for the Banach space scales of $L^p$-spaces, and choosing $N$ large enough yield the desired boundedness.
\end{proof}

The following lemma establishes the local boundedness of the low frequency portion of adjoint operator $\brkt{T_a^\varphi}^*$.
\begin{Lem}\label{Lem:smoothlowfreq}
Let $0<p<1$. Moreover, assume $\varphi(x,\xi) \in \mathcal{C}^\infty(\Rn \times \Rn)$ and
$a_L (x,\xi) \in \mathcal{C}^\infty_c(\Rn \times \Rn)$. Then $\brkt{T_{a_L}^\varphi}^*$ given as in \eqref{eq:adjointOIO} is a bounded operator from $h^p(\Rn)$ to $L^p(\Rn)$.
\end{Lem}

\begin{proof}
Set  $\sigma (y,\xi) := \overline{a_L(y,\xi)}\,e^{-i\varphi(y,\xi) + i y \cdot \xi}$ and
consider the kernel of $\brkt{T_{a_L}^\varphi}^*$ 
\begin{align*}
K^\ast(y,x-y)
&=\int_{\Rl^n} \overline{a_L(y,\xi)}\,e^{-i(\varphi(y,\xi)-y\cdot \xi-(x-y)\cdot\xi)}\ddd\xi = \int_{\Rl^n} \sigma (y,\xi) \,e^{i(x-y)\cdot\xi} \ddd\xi.
\end{align*}
Leibniz's rule and integration by parts yield
\begin{align*}
(x-y)^\alpha \partial^{\beta}_{y} (K^\ast(y,x-y)) 
& = (x-y)^\alpha \int_{\Rl^n}  \partial^{\beta}_{y} (\sigma (y,\xi) \,e^{i(x-y)\cdot\xi}) \ddd\xi \\
& =\sum_{\substack{\alpha_1 +\alpha_2 =\alpha \\ \beta_1 +\beta_2 =\beta \\  \beta_2 \geq \alpha_2}} 
C_{\alpha,\beta} \int_{\Rl^n} \partial^{\alpha_1}_{\xi}\partial^{\beta_1}_y\sigma(y,\xi)\, \xi^{\beta_2 -\alpha_2} \, e^{i(x-y)\cdot.\xi} \, \ddd\xi \\
& = (\rho(y,\cdot))^\wedge (y-x),
\end{align*}
where
$$\rho(y,\xi)
:= \sum_{\substack{\alpha_1 +\alpha_2 =\alpha \\ \beta_1 +\beta_2 =\beta \\  \beta_2 \geq \alpha_2}} 
C_{\alpha,\beta}\, \partial^{\alpha_1}_{\xi}\partial^{\beta_1}_y\sigma(y,\xi)\,  \xi^{\beta_2 -\alpha_2} .$$
Therefore, Plancherel's formula yields that
\eq{
\norm{ (x-y)^\alpha \partial^{\beta}_{y} (K^\ast(y,x-y))}_{L^2_x(\Rl^n)}
&=\norm{ \rho(y,\cdot) }_{L^{2}_{\xi}(\Rl^n)}
 \lesssim \norm{\rho (y,\cdot)} _{L^\infty_\xi(\Rl^n)}
 \lesssim 1.
}

{Hence Lemma \ref{lem:Miyachistuff} can be applied} with 
$T_0:={(T_{a_L}^\varphi)}^*$ and $T_j:=0$, $j \geq 1$, since by Theorem \ref{thm:fujiwara}, $T_0$ is also bounded on $L^2(\Rl^n)$ 
and has an integral representation 
$$\int_{\Rl^n} K ^*(y, x-y) f(y) \dd y$$
with 
\begin{align*}
K^*(y,x-y)
& = \int_{\Rl^n} \overline{a_L (y,\xi)}\, e^{-i\varphi(y,\xi)+ (x-y)\cdot \xi} \, \ddd \xi.  \qedhere
\end{align*}
\end{proof}

Next we prove the main result concerning the regularity of the low frequency portions of oscillatory integral operators.
\begin{Lem}\label{Lem:lowfreq}
Assume that $\psi_0(\xi)\in \mathcal{C}_{c}^{\infty}(\Rl^n)$ is a smooth cut-off function supported in a neighborhood of the origin as in \emph{Definition \ref{def:LP}}, $a(x,\xi)\in S^m_{0,0}(\Rl^n)$ for some $m\in \Rl$,   $a_L (x,\xi) := \psi_0(\xi)\,a(x,\xi)$ and let $\varphi(x,\xi)$ be a phase function. Finally let the operator $T_{a_L}^\varphi$ be defined as in \eqref{eq:OIO}. Then the following statements hold:
\begin{enumerate}
\item[$i)$] If either \eqref{eq:lf1} or \eqref{eq:lf2} holds, then
\begin{equation*}
\Vert T_{a_L}^\varphi f\Vert_{L^\infty (\Rl^n)}\lesssim \Vert f\Vert_{\bmo(\Rl^n)}.
\end{equation*}
 \item[$ii)$] Assume that  $\varphi(x,\xi)$ satisfy the \emph{LF}$(\mu)$-condition \eqref{eq:LFmu} for $0<\mu\leq1$ and that $n/(n+\mu)< p\leq \infty$. Then for any $s_1, s_2\in (-\infty, \infty)$, and $q_1, q_2 \in (0,\infty]$ one has 
\begin{equation*}
\Vert T_{a_L}^\varphi f\Vert_{B_{p,q_2}^{s_2}(\Rl^n)}\lesssim \Vert f\Vert_{B_{p,q_1}^{s_1}(\Rl^n)}.
\end{equation*}
\item[$iii)$] Assume that $\varphi(x,\xi)$ satisfy the \emph{LF}$(\mu)$-condition \eqref{eq:LFmu}  for $0<\mu\leq1$ and that $a(x,\xi)$ has compact support in the $x$-variable. Then for any $s_1, s_2\in (-\infty, \infty)$, and $p, q_1, q_2 \in (0,\infty]$
\begin{equation*}
\Vert T_{a_L}^\varphi f\Vert_{B_{p,q_2}^{s_2}(\Rl^n)}\lesssim \Vert f\Vert_{B_{p,q_1}^{s_1}(\Rl^n)}.
\end{equation*}
\end{enumerate}
Moreover, all the Besov-Lipschitz estimates above can be replaced by the corresponding Triebel-Lizorkin estimates. 
\end{Lem}

\begin{proof}[Proof of \emph{Lemma \ref{Lem:lowfreq}}, $i)$]
 We are going to show that 
${(T_{a_L}^\varphi)}^*:L^1(\Rl^n)\to h^1(\Rl^n)$. By 
Lemma \ref{low freq lemma llf 1} and the definition of the $h^1$-space (regarded as the Triebel-Lizorkin space $F^{0}_{1,2}$) we obtain
\eq{
\norm{{(T_{a_L}^\varphi)}^*f}_{h^1(\Rn)} 
&= \Big\| \Big( \sum_{j=0}^\infty \abs{\psi_j(D){(T_{a_L}^\varphi)}^*f}^2 \Big)^{1/2}\Big\|_{L^1{(\Rn)}} 
{\sim} \norm{{(T_{a_L}^\varphi)}^*f }_{L^1{(\Rn)}} \\
&\lesssim \norm{f }_{L^1{(\Rn)}}, \qedhere }
where we have also used the fact that $\psi_j(D){(T_{a_L}^\varphi)}^*= {(T_{a_L}^\varphi \psi_j(D))}^*=0$ when $j\geq 1$ and used the kernel estimate in Lemma \ref{low freq lemma llf 1} to deal with the last $L^1$--estimate.
\end{proof}

\begin{proof}[Proof of \emph{Lemma \ref{Lem:lowfreq}}, $ii)-iii)$]
Assume that $f_0=\Psi_0(D) f$ where $\Psi$ is a smooth cut-off function that is equal to one on the support of $\psi_0$ so that $T_{a_L}^\varphi f = T_{a_L}^\varphi f_0$. Define the self-adjoint operators
\begin{equation*}
L_\xi := 1-\Delta_\xi, \quad L_y := 1-\Delta_y,
\end{equation*}
and note that
\begin{equation*}
\jap{\xi}^{-2} L_y \,e^{i( x-y)\cdot\xi } =  \jap{x-y}^{-2}
L_\xi\, e^{i( x-y)\cdot\xi } =e^{i( x-y)\cdot\xi }.
\end{equation*}
Take integers $ N_1$ and $ N_2$ large enough. Integrating by parts, we have
\eq{
&\psi_j(D)T_{a_L}^{\varphi}f(x) = \iint_{\Rl^n\times \Rl^n } e^{i( x-y)\cdot\xi }\, \psi_j(\xi)\,T_{a_L}^\varphi f_0(y)\dd y \ddd \xi  \\ &\qquad =  \iint_{\Rl^n\times \Rl^n}  \jap{\xi}^{-2N_1} L_y^{N_1} \left (\jap{x-y}^{-2N_2}\, L_\xi^{N_2} e^{i( x-y)\cdot\xi }\right ) \psi_j(\xi)\,T_{a_L}^\varphi f_0(y)\dd y \ddd \xi  \\ & \qquad =
\iint_{\Rl^n\times \Rl^n }     e^{i( x-y)\cdot\xi }   L_\xi^{N_2} \left (\jap{\xi}^{-2N_1}{\psi_j(\xi)}\right ) \,\jap{x-y}^{-2N_2}\, L_y^{N_1} T_{a_L}^\varphi  f_0(y)\dd y \ddd \xi.
}

Since $\psi_j$ is supported on an annulus of size $2^j$ one has
\eq{
\int_{\Rl^n} \Big| L_\xi^{N_2} \jap{\xi}^{-2N_1}\,\psi_j(\xi) \Big| \ddd \xi 
&\lesssim \sum_{|\alpha|\leq 2N_2} \int_{|\xi| \sim 2^j}  \Big|\partial_\xi^\alpha \Big (\la \xi \ra ^{-2N_1}\,\psi_j(\xi) \Big ) \Big|\ddd \xi
\\  
&\lesssim 2^{jn} \sum_{|\alpha|\leq 2N_2} 2^{-j(2N_1+|\alpha|)} 
\lesssim  2^{j(n-2N_1)}.
}
Also, applying Leibniz's rule and Fa\`a di Bruno's formulae we have that
\eq{
 L_y^{N_1} T_{a_L}^\varphi f(y)
 &= \int_{\Rl^n} L_y^{N_1} \big(a_L(y,\eta)\, e^{i\varphi(y,\eta)}\big)\, \widehat {f_0}(\eta) \ddd   \eta \\
 &= \int_{\Rl^n}
 \sigma(y,\eta)\, e^{i\varphi(y,\eta)}\, \widehat{f_0}(\eta) \ddd \eta
=:T_\sigma^\varphi  f_0(y),}
with 
\nm{eq:sigma}{
\sigma(y,\eta) := \sum_{|\alpha|\leq 2N_1 }\sum_{1\leq |\beta|\leq 2N_1 } \sum_{\ell\leq 2N_1}  C_{\alpha,\beta, \ell} \,\partial_y^\alpha  a_L(y,\eta) \big (\partial_y^\beta \varphi(y,\eta) \big)^\ell .
}

Thus, we have
\begin{equation}\label{pointwise estimate for LP piece}
\left |\psi_j(D)T_{a_L}^\varphi f(x) \right | \lesssim
2^{j(n-2N_1)} \Big (\la \cdot \ra ^{-2N_2} * \left |T_\sigma^\varphi f_0 \right | \Big )(x).
\end{equation}

Using the LF($\mu$) assumption one has
\eq{
\abs{\partial_\eta^\alpha\partial_y^\beta \varphi(y,\eta)} \lesssim 
\abs\eta^{\mu-\abs\alpha},
}
for $\abs\alpha\geq0$, $\abs\beta\geq1$. 
The terms of \eqref{eq:sigma} where $\ell=0$ are bounded by $1$ and the terms where $\ell\geq 1$ are bounded by $\abs\eta^{\mu-\abs\alpha}$.

Hence Lemma \ref{low freq lemma llf 1}, using both \eqref{eq:lf2} ($\ell=0$) and \eqref{eq:lf1} ($\ell\geq 1$), yields that for all $0 < \varepsilon <1$ the kernel of $T^\varphi_\sigma$ satisfies the estimate
\begin{equation*}
|K(x,y)| \lesssim \la x-y\ra^{-n-\varepsilon\mu}.
\end{equation*}

Now it follows from \eqref{pointwise estimate for LP piece}, the kernel estimate above and Lemma \ref{grafakos lemma 1} with \linebreak$ r>n/(n+\mu)$ that
\begin{equation}\label{punktvis lagfrek 1}
\begin{split}
\left |\psi_j(D)T_{a_L}^\varphi f(x) \right |
&\lesssim 2^{j(n-2N_1)} \int_{\Rl^n} \brkt{\int_{\Rl^n} \langle x-z\rangle^{-2N_2} \,\langle z-y\rangle^{-n-\varepsilon \mu} \dd z}\, |f_0(y)|\dd y\\
&\lesssim 2^{j(n-2N_1)} \int_{\Rl^n} \langle x-y\rangle^{-n-\varepsilon\mu} \, |f_0(y)|\dd y \\
& \lesssim  2^{j(n-2N_1)} \Big( M \left (|f_0|^r \right)(x)\Big)^{1/r} .
\end{split}
\end{equation}

This yields that for $r<p\leq \infty$ one has
\begin{equation}\label{jaevla namn 3}
\Vert \psi_j(D)T_{a_L}^\varphi f \Vert_{L^{p}(\Rl^n)}
 \lesssim 2^{j(n-2N_1)} \norm{f_0}_{L^p(\Rl^n)}.
\end{equation}
In case $iii)$ we would also like to extend \eqref{jaevla namn 3} to the range $0<p\leq\infty$ when $a(x,\xi)$ has compact support in $x$. If $ \mathcal{K}:=\supp_{y}\sigma(y,\eta)$, then since $f_0$ is frequency localised, Lemma \ref{grafakos lemma 1} and Peetre's inequality yield that for $r> n/(n+\mu)$, we have the pointwise estimate
\begin{equation}\label{punktvis lagfrek 2}
\begin{split}
\left |\psi_j(D)T_{a_L}^\varphi f(x) \right |
&\lesssim 2^{j(n-2N_1)}  
\Big(\la \cdot \ra ^{-2N_2} * \chi_{\mathcal{K}}\,\Big( M \left (|f_0|^r \right)\Big)^{1/r}\Big)(x)\\
&\lesssim 2^{j(n-2N_1)} \la x \ra ^{-2N_2}  \int_{\mathcal{K}} \Big( M \left (|f_0|^r \right)(y)\Big)^{1/r}\dd y,
\end{split}
\end{equation}
where $\chi_{\mathcal{K}}$ is the characteristic function of ${\mathcal{K}}$.
Now taking the $L^p$-norm, choosing $N_2$ large enough, using the $L^\infty$-boundedness of the Hardy-Littlewood maximal operator, and finally using Lemma \ref{lem:bernstein}, we obtain for $0<p\leq \infty$
\begin{equation}\label{jaevla namn 2}
\begin{split}
\Vert \psi_j(D)T_{a_L}^\varphi f \Vert_{L^{p}(\Rl^n)}
 &\lesssim  2^{j(n-2N_1)}\Vert |f_0|^r \Vert_{L^{\infty}(\Rl^n)}^{1/r}
 \lesssim 2^{j(n-2N_1)} \norm{f_0}_{L^\infty(\Rl^n)}\\ 
 &\lesssim 2^{j(n-2N_1)} \norm{f_0}_{L^p(\Rl^n)}.
\end{split}
\end{equation}
Thus, \eqref{jaevla namn 3} and \eqref{jaevla namn 2} yield for $N_1$ large enough
\begin{align*}
&\norm{T_{a_L}^\varphi f}_{B_{p,q_2}^{s_2}(\Rl^n)} 
= \Big(\sum_{j=0}^\infty 2^{js_2q_2}\norm{\psi_j(D)T_{a_L}^\varphi f}_{L^p(\Rl^n)}^{q_2} \Big)^{1/q_2} \\
& \qquad \lesssim \Big(\sum_{j=0}^\infty 2^{jq_2(s_2+n-2N_1)} \norm{f_0}_{L^{p}(\Rl^n)}^{q_2} \Big)^{1/q_2} 
 =  \norm{f_0}_{L^{p}(\Rl^n)}
\Big(\sum_{j=0}^\infty 2^{jq_2(s_2+n-2N_1)} \Big)^{1/q_2}\\
& \qquad \lesssim \norm{ f_0}_{L^{p}(\Rl^n)}\lesssim \norm{ f}_{B_{p,q_1}^{s_1}(\Rl^n)}.
\end{align*}
In the case of boundedness in Triebel-Lizorkin spaces for $ii)$, we use \eqref{punktvis lagfrek 1} and the assumption that $p>r>n/(n+\mu)$ which yield for $N_1$ large enough that
\begin{equation*}
\begin{split}
\norm{T_{a_L}^\varphi f}_{F_{p,q_2}^{s_2}(\Rl^n)} &= \Big\|\Big(\sum_{j=0}^\infty 2^{js_2q_2}\abs{\psi_j(D)T_{a_L}^\varphi f}^{q_2} \Big)^{1/q_2}\Big\|_{L^p(\Rl^n)} \\ 
& \lesssim \Big\|\Big(\sum_{j=0}^\infty 2^{js_2q_2} \Big|2^{j(n-2N_1)} \Big( M (|f_0|^r )\Big)^{1/r}\Big|^{q_2} \Big)^{1/q_2}\Big\|_{L^p(\Rl^n)}\\ 
& \lesssim  \Big(\sum_{j=0}^\infty 2^{jq_2(s_{2} + n -2 N_1)}\Big)^{1/q_2} 
\Big\| \Big( M \left (|f_0|^r \right)\Big)^{1/r} \Big\|_{L^p(\Rl^n)}\\
& \lesssim \norm{ f_0}_{L^{p}(\Rl^n)}
\lesssim \norm{ f}_{F_{p,q_1}^{s_1}(\Rl^n)}.
\end{split}
\end{equation*}
In the case of boundedness in Triebel-Lizorkin spaces for $iii)$, we use \eqref{punktvis lagfrek 2} and Lemma \ref{lem:bernstein}  to see that for all $p>0$ one has
\begin{equation*}
\begin{split}
\norm{T_{a_L}^\varphi f}_{F_{p,q_2}^{s_2}(\Rl^n)} &= \Big\|\Big(\sum_{j=0}^\infty 2^{js_2q_2}\abs{\psi_j(D)T_{a_L}^\varphi f}^{q_2} \Big)^{1/q_2}\Big\|_{L^p(\Rl^n)}\\ 
& \lesssim  \Big(\sum_{j=0}^\infty 2^{jq_2(s_{2} + n -2 N_2 )}\Big)^{1/q_2} \norm{ \langle \cdot\rangle^{-2N_2}}_{L^p(\Rl^n)} \int_{\mathcal{K}} \Big( M  (|f_0|^r) (x)\Big)^{1/r}\dd x\\
& \lesssim \norm{ f_0}_{L^{\infty}(\Rl^n)}
\lesssim \norm{ f}_{F_{p,q_1}^{s_1}(\Rl^n)},
\end{split}
\end{equation*}
by choosing $N_2$ large enough.
\end{proof}

\begin{Rem}
Note that the type of the phase $($i.e. the $\mu$ in the \emph{LF}$(\mu)$-condition \eqref{eq:LFmu}$)$ enters the picture only at the level of quasi-Banach boundedness of the oscillatory integral operators.
\end{Rem}

\section{Boundedness of middle frequency portion}
\label{sect:midfreq}
In this section we show that for the portion of the operator where the frequency support of the amplitude is bounded below, away from the origin and also bounded from above by a fixed $R\gg 1$, then the middle portion of the operator is bounded on Besov-Lipschitz and Triebel-Lizorkin spaces, as the following lemma shows:\\

\begin{Lem}\label{Lem:midfreq}
Assume that $\psi_0(\xi)\in \mathcal{C}_{c}^{\infty}(\Rl^n)$ is a smooth cut-off function supported in a neighborhood of the origin as in \emph{Definition \ref{def:LP}}, $a(x,\xi)\in S^m_{0,0}(\Rl^n)$ for some $m\in \Rl$,   $a_M (x,\xi) := (\psi_0(\xi/R)-\psi_0(\xi))\,a(x,\xi)$ for some $R>1$ and let $\varphi(x,\xi)$ be a phase function satisfying the $\textart F^k$-condition.
Finally let the operator $T_{a_M}^\varphi$ be defined as in \eqref{eq:OIO}. Then the following statements hold:
\begin{enumerate}
\item[$i)$] $T_{a_M}^\varphi$ satisfies
\begin{equation*}
\Vert T_{a_M}^\varphi f\Vert_{L^\infty (\Rl^n)}\lesssim \Vert f\Vert_{\bmo(\Rl^n)}.
\end{equation*}
 \item[$ii)$] Assume that $\partial^\beta_x \varphi(x,\xi) \in L^\infty(\mathbb{R}^n\times \mathbb S^{n-1})$, for any $|\beta| \geq 1$ and that $\varphi$ satisfy the $L^2$-condition \eqref{eq:L2 condition}. Then for any $s_1, s_2\in (-\infty, \infty)$ and $p,q_1, q_2 \in (0,\infty]$ one has 
\begin{equation*}
\Vert T_{a_M}^\varphi f\Vert_{B_{p,q_2}^{s_2}(\Rl^n)}\lesssim \Vert f\Vert_{B_{p,q_1}^{s_1}(\Rl^n)}.
\end{equation*}
\item[$iii)$] Assume that $a(x,\xi)$ has compact support in the $x$-variable and that $\varphi(x,\xi)$ satisfies the $L^2$-condition \eqref{eq:L2 condition}.  Then for any $s_1, s_2\in (-\infty, \infty)$, and $p, q_1, q_2 \in (0,\infty]$
\begin{equation*}
\Vert T_{a_M}^\varphi f\Vert_{B_{p,q_2}^{s_2}(\Rl^n)}\lesssim \Vert f\Vert_{B_{p,q_1}^{s_1}(\Rl^n)}.
\end{equation*}
\end{enumerate}
Moreover, all the Besov-Lipschitz estimates above can be replaced by the corresponding Triebel-Lizorkin estimates.
\end{Lem}
\begin{proof}
The proof is similar to that of Lemma \ref{Lem:lowfreq}. The only difference is that we cannot use any of \eqref{eq:LFmu}, \eqref{eq:lf1} or \eqref{eq:lf2} to obtain kernel estimates. Instead we observe that for any $N\geq0$
\nm{eq:midfreqkernelestimate}{
\abs{K(x,y)} &= \Big | \int_{\Rn}e^{i(x-y)\cdot\xi}\,e^{i\varphi(x,\xi)-ix\cdot\xi}\,a_M(x,\xi) \ddd\xi\Big| \\&= \Big | \jap{x-y}^{-2N}\int_{\Rn}e^{i(x-y)\cdot\xi}\,(1-\Delta_\xi)^N\,e^{i\varphi(x,\xi)-ix\cdot\xi}\,a_M(x,\xi)\ddd\xi\Big| \\&\lesssim \jap{x-y}^{-2N},
}
using the $\textart F^k$-condition. This is enough to conclude the result in $i)$.\\

For $ii)$ and $iii)$ we need to replace $a_L$ with $a_M$ in \eqref{eq:sigma} and obtain estimate \eqref{eq:midfreqkernelestimate} for the kernel of $T_\sigma^\varphi$. The only problem here is to control the factors of the form $\partial_\eta^\alpha\partial_y^\beta \varphi(y,\eta)$ where $\abs\alpha\geq0$ and $\abs\beta\geq 1$. But they are uniformly bounded because of the $L^2$-condition when $\abs\alpha\geq 1$ and
\eq{
\abs{\partial_y^\beta \varphi(y,\eta)}  &\leq \abs{\partial_y^\beta \varphi(y,\eta)-\partial_y^\beta \varphi(y,\eta_0)}+\abs{\partial_y^\beta \varphi(y,\eta_0)}\lesssim \abs{\eta-\eta_0}+\abs{\partial_y^\beta \varphi(y,\eta_0)}\lesssim 1,
}
for $\abs\alpha=0$ and $\eta \in \supp a_M$, if we choose $\eta_0:=\eta/\abs\eta$. Hence $\abs{\partial_\eta^\alpha\sigma(y,\eta)}\lesssim 1$ which yields the estimate \eqref{eq:midfreqkernelestimate}, and we can proceed as in Lemma \ref{Lem:lowfreq} from equation \eqref{punktvis lagfrek 1} onwards.
\end{proof}

\section{Local \texorpdfstring{$h^p-L^p$}\ \ boundedness}
\label{sec:local Lp-boundedness}

In this section, we prove the local $h^p-L^p$ boundedness of oscillatory integral operators. As it turns out, for the case of $0<p<1$ and the local $h^p-L^p$ boundedness of $T_a^\varphi$, no condition on the phase function is required. Moreover, the order of the amplitude could also be larger than the critical order $m_k(p)$. More explicitly we have

\begin{Prop}\label{Prop:ThpLp}
Let $0<p<1$ and {$m=-n/p$} and suppose that $\varphi (x,\xi)$ is a measurable real-valued function, $a (x,\xi)\in S^{m}_{0,0}(\Rl^n)$ with compact support in the \linebreak$x$-variable. Then $T_a^\varphi$ as given in \eqref{eq:OIO} is a bounded operator from $h^p(\Rn)$ to $L^p(\Rn)$.
\end{Prop}

\begin{proof}
Fix a $p$-atom $\at$ supported in the ball $B:=B(\bar{y},r)$, with $\bar{y} \in \Rn$ and \linebreak$r>0$. Also, make the Littlewood-Paley decomposition using Definition \ref{def:LP}, so that \linebreak$T_a^\varphi = \sum_{j=0}^\infty T_j$ where $T_j := T_a^\varphi\, \psi_j(D)$. By Lemma \ref{lem:hpLpofT}, and since $T_j\at$ has compact support, it is enough to show that
\begin{equation}\label{eq:keyestimate1new}
\| T_j \at \|_{L^\infty(\Rl^n)} \lesssim r^{n-n/p}2^{j\brkt{n-n/p}},
\end{equation}
and  whenever $r <1$
\begin{equation}\label{eq:keyestimate2new}
\| T_j \at \|_{L^\infty(\Rl^n)} \lesssim r^{{N} +1+n-n/p }2^{j\brkt{{N}+1+n-n/p}},
\end{equation} 
for some $N > n/p-n-1$.\\

First of all, Lemma \ref{Lem:pointwiseKjsimplified} taken with $\beta=0$, yields
 for all $x \in \Rn$
\begin{equation*}\label{eq:bound1Tjsimplified}
|T_j\at(x)|
 \leq  \int_{B} |K_j(x,y)| |\at(y)| \dd y 
 \lesssim r^{n-n/p} \, 2^{j(n+m)},
\end{equation*}
which gives \eqref{eq:keyestimate1new}.\\

On the other hand, if $r<1$, we Taylor expand the kernel as follows
\begin{align*}
K_j(x,y)
& = \sum_{ |\beta| \leq N} \frac{(y-\bar{y})^\beta}{\beta!}  (\partial^\beta_y K_j)(x,\bar{y}) \\
&  \qquad + (N+1) \sum_{ |\beta| = N+1}  \frac{(y-\bar{y})^\beta}{\beta!}  \int_0^1 (1-t)^N (\partial^\beta_y K_j) (x,ty + (1-t)\bar{y})\dd t,
\end{align*}
and taking advantage of the vanishing moments of the atom, we obtain
\begin{align*}
T_j\at(x)
&  = (N+1) \sum_{ |\beta| = N+1}  \int_{B} \int_0^1  \frac{(y-\bar{y})^\beta}{\beta!}   (1-t)^N (\partial^\beta_y K_j)(x,ty + (1-t)\bar{y}) 
 \at(y) \dd t  \dd y.
\end{align*}
Therefore, applying once again Lemma \ref{Lem:pointwiseKjsimplified}, with $|\beta|=N+1$
and $N :=[n(1/p-1)]$, we obtain
\begin{equation*}\label{eq:pointwiseTjasimplified}
|T_j\at(x)|
\lesssim r^{N+1+n-n/p} \, 2^{j(N+1+n+m)},
\end{equation*}
which yields \eqref{eq:keyestimate2new}.
\end{proof}

\begin{Rem}
We observe that interpolating the result of \emph{Proposition \ref{Prop:ThpLp}} with the $L^2$-boundedness of operators with amplitudes in $S^0_{0,0}(\Rl^n)$ yields that $T^\varphi_a$ is bounded from $h^p(\Rl^n)$ to $L^p(\Rl^n)$ for $0<p\leq 2$ with a \emph{SND} phase function verifying \eqref{eq:L2 condition} and $m<m_1(p)$.
\end{Rem}
\section{Boundedness of high frequency portion}
\label{sect:highfreq}
In this section, we treat the global regularity of the high frequency portion of oscillatory integral operators. Here we prove $h^p-L^p$ boundedness results.

\begin{Prop}\label{Th:new}
Suppose that $\varphi \in \textart F^k$ is \emph{SND},
for some $k \geq 1$ and satisfy the \linebreak$L^2$-condition \eqref{eq:L2 condition}. Let $a (x,\xi) \in S^{m_k(p)}_{0,0}(\Rl^n)$ and $a_H(x,\xi) := (1-\psi_0(\xi))\,a(x,\xi)$ where $\psi_0$ is given in \emph{Definition \ref{def:LP}}. Then, $T_{a_H}^\varphi$ as in \eqref{eq:OIO}, is a bounded operator from $h^p(\Rn)$ to $L^p(\Rn)$, when $0<p<1$.  In the case $0<k<1$, the result above is true provided that $a(x,\xi) \in S^{m_k(p)}_{1,0}(\Rl^n)$.
\end{Prop}

\begin{proof}
We consider a generic Littlewood-Paley piece of the kernel of $T^\varphi_{a_H}$:
\begin{equation}\label{eq:2Bornot2B}
K_j(x,x-y)
:=\int_{\Rl^n} a_j(x,\xi)\,e^{i(\varphi(x,\xi)-x\cdot\xi+(x-y)\cdot\xi)}\ddd\xi,
\end{equation}
where $a_j(x,\xi) := a_H(x,\xi)\,\psi_j(\xi)$. In light of Lemma \ref{lem:Miyachistuff}, we only need to show that $T^\varphi_{a_H}$ is $L^2$-bounded, which is indeed the case by 
Theorem \ref{thm:fujiwara}, and that 
\eq{
\Vert (x-y)^\alpha \,\partial^{\beta} _{y}K_{j}(x,x-y)\Vert_{L^2_x(\Rl^n)}
&\lesssim 2^{j(\vert\alpha\vert(k-1) +\vert \beta\vert+m_k(p)+n/2)}.
}
However, since differentiating \eqref{eq:2Bornot2B} $\beta$ times in $y$ will only introduce factors of the size $2^{j|\beta|},$ it is enough to establish the above estimate for $\beta=0$. To this end, take ${\Psi}_j$ as in Definition \ref{def:LP}, integrate by parts and rewrite 

\eq{
&(x-y)^\alpha K_j(x,x-y)
=\int_{\Rl^n} a_j(x,\xi)\,e^{i\varphi(x,\xi)-ix\cdot\xi}\,\partial_\xi^\alpha e^{i(x-y)\cdot\xi} \ddd\xi\\
& \quad =\int_{\Rl^n} \partial_\xi^\alpha \Big[a_j(x,\xi)e^{i\varphi(x,\xi)-ix\cdot\xi}\Big]\, e^{i(x-y)\cdot\xi} \,{\Psi}_j (\xi)\ddd\xi \\
& \quad = \sum_{\alpha_1 + \alpha_2=\alpha} \!\!C_{\alpha_1, \alpha_2}\int_{\Rl^n}
\partial_\xi^{\alpha_1} a_j(x,\xi) \,
\partial_\xi^{\alpha_2} e^{i\varphi(x,\xi)-ix\cdot\xi} \,e^{i(x-y)\cdot\xi}\, {\Psi}_j (\xi)\ddd\xi \\
& \quad = \!\!\sum_{\substack{\alpha_1 + \alpha_2=\alpha \\ \lambda_1 + \dots + \lambda_r = \alpha_2} }\!\!\!\!
C_{\alpha_1, \alpha_2, \lambda_1, \dots \lambda_r} \int_{\Rl^n} 
\partial_\xi^{\alpha_1} a_j(x,\xi) \\
& \quad \qquad \qquad   \times
\partial_\xi^{\lambda_1}(\varphi(x,\xi)-x\cdot\xi)
\cdots
\partial_\xi^{\lambda_r}(\varphi(x,\xi)-x\cdot\xi)
\,e^{i\varphi(x,\xi)} \,e^{-iy\cdot\xi}\,{\Psi}_j (\xi) \ddd\xi \\
& \quad =\!\! \sum_{\substack{\alpha_1 + \alpha_2=\alpha \\ \lambda_1 + \dots + \lambda_r = \alpha_2} }\!\!\!\!
C_{\alpha_1, \alpha_2, \lambda_1, \dots \lambda_r} 
2^{j(m_k(p)+(k-1)|\alpha|)}  \!\!
\int_{\Rl^n} b_j^{\alpha_1,\alpha_2, \lambda_1, \dots, \lambda_r}(x,\xi)\,
 e^{i\varphi(x,\xi)}\, e^{-iy\cdot\xi}\,{\Psi}_j (\xi) \ddd\xi \\
& \quad =:\!\! \sum_{\substack{\alpha_1 + \alpha_2=\alpha \\ \lambda_1 + \dots + \lambda_r = \alpha_2} }
\!\!\!\!C_{\alpha_1, \alpha_2, \lambda_1, \dots \lambda_r} 
2^{j(m_k(p)+(k-1)|\alpha|)}   
S_{j}^{\alpha_1, \alpha_2, \lambda_1, \dots \lambda_r}(\tau_{-y}\widehat\Psi_j)(x),
}
where $\tau_{-y}$ is a translation by $-y$, $|\lambda_j| \geq 1$ and
\begin{align*}
b_j^{\alpha_1,\alpha_2, \lambda_1, \dots, \lambda_r}(x,\xi)
& := 2^{-j(m_k(p)+(k-1)|\alpha|)} \\
& \qquad \times \partial_\xi^{\alpha_1} a_j(x,\xi) \, 
\partial_\xi^{\lambda_1}(\varphi(x,\xi)-x\cdot\xi)
\dots
\partial_\xi^{\lambda_r}(\varphi(x,\xi)-x\cdot\xi).
\end{align*}
Now we claim that 
$b_j^{\alpha_1,\alpha_2, \lambda_1, \dots, \lambda_r}(x,\xi) \in S^0_{0,0}(\Rl^n)$ uniformly in $j$. Indeed, since \linebreak$a \in S^{m_k(p)}_{0,0}(\Rl^n)$ and $\varphi\in \textart F^k$ (with $k\geq 1$), we can write
\begin{align*}
\abs{b_j^{\alpha_1,\alpha_2, \lambda_1, \dots, \lambda_r}(x,\xi)}
&\lesssim 2^{-j(m_k(p)+(k-1)|\alpha|)} \, 
2^{jm_k(p)} \,
2^{j(k-1)r} \\
& \leq 2^{-j(k-1)|\alpha|}  \,
2^{j(k-1)r} 
{\lesssim 1}.
\end{align*}

In a similar way, using the $\textart F^k$-condition, we can also check that, for any multi-indices $\gamma$ and $\beta $, 
\begin{equation}\label{eq:bjs}
    \abs{ {\partial_{\xi}^\gamma \partial_x^\beta} \, b_j^{\alpha_1,\alpha_2, \lambda_1, \dots, \lambda_r}(x,\xi)}
\lesssim 1,
\end{equation}
hence $b_j^{\alpha_1,\alpha_2, \lambda_1, \dots, \lambda_r}\in S^0_{0,0}(\Rl^n).$
In the case $0<k<1$, the hypothesis on $a$ and $\varphi$ yield that
$$|{\partial_{\xi}^\alpha \partial_x^\beta} a_j(x,\xi)|
\lesssim 2^{j(m_k(p) -|\alpha|)},$$
and on the support of $a_j$
$$|{\partial_{\xi}^\alpha \partial_x^\beta} (\varphi(x,\xi)-x\cdot\xi)|
\lesssim 
{2^{j(k-|\alpha|)} },$$

which together imply \eqref{eq:bjs}.\\

Therefore, {Theorem \ref{thm:fujiwara}} yields that 
\begin{align*}
\Vert (x-y)^\alpha K_{j}(x,x-y)\Vert_{L^2_x(\Rl^n)}
& \lesssim \!\!\sum_{\substack{\alpha_1 + \alpha_2=\alpha \\ \lambda_1 + \dots + \lambda_r = \alpha_2} }\!\! 2^{j(m_k(p)+(k-1)|\alpha|)} 
\|S_{j}^{\alpha_1, \alpha_2, \lambda_1, \dots \lambda_r}(\tau_{y}\widehat\Psi_j)\|_{L^2(\Rl^n)} \\
& \lesssim 2^{j(\vert\alpha\vert(k-1) +m_k(p))} \|\widehat\Psi_j\|_{L^2(\Rl^n)} 
\lesssim 2^{j(\vert\alpha\vert(k-1) +m_k(p)+n/2)},
\end{align*}
and the proof is completed.
\end{proof}
We would like to have a similar result for the adjoint operator, but in this case we need to add an extra condition. However, this extra condition is automatically fulfilled if one assumes LF$(\mu)$-condition \eqref{eq:LFmu}, and it turns out to be superfluous as far as the $L^p$-boundedness is concerned. Since the result is only applied in these two cases, 
this extra condition will not have any impact on any of the main results. In the following proposition, we let {$e_\ell$} be the unit vectors as in the proof of Theorem \ref{thm:main1globalLp} on page \pageref{thm:main1globalLp}.

\begin{Prop}\label{Tj:Tad global}
Let $a (x,\xi) \in S^{m_k(p)}_{0,0}(\Rl^n)$ and $a_H(x,\xi) := (1-\psi_0(\xi/\sqrt n))\,a(x,\xi)$, where $\psi_0$ is given in \emph{Definition \ref{def:LP}}.
Suppose that, for $k\geq 1$, $\varphi \in \textart F^k$ is \emph{SND} and  satisfies the $L^2$-condition \eqref{eq:L2 condition}. Moreover assume that for all $\xi\in \supp_\xi a_H(x,\xi)$, there exists $1\leq\ell\leq2n$, such that the line segment between $\xi$ and $e_\ell$ does not 
pass through the unit ball $B(0,1)$ and such that $\partial^\beta_x \varphi(x,e_\ell) \in L^\infty(\mathbb{R}^n)$, for all  $|\beta| \geq 1$.
Then, ${(T_{a_H}^\varphi)}^*$ given as in \eqref{eq:adjointOIO} is a bounded operator from $h^p(\Rn)$ to $L^p(\Rn)$ when $0<p<1$.
In the case $0<k<1$, the result above is true provided that $a(x,\xi) \in S^{m_k(p)}_{1,0}(\Rl^n)$.
\end{Prop}
\begin{proof}

The proof follows the same lines as that of Lemma \ref{Lem:smoothlowfreq}. Indeed since ${(T_{a_H}^\varphi)}^*$ is $L^2$-bounded, we only need to show that
\begin{equation}\label{eq:target}
\norm{\rho_j(y,\cdot)} _{L^2_\xi(\Rl^n)}
    \lesssim 2^{j(\vert\alpha\vert(k-1) +\vert \beta\vert+m_k(p)+n/2)},
\end{equation}
where
\begin{equation}\label{defn of the function rho}
    \rho_j(y,\xi)
:= \sum_{\substack{\alpha_1 +\alpha_2 =\alpha \\ \beta_1 +\beta_2 =\beta \\  \beta_2 \geq \alpha_2}} 
C_{\alpha,\beta} \partial^{\alpha_1}_{\xi}\partial^{\beta_1}_y 
\sigma_j(y,\xi)\,  \xi^{\beta_2 -\alpha_2},
\end{equation}
$$\sigma_j (y,\xi) 
:= \overline{a_j(y,\xi)}\,e^{-i\varphi(y,\xi) + i y \cdot \xi},$$
and
$a_j(y,\xi) := a_H(y,\xi)\,\psi_j(\xi)$ is the usual Littlewood-Paley piece. To this end, Leibniz's rule yields that 
\begin{align*}
| \partial^{\alpha_1}_\xi \partial^{\beta_1}_y \sigma_j(y, \xi) |
& = |  \partial^{\alpha_1}_\xi \partial^{\beta_1}_y (
\overline{a_j(y,\xi)} \, e^{-i\varphi(y,\xi)+iy\cdot\xi}) | \\
& \lesssim \sum_{ \substack{\alpha_1'+\alpha_1''= \alpha_1  \\ \beta_1'+\beta_1'' = \beta_1 }} 
|\partial^{\alpha_1'}_\xi \partial^{\beta_1'}_y \overline{a_j(y,\xi)}|
\, |\partial^{\alpha_1''}_\xi \partial^{\beta_1''}_y e^{-i\varphi(y,\xi)+iy\cdot\xi}|.
\end{align*}

Now, if we let $\Phi(y,\xi):=\varphi(y, \xi) - y \cdot \xi$, then 
Fa\`a di Bruno's formulae implies
\begin{align}\label{eq:kintegrallemma3}
\vert \partial^{\alpha}_\xi\partial^{\beta}_y e^{-i\Phi(y,\xi)}  \vert 
& \lesssim \sum _{(\gamma_1,\delta_1)+\dots+(\gamma_r,\delta_r)=(\alpha,\beta)}
 \vert \partial^{\gamma_1}_\xi\partial^{\delta_1}_y
\Phi(y,\xi) \vert ... \vert\partial^{\gamma_r}_\xi\partial^{\delta_r}_y \Phi(y,\xi) \vert,
\end{align}
where the sum above runs over all possible partitions of $(\alpha,\beta)$ such that $|\gamma_\nu|+|\delta_\nu|\geq 1$ for $\nu=1,\dots, r$.\\

The $\textart F^k$-condition isn't enough to estimate the terms in \eqref{eq:kintegrallemma3} and we also need to derive estimates for the derivatives in $x$. 
For any $\xi\in \supp_\xi a_H(y,\xi)$, take $\ell$ as in the statement of this theorem. Then the $L^2$-condition and the mean-value theorem yield that
\nm{eq:phaseestimate}{\abs{\partial_y^\gamma \varphi(y,\xi)}&\leq \abs{\partial_y^\gamma \varphi(y,\xi)-\partial_y^\gamma \varphi(y,e_\ell)}+\abs{\partial_y^\gamma \varphi(y,e_\ell)}\\&\lesssim \abs{\xi-e_\ell}+\abs{\partial_y^\gamma \varphi(y,e_\ell)} \lesssim \abs\xi.}
  
Hence, on the support of $a_j$ one has, for $k \geq 1$,
$$
|\partial^{\gamma}_\xi\partial^{\delta}_y \Phi(y,\xi)|
= \left\{
\begin{array}{ll}
   O(2^j),  &  \gamma =0, \\
   O(2^{j(k-1)}),  & \gamma \neq 0,   
\end{array}
\right.$$
and for $0<k<1$
$$
|\partial^{\gamma}_\xi\partial^{\delta}_y \Phi(y,\xi)|
= \left\{
\begin{array}{ll}
   O(2^j),  &  \gamma =0, \\
   O(2^{j(k-|\gamma|)}),  & \gamma \neq 0,   
\end{array}
\right.$$
where we have used 
the $\textart F^k$-condition, $L^2$-condition and \eqref{eq:phaseestimate}. Therefore, for $k \geq 1$, using \eqref{eq:kintegrallemma3} we get
\begin{align*}
| \partial^{\alpha_1}_\xi \partial^{\beta_1}_y \sigma_j(y, \xi) |
& \lesssim  2^{j(m_k(p)+(k-1)\vert\alpha_1\vert +|\beta_1|)}.
\end{align*}
On the other hand, in the case $0<k<1$ using the assumption $a\in S^{m_k(p)}_{1,0}(\Rl^n)$ we obtain
\begin{align*}
| \partial^{\alpha_1}_\xi \partial^{\beta_1}_y \sigma_j(y, \xi) |
& \lesssim  
\sum_{ \substack{\alpha_1'+\alpha_1''= \alpha_1  \\ \beta_1'+\beta_1'' = \beta_1 }} 
|\partial^{\alpha_1'}_\xi \partial^{\beta_1'}_y \overline{a_j(y,\xi)}|
\, |\partial^{\alpha_1''}_\xi \partial^{\beta_1''}_y e^{-i\varphi(y,\xi)+iy\cdot\xi}| \nonumber \\
& \lesssim \sum_{ \substack{\alpha_1'+\alpha_1''= \alpha_1  \\ \beta_1'+\beta_1'' = \beta_1 }} 2^{j(m_k(p)-|\alpha_1'| + |\alpha_1''|k - |\alpha_1''|+|\beta_1''|)} \nonumber \\
& \lesssim 2^{j(m_k(p)+(k-1)\vert\alpha_1\vert +|\beta_1|)} .
\end{align*}
Thus in both cases
\begin{equation}\label{eq:sigmaj}
\begin{split}
| \partial^{\alpha_1}_\xi \partial^{\beta_1}_y \sigma_j(y, \xi) |
& \lesssim  2^{j(m_k(p)+(k-1)\vert\alpha_1\vert +|\beta_1|)}  \\
& = 2^{j(m_k(p)+(k-1)\vert\alpha\vert -(k-1)|\alpha_2|+|\beta_1|)}  \\
& \lesssim 2^{j(m_k(p)+(k-1)\vert\alpha\vert +|\alpha_2|+|\beta_1|)}.
\end{split}
\end{equation}
Finally, combining \eqref{eq:sigmaj} and \eqref{defn of the function rho} we obtain \eqref{eq:target}. Hence 
{Lemma \ref{lem:Miyachistuff} holds} and the proof is concluded.
\end{proof}

\section{The \texorpdfstring{$h^p-L^p$}\ \ boundedness of Schr\"odinger integral operators}
\label{sect:hpLpScho}

This section deals with the regularity of the Sch\"odinger integral operators. An important tool in the proof of the following theorem is a Littelwood-Paley decomposition of the amplitude, where each Littlewood-Paley annulus is further decomposed into a union of balls with constant radii, in contrast to the second frequency localisation introduced by C. Fefferman in \cite{Feff}, where different pieces of the amplitude are supported in "angular-radial rectangles".

\begin{Th}\label{thm:schrodinger}
Let  $T_a^\varphi$  be a Schr\"odinger integral operator according to $\mathrm{Definition}$ $\ref{def:SIO}$ with amplitude $a(x,\xi) \in S_{0,0}^{m_2(p)}(\Rl^n)$ and phase function $\varphi$ that is \emph{SND}. Then $T_a^\varphi$ is a bounded operator from $h^p(\Rl^n)$ to $L^p(\Rl^n)$ for $	0<p<\infty$. Moreover if $|\nabla_x \varphi(x,0)|\in L^{\infty}(\Rl^n)$ then $T_a^\varphi$ is bounded from $L^\infty(\Rl^n)$ to $\bmo(\Rl^n)$. 

\end{Th}
\begin{proof}
We start by the analysis of the case of $0<p<1$. We make the following decomposition of the integral kernel 
$$K(x,y)= \int_{\Rl^n}  a(x,\xi)\,e^{i\varphi(x,\xi)-iy\cdot \xi}\,  \ddd \xi$$ 
of the operator $T_a^\varphi$.
We introduce a standard Littlewood-Paley partition of unity $\sum_{j=0}^\infty \psi_j(\xi) =1$ with $\supp \psi_0 \subset B(0,2)$, and $\supp \psi_j\subset \{ 2^{j-1}\leq|\xi|\leq 2^{j+1}\}$ for $j\geq 1$. Then for every $j\geq 0$ we cover  
$\supp \psi_j$ with open balls $C_j^\nu$ with radius $1$ and center $\xi_j^\nu$, where $\nu$ runs from $1$ to $O(2^{jn})$. Observe that $|C_j^\nu|\lesssim 1$ uniformly in $j$ and $\nu$. Now take $u\in \mathcal{C}_c^{\infty}(\Rl ^n)$, with $0\leq u\leq 1$ and supported in $B(0,2)$ with $u=1$ on $\overline{B(0,1)}.$  Define $\lambda_j^\nu(\xi) \in \mathcal{C}_c^{\infty}(\Rl^n)$ to be equal to $u(\xi-\xi_j^\nu).$
Next set
$\chi_j^\nu(\xi) := \lambda_j^\nu(\xi) /{\sum_\nu\lambda_j^\nu(\xi)} $ and observe that {for each $\xi \in \supp \psi_j$} the sum  $\sum_\nu\lambda_j^\nu(\xi) \geq 1$, and also
$\sum_{j=0}^\infty \sum_\nu \chi_j^\nu(\xi)\,\psi_j(\xi) =1.$
Now consider the kernel
\begin{equation*}
K_{j}^\nu (x,y) := \int_{\Rl^n}  \psi_j(\xi)\,\chi_j^\nu(\xi)\,e^{i\varphi(x,\xi)-iy\cdot \xi}\,a (x,\xi)  \ddd \xi .
\end{equation*}
Therefore, for any multi-index $\alpha$ and any $j\geq 0$ we have
\begin{align*}
\partial^{\alpha}_{y}K_{j}^{\nu}(x,y) 
& = \int_{\Rl^n}  \psi_j(\xi)\,\chi_j^\nu(\xi)\,
\partial^{\alpha}_{y} e^{i(\varphi(x,\xi )-y\cdot \xi)} \,a (x,\xi) \ddd \xi \\
& = 
\int_{\Rl^n}
e^{i(\varphi(x,\xi )-y\cdot \xi)}\, \sigma_j^{\alpha,\nu}(x,\xi) \ddd\xi,
\end{align*}
where
$$\sigma_{j}^{\alpha,\nu}(x,\xi)
:=  \psi_j(\xi)\,\chi_j^\nu(\xi)\,(-i\xi)^\alpha \, a(x,\xi).
$$
Using the assumption that $a(x,\xi)\in S^{m_2(p)}_{0,0}(\Rl^n)$, we deduce that for any multi-index $\gamma$, any $j\geq 0$ and any $\nu$ one has
\begin{equation}\label{amplitude derivative estim}
 |\partial^{\gamma}_\xi \sigma_j^{\alpha,\nu}(x,\xi)|\lesssim 2^{j(m_2(p)+|\alpha|)}.   
\end{equation}
If we now set $\vartheta(x,\xi):=\varphi(x,\xi)-\xi \cdot \nabla_\xi \varphi(x,\xi_j^\nu)$, then we can write
\begin{equation*}
\partial^{\alpha}_{y}K_{j}^\nu(x,y) = \int_{\Rl^n}  e^{i(\nabla_\xi\varphi(x,\xi_j^\nu )-y)\cdot \xi}\,e^{i\vartheta(x,\xi)}\,\sigma_{j}^{\alpha,\nu}(x,\xi)  \ddd \xi.
\end{equation*}
 Now we claim that the derivatives of $\vartheta$ in $\xi$ are uniformly bounded on the support of $\sigma_{j}^{\alpha,\nu}(x,\xi)$. To this end, the mean-value theorem and \eqref{Schrodinger phase} yield
\eq{
|\partial_{\xi_l}\vartheta(x,\xi)|
&=\abs{\partial_{\xi_l}\varphi(x,\xi)- \partial_{\xi_l} \varphi(x,\xi_j^\nu)}   
=  \Big| \brkt{\xi-\xi_j^\nu}\cdot \int_0^1\partial_{\xi_l}\nabla_\xi \varphi(x,t\xi+(1-t)\xi_j^\nu)\dd t \Big| \\
&\lesssim \abs{ \xi-\xi_j^\nu} \leq 1,
}
and 
\eq{
\abs{\partial_{\xi}^{\mu}\vartheta(x,\xi)} = \abs{\partial_{\xi}^{\mu}\varphi(x,\xi)}\lesssim 1,\quad \text{ for all } \abs\mu \geq 2.
}
Defining the differential operator $L$ by
\eq{
L := 1-i(\nabla_\xi \varphi(x,\xi_j^\nu)-y)\cdot \nabla_\xi,
}
one can easily verify that
\begin{equation*}
\jap{ \nabla_\xi \varphi(x,\xi_j^\nu)-y} ^{-2M} L^M e^{i\brkt{\nabla_\xi\varphi\brkt{x,\xi_j^\nu }-y}\cdot \xi} = e^{i\brkt{\nabla_\xi\varphi\brkt{x,\xi_j^\nu }-y}\cdot \xi},
\end{equation*}
for all integers $M\geq0$. Therefore, integrating by parts yields
\begin{equation*}
\partial^{\alpha}_{y} K_{j}^\nu(x,y)= 
\langle \nabla_\xi \varphi(x,\xi_j^\nu)-y \rangle^{-2M} \int_{\Rl^n} e^{i(\nabla_\xi\varphi(x,\xi_j^\nu )-y)\cdot \xi}\,\brkt{L^*}^Me^{i\vartheta(x,\xi)}\, \sigma_{j}^{\alpha,\nu}(x,\xi)  \ddd \xi.
\end{equation*}
This equality, the observation that $\supp \sigma_{j}^{\alpha,\nu}\subset C^{\nu}_{j},$ with $|C^{\nu}_j|=O(1)$ uniformly in $\nu$ and $j$, the estimates for the derivatives of $\vartheta$, and \eqref{amplitude derivative estim} yield 
\begin{equation}\label{main kernel estimate for AAW 1}
\abs{\partial_y^\alpha K_j^\nu (x,y)} 
\lesssim \frac {2^{j(m_2(p)+|\alpha|)}}{\jap{\nabla_\xi \varphi(x,\xi_j^\nu)-y}^{M}},
\end{equation}
for all multi-indices $\alpha$ and all $j\geq 0$.\\

Let $T_j^\nu$ be the operators corresponding to the kernels $K_j^\nu$ and $\at$ be a $p$-atom supported in the ball $B(\bar y,r)$ with $\bar{y} \in \mathbb{R}^n$ and $r>0$. Define
\eq{B^\nu_j := \set{x: |\nabla_\xi \varphi(x,\xi_j^\nu)-\bar y|\leq 2r}.}

Since $0<p<1$ we have
\begin{align}\label{eq:twoterms}
\norm{T_a^\varphi \at}^p_{L^p(\Rn)} 
& \lesssim \sum_j \sum_\nu 
\norm{{T_j^\nu} \at}^p_{L^p(\Rn)} \nonumber \\
& = \sum_j \sum_\nu \Big(\norm{T_j^\nu \at}^p_{L^p(B_j^\nu)}+ \norm{T_j^\nu \at}^p_{L^p(\Rn \setminus B_j^{\nu})}\Big).
\end{align}

We start with the first term in \eqref{eq:twoterms}. Since, by the SND-condition, the map \linebreak$x\mapsto \nabla_\xi \varphi(x,\xi_j^\nu)$ is a global diffeomorphism, one has that $|B_j^\nu| \lesssim r^n$ uniformly in $j$ and $\nu$. 
Therefore, the $L^2$-boundedness of $2^{-jm_2(p)}T_j^\nu$ proven in \cite{Fuj} and H\"older's inequality yield
\eq{\norm{T_j^\nu \at}_{L^p(B_j^\nu)} &\lesssim r^{n\brkt{1/p - 1/2}}2^{jm_2(p)}\norm{2^{-jm_2(p)}T_j^\nu \at}_{L^2(\Rn)} \\ 
&\lesssim r^{n\brkt{1/p - 1/2}}2^{jm_2(p)}\norm{\at}_{L^2(\Rn)}\lesssim 2^{jm_2(p)},}

where the $L^2$-boundedness of the second inequality is uniform in $j$ and $\nu$. This is because the symbol of $2^{-jm_2(p)}T_j^\nu$ fulfills 
\eq{
\left |  \partial_\xi^\alpha \partial_x^\beta (2^{-jm_2(p)} \,\psi_j(\xi)\,\chi_j^\nu(\xi) \, a(x,\xi))\right | \lesssim 1,
}
uniformly in $j$ and $\nu$ and is hence an element of $S^0_{0,0}(\Rl^n)$.\\

We turn to the second term in \eqref{eq:twoterms} and estimate that in two different ways. First, we observe that for $x\in \Rn \setminus B_j^{\nu}$ and $y\in B$ one has
\eq{
|\nabla_\xi \varphi(x,\xi_j^\nu)-y| 
& \geq \frac 12  |\nabla_\xi \varphi(x,\xi_j^\nu)-\bar y |.} 
Now using the SND-condition of the phase and \eqref{main kernel estimate for AAW 1} with $\alpha=0$, we obtain (taking $M$ large enough)
\nm{eq:schrodingerpieceestimate}{
\norm{T_j^\nu \at}_{L^p(\Rn \setminus B_j^{\nu})}^p 
&\lesssim \int_{\Rn \setminus B_j^{\nu}} 
\Big(\int_B |{ K_j^\nu(x,y)}||\at(y)|\dd y \Big)^p \dd x \\ 
&\lesssim \int_{\Rn} 
\Big(\int_B \frac {2^{jm_2(p)}r^{-n/p}}{\jap{\nabla_\xi \varphi(x,\xi_j^\nu)-\bar y}^{M}}\dd y \Big)^p \dd x 
\lesssim 2^{j(np-2n)}r^{np-n}.}

Second, if $r<1$, Taylor expansion of $K_j^\nu$ in the $y$-variable around $\bar y$, using the moment conditions of $\at$, and finally \eqref{main kernel estimate for AAW 1} yield 
that for $N:=[n(1/p-1)]$
\eq{\norm{T_j^\nu \at}_{L^p(\Rn \setminus B_j^{\nu})}^p 
&\lesssim \sum_{|\alpha|=N+1}
\int_{\Rn \setminus B_j^{\nu}} 
\Big(\int_B |{ \partial_y^\alpha K_j^\nu(x,y^*)}||y-\bar y|^{N+1}|\at(y)|\dd y \Big)^p \dd x \\ 
& \lesssim 2^{j(np-2n+p(N+1))}r^{np-n+p(N+1)},}
where $y^*$ is a point on the line segment connecting $y$ and $\bar{y}$. 
Note that we have also used that for $x\in \Rl^n\setminus 2B$, one has 
\begin{equation*}
    \vert \nabla_\xi \varphi(x,\xi_j^\nu)-\bar{y}\vert \lesssim \vert \nabla_\xi \varphi(x,\xi_j^\nu)-y^*\vert.
\end{equation*}
Since $r<1$, take the unique integer $\ell\in \mathbb Z_+$ such that $2^{-\ell-1}\leq r < 2^{-\ell}$. Then recalling that there are $O(2^{jn})$ terms in the sum in $\nu$, we have
\eq{
&\sum_{j=0}^\infty \sum_\nu \Big(\norm{T_j^\nu \at}^p_{L^p(B_j^\nu)}+ \norm{T_j^\nu \at}^p_{L^p(\Rn \setminus B_j^{\nu})}\Big) \\ 
&\qquad \lesssim \sum_{j\geq \ell} \sum_\nu \Big(2^{-2jn\brkt{1 - p/2}}+2^{j(np-2n)}r^{np-n}\Big)\\ 
&\qquad\qquad+ \sum_{j<  \ell} \sum_\nu \Big(2^{-2jn\brkt{1 - p/2}}+2^{j(np-2n+pN+p)}r^{np-n+pN+p} \Big)\\ 
&\qquad\lesssim \sum_{j\geq \ell}  
\Big(2^{-jn(1-p)}+2^{j(np-n)}r^{np-n} \Big) \\ 
&\qquad\qquad+ \sum_{j<  \ell}  
\Big(2^{-jn(1-p)}+2^{j(np-n+pN+p)}r^{np-n+pN+p} \Big)\\ 
&\qquad\lesssim 1 +2^{\ell(np-n)}r^{np-n}+2^{\ell(np-n+pN+p)}r^{np-n+pN+p}\sim 1.}

Now, if $r\geq 1$, we do the same calculation as above, except that we take $\ell=0$ and do not consider the case $j<\ell$. Hence, only \eqref{eq:schrodingerpieceestimate} is needed to estimate $\norm{T_j^\nu \at}^p_{L^p(\Rn \setminus B_j^{\nu})}$, and we conclude that $\norm{T_a^\varphi \at}^p_{L^p(\Rn)}$ is also  uniformly bounded when $r\geq 1$.\\

Interpolating this with the $L^2$-boundedness result in \cite{Fuj} yields the result for \linebreak$0<p\leq 2$.\\

For the $L^p$--boundedness of $T_a^\varphi$ in the range $2\leq p<\infty$, using Remark \ref{Rem:varphi0} we can without loss of generality assume that $\varphi(x,0)=0$ in $T_a^\varphi$. Now, using duality and interpolation, the $L^p$--boundedness of $T_a^\varphi$ (with this kind of phase function) would be a consequence of the $h^p(\Rl^n)$ to $L^p(\Rl^n)$ boundedness of the adjoint operator ${(T_a^{\varphi})}^*$, for $0<p\leq 2.$\\

Therefore we start by showing the $h^p(\Rl^n)$ to $L^p(\Rl^n)$ boundedness of the adjoint operator {${(T_a^\varphi)}^*$} (with $\varphi(x,0)=0$), for $0<p<1$ and make the following observations.
The kernel of ${(T_a^\varphi)}^*$ is given by
\begin{equation*}
K_{j}^{\nu*}(x,y) = \int_{\Rl^n}  \psi_j(\xi)\,\chi_j^\nu(\xi)\,e^{-i(\varphi(y,\xi )-x\cdot \xi)}\,\overline{a (y,\xi)}  \ddd \xi,
\end{equation*}
therefore for any multi-index $\alpha$ we have
\begin{align*}
\partial^{\alpha}_{y}K_{j}^{\nu*}(x,y) 
& = \int_{\Rl^n}  \psi_j(\xi)\,\chi_j^\nu(\xi)\,
\partial^{\alpha}_{y} \Big(e^{-i(\varphi(y,\xi )-x\cdot \xi)}\,\overline{a (y,\xi)}\Big) \ddd \xi \\
& = 
\int_{\Rl^n}
e^{-i(\varphi(y,\xi )-x\cdot \xi)}\,
\sigma_{j}^{\alpha,\nu*}(y,\xi) \ddd\xi,
\end{align*}
where
$$\sigma_{j}^{\alpha,\nu*}(y,\xi)
:=  \psi_j(\xi)\,\chi_j^\nu(\xi)\,\sum_{\substack{\alpha_1 + \alpha_2=\alpha \\ \lambda_1 + \dots + \lambda_r = \alpha_2} }
C_{\alpha_1, \alpha_2, \lambda_1, \dots \lambda_r}\,
\partial_y^{\alpha_1} \overline{a(y,\xi)}\,
\partial_y^{\lambda_1}\varphi(y,\xi)
\cdots
\partial_y^{\lambda_r}\varphi(y,\xi),
$$
and $|\lambda_j| \geq 1$.
Now, for $|\lambda_j + \beta| \geq 2$,
$$|\partial^{\lambda_j}_y \partial^{\beta}_{\xi} \varphi(y,\xi)|
\lesssim 1,$$ 
and using that $\varphi(y,0)=0$ and the mean-value theorem, we obtain $$| \nabla_y \varphi(y,\xi)|
\lesssim |\xi|.$$ 
From these estimates we deduce that for any multi-index $\gamma$ one has
$|\partial^{\gamma}_\xi\sigma_{j}^{\alpha,\nu*}(y,\xi)|\lesssim 2^{j(m_2(p)+|\alpha|)}$.
Therefore, following the same line of reasoning as for the case of $T^\varphi_a$ yields for all multi-indices $\alpha$ and all $j\geq 0$ that
$$|\partial^{\alpha}_{y}K_{j}^{\nu*}(x,y)|
\lesssim \frac {2^{j(m_2(p)+|\alpha|)}}{\jap{\nabla_\xi \varphi(y,\xi_j^\nu)-x}^{M}}.$$
Now the rest of the proof proceeds almost exactly as in the case of $T_a^\varphi$.\\

Having established the $h^p - L^p$ boundedness of  $(T_a^\varphi)^{\ast}$ for $0<p<1$, we can use interpolation to extend this to the desired range $0<p\leq 2$. Summing up, this (together with duality and interpolation) shows the $h^p$-$L^p$ boundedness of $T_a^\varphi$ for $0<p<\infty.$\\ 

Now for the boundedness of $T_a^\varphi$ from $L^\infty(\Rl^n)$ to $\bmo(\Rl^n)$ one can write \linebreak$T_a^\varphi= e^{i\varphi(x,0)} T_a^{\tilde{\varphi}}$ with  $\tilde{\varphi}(x,0)=0.$  Then given the assumption on the phase function of Schr\"odinger integral operators and the extra assumption $|\nabla_x\varphi(x,0)|\in L^{\infty}(\Rl^n)$ on the phase, one can use \eqref{poitwise multiiplier} to reduce matters to the boundedness of $T_a^{\tilde{\varphi}}$. But the boundedness of $T_a^{\tilde{\varphi}}$ from $L^\infty(\Rl^n)$ to $\bmo(\Rl^n)$ is a consequence of the boundedness of $(T_a^{\tilde{\varphi}})^{\ast}$ from $h^1(\Rl^n)$ to $L^1(\Rl^n)$ which is achieved in the same way, as in the analysis of $(T_a^\varphi)^{\ast}$ above. The details are left to the interested reader.
\end{proof}

\section{Action of parameter-dependent pseudodifferential operators on oscillatory integrals}\label{section:left comp of OIO with pseudo} 
Here we prove the  result concerning the composition of parameter-dependent pseudodifferential operators and oscillatory integral operators, and also derive an asymptotic expansion for the composition operator.
\begin{proof}[\textbf{\emph{Proof of Theorem \ref{thm:left composition with pseudo}}}]
The idea of the proof is similar to that of the asymptotic expansion proved in \cite{RRS}, however the details are somewhat different. Let $\chi(x-y)\in \mathcal{C}^{\infty}(\mathbb{R}^n \times \Rl^n)$ such that $0\leq \chi\leq 1,$ $\chi(x-y)\equiv 1$ for $|x-y|< \kappa/2$ and $\chi(x-y)=0$ for $|x-y|>\kappa$, for some small $\kappa$ to be specified later. We now decompose $\sigma_t(x,\xi)$ into two parts $\textbf{I}_1 (t,x,\xi)$ and $\textbf{I}_2 (t,x,\xi)$ where
\begin{equation*}
  \textbf{I}_1 (t,x,\xi)
  :=\iint_{\Rl^n\times \Rl^n} a(y,\xi)\,b(x, t\eta)\, (1-\chi(x-y))\,e^{i(x-y)\cdot\eta+i\phase(y,\xi)-i\phase(x,\xi)}\ddd\eta\dd y,
  \end{equation*}
  and
\begin{equation*}
  \textbf{I}_2 (t,x,\xi)
  :=\iint_{\Rl^n\times \Rl^n} a(y,\xi)\,b(x, t\eta)\,\chi(x-y)\,e^{i(x-y)\cdot\eta+i\phase(y,\xi)-i\phase(x,\xi)}\ddd\eta\dd y.
  \end{equation*}

\quad\\

\textbf{Step 1 -- The analysis of $\textbf{I}_1(t,x,\xi)$}\\
To this end, we introduce the differential operators
\[
L_{\eta} 
:=-i\frac{x -y}{|x-y|^2}\cdot \nabla_{\eta} \quad \text{and} \quad
L_{y} 
:=\frac{1}{\langle \nabla_{y}\phase(y,\xi)\rangle ^2 -i\Delta_{y}\phase(y,\xi)}(1-\Delta_{y}).
\]
Because of \eqref{asymptotic condition2}, one has 
$$|\langle \nabla_{y}\phase(y,\xi) \rangle^2 -i\Delta_{y}\phase(y,\xi)|\geq \langle \nabla_{y}\phase(y,\xi) \rangle^2 \gtrsim \xxi^2.$$
Now integration by parts yields
\eq{
  \textbf{I}_1 (t,x,\xi)&=\iint_{\Rl^n\times \Rl^n} (L_{y}^*)^{N_{2}} \set{e^{-iy\cdot\eta} \,a(y,\xi)\, (L_{\eta}^*)^{N_1}[(1-\chi(x-y))\,b(x, t\eta)]}\\ &\qquad \qquad \times e^{ix\cdot\eta+i\phase(y,\xi)-i\phase(x,\xi)} \ddd\eta \dd y.
}

Now since $0<t\leq 1$, provided $0<N_3<N_1-s$, we have
\begin{align*}
\abs{\partial^{N_1}_{\eta_j}b(x, t\eta)} & \lesssim t^{N_1} \langle t\eta\rangle^{s-N_1}= t^{N_1} \langle t\eta\rangle^{-N_3}\langle t\eta\rangle^{s-(N_1 -N_3)} \\
& \lesssim t^{N_{1}} \brkt{t^2+|t\eta|^2}^{-N_3/2} \langle t\eta\rangle^{s-(N_{1}-N_3)}\lesssim t^{N_{1}-N_3}\langle \eta\rangle^{-N_3}.
\end{align*}
Therefore, choosing $N_1 >n$ and $2N_2<N_3 -n$

\eq{
  | \textbf{I}_1 (t,x,\xi)| &\lesssim t^{N_1 -N_3}\xxi^{-2N_2 +m}\iint_{|x-y|>\kappa} \langle \eta\rangle^{2N_2} |x-y|^{-N_1} \langle\eta\rangle^{-N_3} \ddd\eta \dd y \\ &\lesssim  t^{N_1 -N_3} \xxi^{-2N_2 +m}.
}
Estimating derivatives of $\textbf{I}_1(t,x,\xi)$ with respect to $x$ and $\xi$ may introduce factors estimated by powers of $\xxi$, $\langle \eta\rangle$, and $|x-y|$, which can all be handled by choosing $N_1$ and $N_2$ appropriately. Therefore, for all $N$ and any $\nu>0$
\begin{equation*}
  \abs{ \partial^{\alpha}_{\xi} \partial^{\beta}_{x}\textbf{I}_1 (t,x,\xi)} \lesssim t^{\nu}\xxi^{-N},
  \end{equation*}
and so $\textbf{I}_1 (t,x,\xi)$ forms part of the error term $t^{\eps M}r(t,x,\xi)$ in \eqref{asymptotic expansion}.\\

\textbf{Step 2 -- The analysis of $\textbf{I}_2(t,x,\xi)$}\\
First, we make the change of variables $\eta=\nabla_x\phase(x,\xi)+\zeta$ in the integral defining $\textbf{I}_2 (t,x,\xi)$ and then expand $b(x, t\eta)$ in a Taylor series to obtain
\begin{equation*}
  \begin{aligned} b(x,t\nabla_x\phase(x,\xi)+t\zeta) & = \sum_{0 \leq |\alpha|<M} t^{|\alpha|}\frac{\zeta^\alpha}{\alpha !} \brkt{\partial_\eta^\alpha b}(x,t\nabla_x\phase(x,\xi)) \\ & \qquad+ t^{M} \sum_{|\alpha|=M} C_\alpha {\zeta^\alpha} r_\alpha(t, x,\xi,\zeta),
  \end{aligned} \end{equation*}
 where 
\begin{equation}\label{eq:ralapha}
    r_\alpha(t, x,\xi,\zeta) 
 := \int_0^1 (1-\tau)^{M-1} \brkt{\partial_\eta^{\alpha} b}(x, t\nabla_x\phase(x,\xi)+\tau t\zeta)\dd \tau.
\end{equation}
If we set
\[
\Phi(x,y,\xi)
:=\phase(y,\xi)-\phase(x,\xi)+(x-y)\cdot\nabla_x\phase(x,\xi),
\]
we obtain
\[
\textbf{I}_2 (t,x,\xi)= \sum_{|\alpha|<M} \frac{t^{\eps|\alpha|}}{\alpha!}\, \sigma_{\alpha}(t,x,\xi) + t^{\eps M}\, \sum_{|\alpha|=M} C_\alpha\, R_{\alpha}(t,x,\xi),
\]
where, using integration by parts, we have
\begin{align*}
\sigma_{\alpha}(t,x,\xi) 
& := t^{(1-\eps)\abs\alpha}\iint_{\Rl^n\times \Rl^n} e^{i(x-y)\cdot\zeta+i\Phi(x,y,\xi)} \,\zeta^{\alpha}\, a(y,\xi) \\
 &\qquad\qquad\qquad\times\chi(x-y)\, (\partial_\eta^\alpha b)(x, t\nabla_x\phase(x,\xi))\dd y\ddd\zeta \\
& =t^{(1-\eps)\abs\alpha}\brkt{\partial_\eta^\alpha b}(x,t\nabla_x\phase(x,\xi)) (i)^{-|\alpha|}\partial_y^{\alpha}\left[ e^{i\Phi(x,y,\xi)}\, a(y,\xi)\,\chi(x-y) \right]_{|_{_{y=x}}},
\end{align*}
and
\[
R_{\alpha}(t,x,\xi) 
:= t^{(1-\eps)\abs\alpha}\iint_{\Rl^n\times \Rl^n} e^{i(x-y)\cdot\zeta} e^{i\Phi(x,y,\xi)} \zeta^{\alpha}\,a(y,\xi)\, \chi(x-y) \, r_\alpha(t, x,\xi,\zeta)\, \dd y\ddd\zeta.
\]

\quad\\

\textbf{Step 2.1 -- The analysis of $\sigma_{\alpha}(t,x,\xi)$}\\
We now claim that
\begin{equation} \label{half}
\left|\partial_y^{\gamma} e^{i\Phi(x,y,\xi)}{}_{|_{y=x}}\right|
\lesssim \bra{\xi}^{|\gamma|/2}.
\end{equation}
We first observe that when $\gamma = 0$, \eqref{half} is obvious. To obtain \eqref{half} for $\gamma \neq 0$ we recall Fa\`a di Bruno's formulae
\[
\partial_y^{\gamma} e^{i\Phi(x,y,\xi)}=\sum_{\gamma_1 + \cdots+ \gamma_k =\gamma} C_\gamma \brkt{\partial^{\gamma_{1}}_{y}\Phi(x,y,\xi)}\cdots \brkt{\partial^{\gamma_{k}}_{y}\Phi(x,y,\xi)}\,e^{i\Phi(x,y,\xi)},
\]
where the sum ranges of $\gamma_j$ such that $|\gamma_{j}|\geq 1$ for $j=1,2,\dots, k$ and $\gamma_1 + \cdots+ \gamma_k =\gamma$ for some $k \in \Z_+$. Since $\Phi(x,x,\xi)=0$ and $\left.\partial_y\Phi(x,y,\xi)\right|_{y=x}=0$, setting $y=x$ in the expansion above leaves only terms in which $|\gamma_j|\geq 2$ for all $j = 1,2,\dots,k$. But $\sum_{j=1}^{k} |\gamma_j |\leq |\gamma|,$ so we actually have $2k\leq |\gamma|$, that is $k\leq  |\gamma|/2$. Estimate \eqref{asymptotic condition3} on the phase tells us that $|\partial^{\gamma_{j}}_{y}\Phi(x,y,\xi)| \lesssim \xxi$, so
\[
\left|\partial_y^{\gamma} e^{i\Phi(x,y,\xi)}{}_{|_{y=x}}\right| \lesssim \xxi \cdots \xxi \lesssim \xxi ^{k} \lesssim \bra{\xi}^{|\gamma|/2},
\]
which is \eqref{half}.\\

If we use the fact that $t\leq 1$ and the assumption $i)$ of Theorem \ref{thm:left composition with pseudo} on the phase function $\phase$, then we have
\eq{
|\sigma_{\alpha}(t,x,\xi)| &\lesssim t^{(1-\eps)\abs\alpha}\bra{t\nabla_x\phase(x,\xi)}^{s-|\alpha|} \bra{\xi}^{|\alpha|/2} \bra{\xi}^m
\\ &\lesssim  t^{(1-\eps)|\alpha|} \,\bra{t\xi}^{s-(1-\varepsilon)|\alpha|}\,\bra{t\xi}^{-\varepsilon|\alpha|}\, \bra{\xi}^{m+|\alpha|/2}
\\ &\lesssim t^{\min(s,0)}  \bra{\xi}^{s+m-(1/2-\varepsilon)|\alpha|},
}
when $|\alpha| > 0$.\\ 

By the assumptions of the theorem, the derivatives of $\sigma_{\alpha}$ with respect to $x$ or $\xi$ do not change the estimates when applied to $b$, and the same is true when derivatives are applied to $\partial_y^{\alpha} e^{i\Phi(x,y,\xi)}|_{y=x}$. Therefore, for all multi-indices $\beta$, $\gamma\in \mathbb{Z}_{+}$,
\[
\abs{\partial^{\beta}_{\xi} \partial^{\gamma}_{x} \sigma_{\alpha}(t,x,\xi)} \lesssim  t^{\min(s,0)}  \bra{\xi}^{s+m-(1/2-\varepsilon)|\alpha|-\rho\abs\beta},
\]
as required.\\

\textbf{Step 2.2 -- The analysis of $R_{\alpha}(t,x,\xi)$}\\
Take $g\in \mathcal{C}_c^\infty(\Rl^n)$ such that $g(x)=1$ for $|x|<\delta/2$ and $g(x)=0$ for $|x|>\delta$, for some small $\delta>0$ to be chosen later. We then decompose
\begin{equation*}\label{EQ:Ralphas}
 \begin{aligned} R_{\alpha}(t,x,\xi) 
 & = t^{(1-\eps)\abs\alpha}\iint_{\Rl^n\times \Rl^n} e^{i(x-y)\cdot\zeta} g\Big(\frac{\zeta}{\bra{\xi}}\Big) \\
 &\qquad\qquad\qquad\times \partial_y^{\alpha} \left[ e^{i\Phi(x,y,\xi)}\, \chi(x-y)\,\,a(y,\xi)\, r_\alpha(t,x,\xi,\zeta) \right] \dd y\ddd\zeta \\ 
 & \quad + t^{(1-\eps)\abs\alpha} \iint_{\Rl^n\times \Rl^n} e^{i(x-y)\cdot\zeta} \Big(1-g\Big( \frac{\zeta}{\bra{\xi}}\Big)\Big) \\
 &\qquad\qquad\qquad\times \partial_y^{\alpha}\left[  e^{i\Phi(x,y,\xi)}\,\chi(x-y)\,\,a(y,\xi)\, r_\alpha(t,x,\xi,\zeta)\right]\dd y \ddd\zeta \\
 &=: R_\alpha^I(t,x,\xi)
 + R_\alpha^{I\!\!I}(t,x,\xi). 
   \end{aligned}
   \end{equation*}

\quad\\

\textbf{Step 2.2.1 -- The analysis of $R^I_{\alpha}(t,x,\xi)$}\\
Note that the inequality
\[
\bra{\xi}\leq 1+|\xi|\leq \sqrt{2}\bra{\xi},
\]
and \eqref{asymptotic condition2} yield
$$
\bra{t\nabla_x\phase(x,\xi)+t\tau\zeta} 
\leq  (C_2\sqrt{2}+\delta)\bra{t\xi},  $$
and
\begin{align*}
\sqrt{2}\bra{t\nabla_x\phase(x,\xi)+t\tau\zeta} 
& \geq 1+|t\nabla_x\phase|-|t\zeta|  \\
& \geq 1+C_1|t\xi|-t\delta\bra{\xi} \\
& \geq (1-\delta)+(C_1-\delta)|t\xi| \geq (\min\{1,C_1\}-\delta)\bra{t\xi}.
\end{align*}

Therefore, if we choose $\delta<\min\{1,C_1\}$, then for any $\tau\in (0,1)$, $\bra{t\nabla_x\phase(x,\xi)+t\tau\zeta}$ and $\bra{t\xi}$ are equivalent.\\

This yields that for $|\zeta|\leq r\bra{\xi}$, $\partial^\beta_\zeta r_\alpha(t, x,\xi,\zeta)$  are dominated by $ t^{\abs\beta}\bra{t\xi}^{s-|\alpha|-|\beta|}.$ Furthermore, for $t\leq 1$, it follows from the representation \eqref{eq:ralapha} for $r_\alpha$ that
\begin{equation}\label{eq:estr}
\begin{aligned}
& \Big|\partial_\zeta^{\beta}\Big( g\Big( \frac{\zeta}{\bra{\xi}}\Big) r_\alpha(t,x,\xi,\zeta)\Big)\Big| 
 \lesssim\sum_{\gamma\leq\beta} \Big|\partial_\zeta^\gamma g\Big(\frac{\zeta}{\bra{\xi}}\Big) \partial_\zeta^{\beta-\gamma} r_\alpha(t,x,\xi,\zeta)\Big| \\
& \qquad \qquad \leq C_{\alpha,\beta} \sum_{\gamma\leq\beta}t^{|\beta|-|\gamma|}\bra{\xi}^{-|\gamma|} \bra{t\xi}^{s-|\alpha|-|\beta|+|\gamma|} \\
& \qquad \qquad \lesssim \sum_{\gamma\leq\beta}t^{\min(s,0)+|\beta|-|\gamma|-(1-\varepsilon)|\alpha|}\bra{\xi}^{-|\gamma|}\\ 
&\qquad \qquad \qquad \qquad \times \bra{\xi}^{s-(1-\varepsilon)|\alpha|}   t^{-(|\beta|-|\gamma|)}\bra{\xi}^{-(|\beta|-|\gamma|)}\\
& \qquad \qquad \lesssim t^{\min(s,0)-(1-\varepsilon)|\alpha|} \,\bra{\xi}^{s-(1-\varepsilon)|\alpha|-|\beta|}.
\end{aligned}
\end{equation}

At this point we also need estimates for $\partial_y^\alpha e^{i\Phi(x,y,\xi)}$ off the diagonal, that is, when $x\neq y.$ This derivative has at most $|\alpha|$ powers of terms $\nabla_y\phase(y,\xi)-\nabla_x\phase(x,\xi)$, possibly also multiplied by at most $|\alpha|$ higher order derivatives $\partial_y^\beta\phase(y,\xi)$, which can be estimated by $(|y-x|\bra{\xi})^{|\alpha|}$ using \eqref{asymptotic condition3}. The term containing the difference $\nabla_y\phase(y,\xi)-\nabla_x\phase(x,\xi)$ is the product of at most $|\alpha|/2$ terms of the type $\partial_y^\beta\phase(y,\xi)$, which can be estimated by $\bra{\xi}^{|\alpha|/2}$ in view of \eqref{asymptotic condition3}. These observations yield
\[
\abs{\partial_y^\alpha e^{i\Phi(x,y,\xi)}}\lesssim (1+|x-y|\bra{\xi})^{|\alpha|} \bra{\xi}^{|\alpha|/2},
\]
and therefore we also have
\begin{equation}\label{eq:ests}
\left|\partial_y^{\alpha}\left[ e^{i\Phi(x,y,\xi)} \chi(x-y) \right] \right|\lesssim (1+|x-y|\bra{\xi})^{|\alpha|}\bra{\xi}^{|\alpha|/2}.
\end{equation}

Let
\[
L_\zeta
:=\frac{(1-\bra{\xi}^2\Delta_\zeta)}{1+\bra{\xi}^2|x-y|^2}, \quad \text{so} \quad L_\zeta^N e^{i(x-y)\cdot\zeta}=e^{i(x-y)\cdot\zeta}.
\]
Integration by parts with $L_\zeta$ yields
\[
 \begin{aligned}
  R^I_{\alpha}(t,x,\xi) 
 & = t^{(1-\eps)\abs\alpha}\iint_{\Rl^n\times \Rl^n} \frac{e^{i(x-y)\cdot\zeta}\,\partial_y^{\alpha}\left[ \chi(x-y)\, a(y,\xi)\, e^{i\Phi(x,y,\xi)}\right]}{(1+\bra{\xi}^2 |x-y|^2)^N} \\
 &\qquad \qquad \times (1-\bra{\xi}^{2}\Delta_\zeta)^N \Big\{ g\Big(\frac{\zeta}{\bra{\xi}}\Big) \, r_\alpha(t,x,\xi,\zeta) \Big\} \dd y \ddd\zeta \\
 &  = t^{(1-\eps)\abs\alpha}\iint_{\Rl^n\times \Rl^n} \frac{e^{i(x-y)\cdot\zeta}\,\partial_y^{\alpha}\left[\chi(x-y)\,a(y,\xi)\,e^{i\Phi(x,y,\xi)}\right]}{(1+\bra{\xi}^2 |x-y|^2)^N}\\
 &\qquad \qquad \times \sum_{|\beta|\leq 2N} c_{\beta}\bra{\xi}^{|\beta|} \Big\{ \partial_\zeta^{\beta}\Big( g\Big(\frac{\zeta}{\bra{\xi}}\Big) r_\alpha(t,x,\xi,\zeta)\Big) \Big\} \dd y \ddd\zeta.
 \end{aligned}
 \]
Using estimates (\ref{eq:estr}), (\ref{eq:ests}) and that the size of the support of $g(\zeta/\bra{\xi})$ in $\zeta$ is bounded by $(\delta\bra{\xi})^n$, yield
\[
\begin{aligned}
\abs{R^I_{\alpha}(t,x,\xi)} & \lesssim t^{\min(s,0)}\,\sum_{|\beta|\leq 2N} \bra{\xi}^{n+|\beta|} \bra{\xi}^{-(1-\varepsilon)|\alpha|-|\beta|} \bra{\xi}^{|\alpha|/2+s+m} \\
&\qquad \qquad \times\int_{|x-y|<\kappa}\frac{(1+|x-y|\bra{\xi})^{|\alpha|}} {(1+\bra{\xi}^2 |x-y|^2)^N} \dd y \\ 
& \lesssim t^{\min(s,0)}\sum_{|\beta|\leq 2N} \bra{\xi}^{n+|\beta|} \bra{\xi}^{s-(1-\varepsilon)|\alpha|-|\beta|} \bra{\xi}^{|\alpha|/2+m} \\
&\qquad \qquad \times \langle \xi\rangle^{-n} \int_{0}^{\infty}\frac{\tau^{n-1}(1+\tau)^{|\alpha|}} {(1+\tau^2)^N} \dd \tau\\
& \lesssim t^{\min(s,0)} \bra{\xi}^{s+m-\brkt{1/2- \varepsilon}|\alpha|},
\end{aligned}
\]
if we choose $N>(n+|\alpha|)/2$,  and the hidden constants in the estimates are independent of $t$ (because of (\ref{eq:estr})). The derivatives of $R_{\alpha}^I(t,x,\xi)$ with respect to $x$ and $\xi$ give an extra power of $\zeta$ under the integral. This amounts to taking more $y$-derivatives, yielding a higher power of $\bra{\xi}.$ However, for a given number of derivatives of the remainder $R_{\alpha}^I(t,x,\xi)$, we are free to choose $M=|\alpha|$ as large as we like and therefore the higher power of $\bra{\xi}$ will not cause a problem. Thus for all multi-indices $\beta$, $\gamma,$ and $|\alpha|$ large enough we have 
\begin{equation*}
\abs{\partial^{\beta}_{\xi} \partial^{\gamma}_{x} R_{\alpha}^{I}(t,x,\xi)} \lesssim  t^{\min(s,0)} \bra{\xi}^{s+m-\brkt{1/2- \varepsilon}|\alpha|-\rho\abs\beta},
\end{equation*}
where the hidden constant in the estimate does not depend on $t$.\\

\textbf{Step 2.2.2 -- The analysis of $R_{\alpha}^{I\!\!I}(t,x,\xi)$}\\
Define
\[
\Psi(x,y,\xi,\zeta)
:=(x-y)\cdot\zeta+\Phi(x,y,\xi)= (x-y)\cdot(\nabla_x\phase(x,\xi)+\zeta)+\phase(y,\xi)-\phase(x,\xi).
\]
It follows from \eqref{asymptotic condition2} and \eqref{asymptotic condition3} that if we choose $\kappa <\delta/8C_0,$ then since $|x-y|<\kappa$ on the support of $\chi$, one has (using that we are in the region $|\zeta|\geq \delta\bra{\xi}/2$)
\begin{equation*}\label{eq:rho}
\begin{aligned}
|\nabla_y\Psi| & =|-\zeta+\nabla_y\phase-\nabla_x\phase|\leq 2C_2(|\zeta|+\bra{\xi}), \quad \text{and} \\
|\nabla_y\Psi| & \geq |\zeta|-|\nabla_y\phase-\nabla_x\phase| \geq \frac{1}{2}|\zeta|+\left(\frac{\delta}{4}-C_0|x-y| \right)\bra{\xi}\geq C(|\zeta|+\bra{\xi}).
\end{aligned}
\end{equation*}
Now, using \eqref{asymptotic condition3}, for any $\beta$ we have the estimate
\begin{equation}\label{no idea what to call 1}
\abs{\partial_y^\beta\brkt{e^{-i\Phi(x,y,\xi)}\,\d_y^{\gamma} e^{i\Phi(x,y,\xi)}}}\lesssim\bra{\xi}^{|\gamma|}.
\end{equation}
For $M=|\alpha|>s$ we also observe that
\begin{equation} \label{eq:rs}
|r_\alpha(t, x,\xi,\zeta)|\lesssim 1.
\end{equation}
For the differential operator defined to be 
$$L_y:=i|\nabla_y\Psi|^{-2}\sum_{j=1}^n (\partial_{y_j}\Psi) \,\partial_{y_j},$$ 
induction shows that $L_y^N$ has the form
\begin{equation*} (L_y^*)^N=\frac{1}{|\nabla_y\Psi|^{4N}}\sum_{|\beta|\leq N} P_{\beta,N}\,\partial_y^\beta, 
\end{equation*}
where
$$ P_{\beta,N}
:=\sum_{|\mu|=2N} c_{\beta\mu\delta_j}\,(\nabla_y\Psi)^\mu\, \partial_y^{\delta_1}\Psi\cdots \partial_y^{\delta_N}\Psi,$$
$|\delta_j|\geq 1$ and $\sum_{j=M}^{N}|\delta_j|+|\beta|=2N$. It follows from \eqref{asymptotic condition3} that $|P_{\beta,N}|\leq C(|\zeta|+\bra{\xi})^{3N}.$ Now Leibniz's rule yields
\begin{equation*}
\begin{aligned}
 R^{I\!\!I}_{\alpha}(t,x,\xi) 
 &= t^{(1-\eps)\abs\alpha}\iint_{\Rl^n\times \Rl^n} e^{i(x-y)\cdot\zeta} \,
 \Big(1-g\Big( \frac{\zeta}{\bra{\xi}}\Big)\Big)\, r_\alpha(x,\xi,\zeta) \\
 & \qquad \qquad \times\partial_y^{\alpha}\left[ e^{i\Phi(x,y,\xi)}\, a(y,\xi)\,\chi(x-y) \right] \dd y \ddd\zeta \\
& = t^{(1-\eps)\abs\alpha}\iint_{\Rl^n\times \Rl^n} e^{i\Psi(x,y,\xi,\zeta)} \,
\Big(1-g\Big(\frac{\zeta}{\bra{\xi}}\Big)\Big)\, r_\alpha(t,x,\xi,\zeta) \\
& \qquad\qquad \times \sum_{\gamma_1+\gamma_2 +\gamma_3=\alpha} \brkt{e^{-i\Phi(x,y,\xi)}\partial_y^{\gamma_1} e^{i\Phi(x,y,\xi)}}\, \partial^{\gamma_2}_{y}\chi(x-y) \,\partial^{\gamma_3}_{y} a(y,\xi)  \dd y \ddd\zeta \\
& = t^{(1-\eps)\abs\alpha}\iint_{\Rl^n\times \Rl^n} e^{i\Psi(x,y,\xi,\zeta)} |\nabla_y\Psi|^{-4N}\sum_{|\beta|\leq N} P_{\beta,N}(x,y,\xi,\zeta)\\
& \qquad\qquad \times \Big(1-g\Big(\frac{\zeta}{\bra{\xi}} \Big)\Big) r_\alpha(t,x,\xi,\zeta)  \sum_{\gamma_1+\gamma_2 +\gamma_3=\alpha} \partial_y^\beta \big [ \brkt{e^{-i\Phi(x,y,\xi)}}\\
& \qquad\qquad \times \partial_y^{\gamma_1} e^{i\Phi(x,y,\xi)}\, \partial^{\gamma_2}_{y}\chi(x-y)\, \partial^{\gamma_3}_{y} a(y,\xi)\big ] \dd y \ddd\zeta.
\end{aligned}
\end{equation*}
It follows now from \eqref{no idea what to call 1} and (\ref{eq:rs}) that
\[
\begin{aligned}
|R^{I\!\!I}_{\alpha}(t,x,\xi)| & \lesssim t^{(1-\eps)\abs\alpha}\int_{|\zeta|\geq \delta\bra{\xi}/2} \int_{|x-y|<\kappa} (|\zeta|+\bra{\xi})^{-N} \bra{\xi}^{|\alpha|+m}\dd y \ddd\zeta\\ 
&\lesssim t^{(1-\eps)\abs\alpha}\bra{\xi}^{|\alpha|+m} 
\int_{|\zeta|\geq \delta\bra{\xi}/2}|\zeta|^{-N}\ddd\zeta  \leq C \bra{\xi}^{|\alpha|+n+m-N},
\end{aligned}
\]
which yields the desired estimate when $N>|\alpha|+n$. For the derivatives of $R_{\alpha}^{I\!\!I}(t,x,\xi)$, we can get, in a similar way to the case for $R_{\alpha}^I$, an extra power of $\zeta$, which can be taken care of by choosing $N$ large and using the fact that $|x-y|<\kappa.$  Therefore for all multi-indices $\beta$, $\gamma\in \mathbb{Z}_{+},$
\begin{equation*}
            \abs{\partial^{\beta}_{\xi} \partial^{\gamma}_{x} R_{\alpha}^{I\!\!I}(t,x,\xi)} \lesssim \bra{\xi}^{|\alpha|+n+m-N},
\end{equation*}
where the constant hidden in the estimate does not depend on $t$. The proof of Theorem \ref{thm:left composition with pseudo} is now complete.
\end{proof}

\section{Regularity on Besov-Lipschitz spaces}
\label{sec:Besov}

In this section, we prove sharp boundedness results of oscillatory integral operators on Besov-Lipschitz spaces. The idea here is to boost all 
$h^p - L^p$ 
results in above sections to 
$B_{p,q}^{s+m-m_k(p)} - B_{p,q}^{s}$
using the calculus of Theorem \ref{thm:left composition with pseudo}. To this end, we prove the following proposition:

\begin{Prop}\label{Prop: hp-Lp}
Let $k \geq 1$, $0<p\leq \infty$ and $a(x,\xi) \in S^{m_k(p)}_{0,0}(\Rl^n)$ with compact support in the $x$-variable. Assume that $\varphi\in \textart F^k$ \emph{is SND} satisfies the $L^2$-condition \eqref{eq:L2 condition} and the \emph{LF}$(\mu)$-condition \eqref{eq:LFmu} for some $0<\mu\leq 1$.
If $0<p<\infty$, then \linebreak$T_a^\varphi :h^{p}(\Rl^n)\to L^{p}(\Rl^n)$ and for $p=\infty$ one has $T_a^\varphi :L^{\infty}(\Rl^n)\to \bmo (\Rl^n)$. If one removes the condition of compact support of $a(x,\xi)$ in $x$, then the aforementioned boundedness result is valid, but $p$ has to be taken strictly larger than $n/(n+\mu).$ In the case $0<k<1$, the results above are true provided that $a(x,\xi) \in S^{m_k(p)}_{1,0}(\Rl^n)$. 
\end{Prop}
\begin{proof}
For the high frequency portion of the operator (here is the compact support in the spatial variable not relevant), we use Propositions \ref{Th:new} and \ref{Tj:Tad global} to show that the operators $T^{\varphi}_a$ and ${(T^{\varphi}_a)}^*$ are bounded from $h^{p}(\mathbb{R}^n)$ to $L^{p}(\mathbb{R}^n)$ for all $0<p <1$. Observe that the condition $\partial^\beta_x \varphi(x,e_\ell) \in L^\infty(\mathbb{R}^n)$ is satisfied for all $\ell$ due to the LF$(\mu)$-condition. Now using analytic interpolation, duality and the $L^2$-boundedness provided in Theorem \ref{thm:fujiwara}, yields the desired result for the high frequency portion of the operator $T_a^{\varphi}$.\\

For the low and middle frequency portions of the operator, we just use Lemma \ref{Lem:lowfreq} and Lemma \ref{Lem:midfreq} in the Triebel-Lizorkin case with $s=0$ and $q=2$.
\end{proof}

\begin{Lem}\label{lem:besovuniformboundedness}
Let $k\geq 1$, $0<p\leq \infty$, $m \in \Rl$ and $a(x,\xi)\in S^{m}_{0,0}(\Rl^n)$. Assume that $\varphi$ \emph{is SND} and satisfies the $L^2$-condition \eqref{eq:L2 condition}. If $\psi_j$ is defined as in \emph{Definition \ref{def:LP},} then the operator $T_j$ given by
\begin{equation*}\label{eq:defTj}
T_j f(x) 
:= \int_{\Rl^n} e^{i\varphi(x,\xi)}\, a(x,\xi)\,\psi_j(\xi)\, \widehat f(\xi) \ddd \xi,
\end{equation*}
 satisfies
\begin{equation*}\label{eq:uniformboundednessinj}
\norm{ T_j f}_{L^p(\Rl^n)} \lesssim 2^{j(m-m_k(p))} \norm{\Psi_j(D)  f}_{L^p(\Rl^n)},
\end{equation*}
for $j\in \mathbb Z_+$, provided that one of the following holds true$:$
\begin{enumerate}
    \item[$i)$] $\varphi\in \textart F^k$, $a(x,\xi)$ is compactly supported in the $x$-variable and has frequency support in $\Rl^n \setminus B(0,R)$, for the $R$ given in \emph{Lemma \ref{Lem:saviouroftoday}}.
    \item[$ii)$] $\varphi\in \textart F^k$, $\partial^\beta_x \varphi(x,\xi) \in L^\infty(\mathbb{R}^n\times\mathbb S^{n-1})$, for any $|\beta| \geq 1$ and $a(x,\xi)$ has frequency support in $\Rl^n \setminus B(0,R)$, for the $R$ given in \emph{Lemma \ref{Lem:saviouroftoday}}.
    \item[$iii)$] $\varphi\in \textart F^k\cap \mathcal C^\infty(\Rn\times\Rn)$. 
    \item[$iv)$] $k= 2$ and $\varphi$ satisfies \eqref{Schrodinger phase}.
\end{enumerate}

If one removes the requirement on the frequency support of $a(x,\xi)$ and adds the \emph{LF}$(\mu)$-condition \eqref{eq:LFmu} in $i)-ii)$, then one obtains the result for $0<p<\infty$ in $i)$ and $n/(n+\mu)<p<\infty$ in $ii)$.\\ 

In the case $0<k<1$, the results for $i)-iii)$ above are true provided that $a(x,\xi) \in S^{m}_{1,0}(\Rl^n)$.
\end{Lem}
\begin{proof}

Using Remark \ref{Rem:varphi0} 
we can without loss of generality assume that $\varphi(x,0)=0$ in $iii)-iv)$. Observe that using the mean value theorem, and either $L^2$-condition \eqref{eq:L2 condition} or  \eqref{Schrodinger phase},  yields that $\partial^{\beta}_x \varphi(x,\xi) \in L^{\infty}(\Rl^n \times \mathbb{S}^{n-1})$ for $\abs\beta \geq 1.$\\

To simplify the calculation we set $\sigma(x,\xi) :=\jap{\xi}^{m_k(p) -m}\,a(x,\xi)$ so that $\sigma\in S_{0,0}^{m_k(p)}(\Rn)$.\\ 

We start with the case $p\in (0,\infty)$. Proposition \ref{Prop: hp-Lp} in $i)-iii)$ and Theorem \ref{thm:schrodinger} in $iv)$ yields that $T_\sigma^\varphi:h^p(\Rl^n)\to L^p(\Rl^n)$. Next, we use the definition of the local Hardy space $h^p(\Rl^n)$ (see Definition \ref{def:Triebel}) and Definition \ref{def:LP} to obtain
\begin{align*}
\norm{T_j f}_{L^p(\Rl^n)} 
&\lesssim 2^{j(m-m_k(p))}\norm{T_\sigma ^\varphi \psi_j(D) f}_{L^p(\Rl^n)}
\lesssim 2^{j(m-m_k(p))}\norm{ \psi_j(D) f}_{h^p(\Rl^n)} \\
&\sim 2^{j(m-m_k(p))}
\Big\| \Big(\sum_{\ell=0}^\infty  \abs{\Psi_\ell(D) \psi_j(D) f}^2 \Big)^{1/2}\Big\|_{L^p(\Rl^n)} \\
&\lesssim 2^{j(m-m_k(p))}\norm{\Psi_j(D)f}_{L^p(\Rl^n)}.
\end{align*}
If one removes the requirement on the frequency support of $a(x,\xi)$ and adds either \eqref{eq:lf1} or \eqref{eq:lf2} in $i)-ii)$, then Lemma \ref{Lem:lowfreq} $iii)$ yields the $h^p- L^p$ result for the low frequency part of $T_\sigma^\varphi$.\\

We turn to the case when $p=\infty$, which can only be proved under the assumption $\abs\xi\geq R$ in $i)-ii)$. Observe that Proposition \ref{Prop: hp-Lp} in $i)-iii)$ and Theorem \ref{thm:schrodinger} in case $iv)$ give us the $h^1-L^1$ boundedness of the adjoint operator ${(T_\sigma^\varphi)}^*$. We set $f_\ell:= \psi_\ell(D)f$. Now the assumptions on the phase and 
Lemma \ref{Lem:saviouroftoday} enable us to apply formula \eqref{eq:highfreq1} to $(\Psi_\ell(D) T_{\sigma}^{\varphi})^*$, which in turn yields that
\nm{eq:uniformLinftybounedness}{
\norm{T^*_j f}_{L^1(\Rl^n)} 
&= 2^{j(m-m_k(p))}\norm{\psi_j(D)\brkt{T_{\sigma}^{\varphi}}^*f}_{L^1(\Rl^n)} \\ 
& =  2^{j(m-m_k(p))} \Big\|\sum_{\ell=0}^\infty \psi_j(D)\brkt{T_{\sigma}^{\varphi}}^*\Psi_\ell(D)\psi_\ell(D)f  \Big\|_{L^1(\Rl^n)} \\
&\lesssim 2^{j(m-m_k(p))}\sum_{\abs\alpha < M}\Big\|\sum_{\ell=0}^{N}2^{- \ell\eps |\alpha|}\psi_j(D){(T^\varphi_{\sigma_{\alpha,\ell}})}^*f_\ell  \Big\|_{L^1(\Rl^n)} \\
&\qquad +2^{j(m-m_k(p))}
\Big\|\sum_{\ell=0}^{\infty}2^{-\ell\eps M} \psi_j(D) (T_{r_\ell}^\varphi)^* f_\ell\Big\|_{L^1(\Rl^n)} \\ 
& =: \mathrm{I}+\mathrm{II},
}
where $N<\infty$ by the properties of the Littlewood-Paley sums.
We consider the main terms I above. Observe that using the semi-norm estimate \eqref{eq:estimate1} for $\sigma_{\alpha,\ell}$, we can claim that
\begin{align*}
&\Big\|\sum_{\ell=0}^{N} 2^{- \ell\eps |\alpha|} \psi_j(D){(T^\varphi_{\sigma_{\alpha,\ell}})}^*f_\ell  \Big\|_{L^1(\Rl^n)} 
 \lesssim \sum_{\ell=0}^{N}\norm{ 2^{- \ell\eps |\alpha|} \psi_j(D){(T^\varphi_{\sigma_{\alpha,\ell}})}^*f_\ell  }_{L^1(\Rl^n)} \\
&\qquad \lesssim\sum_{\ell=0}^N  \norm{f_\ell}_{h^1(\Rl^n)} = \sum_{\ell=0}^N\Big\| \Big(\sum_{j=0}^\infty  \abs{ \psi_j(D)\Psi_\ell(D) f}^2 \Big)^{1/2}\Big\|_{L^1(\Rl^n)}
\lesssim \norm{f}_{L^1(\Rl^n)}.
\end{align*}
To prove this claim, we first observe that 
$\psi_j(D)$ maps $L^1$ into itself (with a norm independent of $j$). Then we use Proposition \ref{Prop: hp-Lp} in $i)-iii)$ and Theorem \ref{thm:schrodinger} in case $iv)$ to obtain the desired result.\\

For the remainder term II we use the representation
\eq{
\brkt{T_{r_\ell} ^\varphi}^* f(x) = \int_{\Rl^n} K(x,y) f(y) \dd y.
}
Then in case $i)-iii)$ Integration by parts yields
\begin{align*}
\abs{K(x,y)} 
&= \Big|\int_{\Rl^n} e^{-i(x-y)\cdot \xi}\,e^{i\varphi(y,\xi)-iy\cdot\xi}\,\overline{r_\ell (y,\xi)} \ddd\xi \Big| \\ 
&\lesssim \jap{x-y}^{-2N} \int_{\Rl^n} \abs{(1+i(x-y)\cdot\nabla_\xi)^N\, \Big(e^{i\varphi(y,\xi)-iy\cdot\xi}\,\overline{r_\ell (y,\xi)}\Big)} \ddd\xi \\
&\lesssim {\jap{x-y}^{-2N}} \int_{|\xi| \gtrsim 1} \sum_{\abs{\alpha_1}+\dots+ \abs{\alpha_{N}}\leq N}  \abs{\partial_\xi^{\alpha_N} \overline{r_\ell (y,\xi)}}\\ &\qquad\qquad \qquad \qquad \times\prod_{\nu=1}^{N-1} C_{\alpha}\abs{(x-y)^{\alpha_\nu } {\partial_\xi^{\alpha_\nu} (\varphi(y,\xi)-y\cdot\xi )}}  \ddd\xi.
\end{align*}
Now since by estimate \eqref{eq:estimate3} we have that $r_\ell \in S^{m-\brkt{1/2- \varepsilon}M}_{0,0}{(\Rn)}$, choosing $M$ large enough, the $\textart F^k$-condition yields that
\eq{
\abs{K(x,y)} \lesssim \jap{x-y}^{-N},
}
for any $N>0$. In case $iv)$ we estimate
\eq{
\abs{K(x,y)} 
&= \abs{\int_{\Rl^n} e^{i(x-\nabla_\xi\varphi (y,0))\cdot \xi}\,e^{i\nabla_\xi\varphi (y,0)\cdot\xi-i\varphi(y,\xi)}\,\overline{ {r_\ell}(y,\xi)} \ddd\xi} \\ 
&\lesssim  \jap{x-\nabla_\xi\varphi (y,0)}^{-2N} \int_{\Rl^n} \Big|(1+i(x-\nabla_\xi\varphi (y,0))\cdot\nabla_\xi)^N \\
& \qquad \qquad \times e^{i\varphi(y,\xi)-i\nabla_\xi\varphi (y,0)\cdot\xi}\,\overline{ {r_\ell}(y,\xi)} \Big| \ddd\xi \\
&\lesssim \jap{x-\nabla_\xi\varphi (y,0)}^{-2N} \int_{\Rl^n} \sum_{\abs{\alpha_1}+\dots+ \abs{\alpha_{N}}\leq N}  \abs{\partial_\xi^{\alpha_N} \overline{ {r_\ell}(y,\xi)}}\\
&\qquad\qquad\times\prod_{\nu=1}^{N-1} \abs{(x-\nabla_\xi\varphi (y,0))^{\alpha_\nu } \,\partial_\xi^{\alpha_\nu} (\varphi(y,\xi)-\nabla_\xi\varphi (y,0)\cdot\xi )}  \ddd\xi .
}
Here we observe that 
\begin{equation*}
    |\partial_\xi^{\alpha_\nu} (\varphi(y,\xi)-\nabla_\xi\varphi (y,0)\cdot\xi )|
    = \left\{
    \begin{array}{ll}
         O(1), &  \abs{\alpha_\nu} \geq 2,\\
         O(|\xi|), &  \abs{\alpha_\nu}=1, 
    \end{array}
    \right.
\end{equation*}
where we have used the fact that when $\abs{\alpha_\nu}\geq 2$, then \eqref{Schrodinger phase}  yields the first estimate and when $\abs{\alpha_\nu}=1$, then the mean-value theorem yields the second.
 Therefore, once again choosing $M$ large enough, we have for any $N>0$ that
 $$\abs{K(x,y)}\lesssim  \jap{x-\nabla_\xi\varphi (y,0)}^{-N}$$ 
and hence 
\eq{
\norm{\brkt{T_{r_\ell}^\varphi}^* f}_{L^1(\Rl^n)} 
\lesssim 
\norm f_{L^1(\Rl^n)}. 
}
Now we estimate the remainder term of \eqref{eq:uniformLinftybounedness}. It is bounded by
\eq{
\sum_{\ell=0}^\infty 2^{-\ell\varepsilon M}\norm{ \psi_j(D) \brkt{T_{ {r_\ell}}^\varphi}^* f_\ell}_{L^1(\Rl^n)} \lesssim \sum_{\ell=0}^\infty 2^{-\ell\varepsilon M}\norm{f}_{L^1(\Rl^n)} \lesssim \norm{f}_{L^1(\Rl^n)} .
}
Therefore,
$$\norm{T^*_jf}_{L^1(\Rl^n)} \lesssim \norm{f}_{L^1(\Rl^n)},$$
and a duality argument yields
\eq{\norm{T_j f}_{L^\infty(\Rl^n)} =\norm{T_j \Psi_j(D) f}_{L^\infty(\Rl^n)} \lesssim \norm{\Psi_j(D)f}_{L^\infty(\Rl^n)}.}
Now if one removes the requirement on the frequency support of $a(x,\xi)$ and adds either \eqref{eq:lf1} or \eqref{eq:lf2} in $i)-ii)$, then Lemma \ref{Lem:lowfreq} $i)$ yields $L^\infty - \bmo$ boundedness and the rest of the argument proceeds as before.
\end{proof}

Now we are finally ready to prove the regularity of oscillatory integral operators on Besov-Lipschitz spaces. 

\begin{proof}[\emph{\textbf{Proof of Theorems \ref{thm:main4} and \ref{thm:main4glob}, part \emph{i})}}]
\label{main highfreq Besov oscillatory}
For the low and middle frequency portions of the operator, we just use Lemma \ref{Lem:lowfreq} and Lemma \ref{Lem:midfreq} parts $ii)$  and $iii)$. Observe that $\partial^\beta_x \varphi(x,\xi) \in L^\infty(\mathbb{R}^n\times\mathbb S^{n-1})$ is a consequence of the LF$(\mu)$-condition \eqref{eq:LFmu}. Thus from now on we concentrate on the high frequency portion of the operator. We divide the proof into three steps. In Step 1 we invoke a composition formula which yields a sum of two terms (a main term and a remainder term) that need to be analysed separately, and conclude that the main term  is $L^p$-bounded (in the sense of Lemma \ref{lem:besovuniformboundedness}). In Step 2 we show $B_{p,q}^s-L^p$ boundedness for the remainder term, and in Step 3 we complete the proof by deducing the $ B_{p,q}^{s+m-m_k(p)} - B_{p,q}^s$ boundedness. In Step 4 we deal with the case when the phase function is smooth everywhere in Theorem \ref{thm:main4glob}.\\

\textbf{Step 1 -- A composition formula and boundedness of the main term}\\ In the definition of the Besov-Lipschitz norm, the expression $\psi_j(D)T_a^{\varphi}f$ plays a central role. To obtain favourable estimates for $\psi_j(D)T_a^{\varphi}f$ we use formula \eqref{eq:highfreq1} with $M$ chosen large enough, which states
\begin{equation}\label{eq:new}
\psi\left (2^{-j}D \right ) T_a ^\varphi = \sum_{|\alpha|\leq M-1} \frac{2^{-j\eps|\alpha|}}{\alpha!}T_{\sigma_{\alpha,j}}^{\varphi} + 2^{-j\eps M}T_{ r_j}^{\varphi}.
\end{equation}

From Lemma \ref{lem:besovuniformboundedness} we have, after a change of variables, that
\begin{equation} \label{eq:highfreq8}
\begin{split}
\norm{T_{\sigma_{\alpha,j}}^{\varphi} f}_ {L^p(\Rl^n)} \lesssim 2^{j(m-m_k(p))}   \norm{\Psi_j(D)f}_{L^p(\Rl^n)}.
\end{split}
\end{equation}

\textbf{Step 2 -- The remainder term}\\
We decompose $T_{ r}^\varphi$ of \eqref{eq:new} into Littlewood-Paley pieces as follows:
\eq{T_{  r_j}^{\varphi} f(x) 
= \sum_{{\ell=0}}^\infty \int_{\Rl^n} e^{i\varphi(x,\xi)}\,  r_j(x,\xi)\,\psi_\ell(\xi)\,\widehat{f}(\xi) \ddd\xi 
=: \sum_{{\ell=0}}^\infty  T^{\varphi}_{ {r}_{j,\ell}}f(x),}
where the $\psi_\ell$'s are given in Definition \ref{def:LP}.
We use the fact that for $0<p\leq\infty,$
\begin{equation}\label{quasibanch triangel}
\norm{f+g}_{L^p(\Rl^n)} \leq 2^{C_p}\left (\norm{f}_{L^p(\Rl^n)}+\norm{g}_{L^p(\Rl^n)}\right ),
\end{equation}
where $ C_p:=\max (0,1/p -1 )$.
Now Fatou's lemma and iteration of \eqref{quasibanch triangel} yield that
\begin{align*}
\norm{T_{r_j}^{\varphi} f}_{L^p(\Rl^n)}
&=\Big\| \sum_{\ell=0}^\infty T^{\varphi}_{ {r}_{j,\ell}} f \Big\|_{L^p(\Rl^n)} 
\leq \liminf_{N\to\infty}
\Big\|\sum_{\ell=0}^N T^{\varphi}_{ {r}_{j,\ell}} f \Big\|_{L^p(\Rl^n)} \\
& \lesssim \liminf_{N\to\infty}\sum_{\ell=0}^N2^{\ell C_p}
\| T^{\varphi}_{ {r}_{j,\ell}} f\|_{L^p(\Rl^n)} 
\lesssim\sum_{\ell=0}^\infty 2^{\ell C_p}
\|T^{\varphi}_{ {r}_{j,\ell}}f\|_{L^p(\Rl^n)},
\end{align*}
where the hidden constant in the last estimate depends only on $p$. Therefore, applying Lemma \ref{lem:besovuniformboundedness} with $m-\brkt{1/2- \varepsilon}M $ instead of $m$ (recall that $r_j$ vanishes for all $\xi$, for which $a$ vanishes), we obtain
\begin{equation}\label{eq:highfreq6}
\begin{split}
\|T_{  r_j}^{\varphi} f\|_{L^p(\Rl^n)} 
&\lesssim \sum_{\ell =0}^\infty 2^{\ell C_p}
\|T^{\varphi}_{ {r}_{j,\ell}}f\|_{L^p(\Rl^n)} \\ &\lesssim \sum_{\ell =0}^\infty 2^{\ell \left (C_p+m-m_k(p)-\brkt{1/2- \varepsilon}M  \right )}\norm{\Psi_\ell (D)f}_{L^p(\Rl^n)}.
\end{split}
\end{equation}
Note that the estimate \eqref{eq:highfreq6} is uniform in $j$. Now we claim that
\begin{equation}\label{eq:highfreq9}
T_{  r_j}^{\varphi}:B_{p,q}^s(\Rl^n)\to L^p(\Rl^n).
\end{equation}
To see this, we shall analyse the cases $0<q<1$ and $1\leq q \leq \infty$ separately. Starting with the former, we have
\begin{equation*}
\begin{split}
\norm{T_{r_j}^{\varphi} f}_{L^p(\Rl^n)} \lesssim 
&\sum_{\ell =0}^\infty 2^{\ell \left (C_p+m-m_k(p)-\brkt{1/2- \varepsilon}M \right )}\norm{\Psi_\ell (D)f}_{L^p(\Rl^n)}  \\  \lesssim & \sum_{\ell =0}^\infty 2^{\ell (s+m-m_k(p))}\norm{\Psi_\ell (D)f}_{L^p(\Rl^n)} \\ 
\leq &\Big(\sum_{\ell =0}^\infty 2^{\ell q(s+m-m_k(p))}\norm{\Psi_\ell (D)f}_{L^p(\Rl^n)}^q\Big)^{1/q}
= \norm{f}_{B_{p,q}^{s+m-m_k(p)}(\Rl^n)},
\end{split}
\end{equation*}
where we used (\ref{eq:highfreq6}) for the first inequality and that $M$ is large enough for the second. 
For $1\leq q \leq \infty$, H\"older's inequality in the sum over $\ell$ and picking $M$ large enough yield
\begin{equation*}
\begin{split}
\norm{T_{r_j}^{\varphi}f}_{L^p(\Rl^n)} &\lesssim \sum_{\ell =0}^\infty 2^{\ell \left (C_p+m-m_k(p)-\brkt{1/2- \varepsilon}M \right )}\norm{\Psi_\ell (D)f}_{L^p(\Rl^n)} \\ &= \sum_{\ell =0}^\infty 2^{\ell \left (-s +C_p-\brkt{1/2- \varepsilon}M \right )}\brkt{2^{\ell (s+m-m_k(p))}\norm{\Psi_\ell (D)f}_{L^p(\Rl^n)}}  \\ 
&\lesssim \Big( \sum_{\ell =0}^\infty 2^{\ell q'\left (-s +C_p-\brkt{1/2- \varepsilon}M \right )}\Big)^{1/q'} \Big( \sum_{\ell =0}^\infty 2^{\ell q(s+m-m_k(p))}\norm{\Psi_\ell (D)f}^{q}_{L^p(\Rl^n)}\Big)^{1/q} \\ 
& \lesssim \norm{f}_{B_{p,q}^{s+m-m_k(p)}(\Rl^n)},
\end{split}
\end{equation*}
which implies \eqref{eq:highfreq9}. Note that the calculation above also holds for $q=\infty$ with the usual interpretation of H\"older's inequality. \\\\

\textbf{Step 3 -- The $\mathbf{B_{p,q}^{s+m-m_k(p)} - B_{p,q}^s}$ boundedness}\\
The results in (\ref{eq:highfreq8}) and (\ref{eq:highfreq9}) yield that
\begin{align*}
&\norm{T_a^\varphi f}_{B_{p,q}^{s}(\Rl^n)} 
= \Big(\sum_{j=0}^\infty \left (2^{js}\norm{\psi\brkt{2^{-j}D}T_a^\varphi f}_{L^p(\Rl^n)}\right )^q\Big)^{1/q}\\ 
& \quad \lesssim  \Big (\sum_{j=0}^\infty \Big (\sum_{|\alpha|\leq M-1} 2^{js}\norm{T_{\sigma_{\alpha,j}} f}_{L^p(\Rl^n)} + 2^{-j(\varepsilon M-s)}\norm{T_{r_j} f}_{L^p(\Rl^n)}\Big )^q \Big )  ^{1/q}  \\ 
& \quad \lesssim  \Big( \sum_{j=0}^\infty \left ( 2^{j(s+m-m_k(p))}\norm{ \Psi_j(D)f}_{L^p(\Rl^n)} + 2^{-j\brkt{\varepsilon M-s}} \norm{ f}_{B_{p,q}^{s+m-m_k(p)}(\Rl^n)}\right  )^q\Big)^{1/q} \\
& \quad \lesssim \Big (\sum_{j=0}^\infty  2^{jq(s+m-m_k(p))}\norm{ \Psi_j(D)f}_{L^p(\Rl^n)}^q + \sum_{j=0}^{\infty}2^{-jq\brkt{\varepsilon M-s}}\norm{ f}_{B_{p,q}^{s+m-m_k(p)}(\Rl^n)}^q \Big )^{1/q}\\ 
& \quad \lesssim \norm{ f}_{B_{p,q}^{s+m-m_k(p)}(\Rl^n)}.
\end{align*}\hspace*{1cm}

\textbf{Step 4 -- The smooth case}\\
For the smooth version we don't need separate proofs for low, middle and high frequencies. Note that $T_a^\varphi= e^{i\varphi(x,0)} T_a^{\tilde{\varphi}}$, where $\tilde{\varphi}(x,0)= 0.$ Therefore, using the condition $|\nabla_x\varphi(x,0)|\in L^{\infty}(\Rl^n)$ and the $L^2$-condition (for all $x$ and $\xi$), we have by \eqref{poitwise multiiplier} that 
$\Vert T_a^\varphi f\Vert_{B^{s}_{p,q}}\lesssim \Vert T_a^{\tilde{\varphi}}f\Vert_{B^{s}_{p,q}}.$ 
Now apply Lemma \ref{lem:besovuniformboundedness} $iii)$ and Lemma \ref{Lem:saviouroftoday} $iii)$ to $T_a^{\tilde{\varphi}}$ and continue as above and the proof is complete.
\end{proof}

We can also establish the boundedness of Schr\"odinger integral operators on Besov-Lipschitz spaces. 

\begin{proof}[\emph{\textbf{Proof of Theorems \ref{thm:main5} and \ref{thm:main5glob}, part \emph{i})}}] 
Theorem \ref{thm:main5} is a special case of Theorem \ref{thm:main5glob} so it is enough to consider the latter. This is identical to Step 4 in the previous proof, except that Lemma \ref{lem:besovuniformboundedness} $iv)$ is used instead. 

\end{proof}

\section{Regularity on Triebel-Lizorkin spaces}
\label{sec:Triebel}

In this section we prove various Triebel-Lizorkin regularity results as corollaries of the previous Besov-Lipschitz results. We observe that, if we do not let the order $m$ of the amplitude to go all the way to the endpoint, then we have Triebel-Lizorkin boundedness for all $p$'s and $q$'s. 

\begin{proof}[\emph{\textbf{Proof of Theorems \ref{thm:main4},  \ref{thm:main5}, \ref{thm:main4glob} and \ref{thm:main5glob}, part \emph{ii})}}]

Using the embedding \eqref{embedding of TL}, equality \eqref{equality of TL and BL},
and part \emph i) of the theorems, we have that
\begin{align*}
\Vert T_a^\varphi f\Vert_{F^{s}_{p,q}(\Rl^n)}
&\lesssim \Vert T_a^\varphi f\Vert_{F^{s+\varepsilon/2}_{p,p}(\Rl^n)}
\lesssim \Vert f\Vert_{F^{s+m-m_k(p)+\varepsilon/2}_{p,p}(\Rl^n)} \\
& \lesssim \Vert f\Vert_{F^{s+m-m_k(p)+\varepsilon}_{p,q}(\Rl^n)}. \qedhere
\end{align*}
\end{proof}

\begin{proof}
[\textbf{\emph{Proof of Theorems \ref{thm:main4}, \ref{thm:main5}, \ref{thm:main4glob} and \ref{thm:main5glob}, 
parts \emph{iii}) and \emph{iv})}}]
We divide the \\
proof into different steps.\\

\textbf{Step 1 -- The diagonal $\mathbf{p=q}$}\\
The theorem is true for the diagonal $p=q$ because of the Besov-Lipschitz results in 
Theorems \ref{thm:main4}--\ref{thm:main4glob}, \ref{thm:main5glob}, part \emph i) and the fact that $F_{p,p}^s(\Rl^n)=B_{p,p}^s(\Rl^n)$.\\

\textbf{Step 2 -- The $\mathbf{h^p -h^p}$ boundedness}\\
 For Theorems \ref{thm:main4} and \ref{thm:main4glob} $iii)$--$iv)$, we split the proof into low, middle and high frequency parts. The low and middle frequency parts were treated in Lemma \ref{Lem:lowfreq} and Lemma \ref{Lem:midfreq}. Observe that $\partial^\beta_x \varphi(x,\xi) \in L^\infty(\mathbb{R}^n\times\mathbb S^{n-1})$ is a consequence of the LF$(\mu)$-condition \eqref{eq:LFmu}. \\
 
For the high frequency cases, $a_H(x,\xi) := (1-\psi_0(\xi/R)\,a(x,\xi)$, recall that Proposition \ref{Prop: hp-Lp} yields the $h^p - L^p$ boundedness, which we will now lift  to the $h^p- h^p$ level.\\

To this end, it is enough to show that if $b(D)$ is a Fourier multiplier with $b\in S^0_{1,0}(\Rl^n)$, and $t$ a parameter in $(0,1]$, then the composition $b (tD)T_{a_H}^{\varphi}$ is $h^p - L^p$ bounded with a norm that doesn't depend on $t$. But this is indeed the case, since using the composition formula \eqref{asymptotic expansion} with $M=1$ we see that
$$b (tD)T_{a_H}^{\varphi}= T_{a_Hb(t\cdot)}^{\varphi} + t^{\eps} T_{r}^\varphi, $$
where $\abs{\partial^{\alpha}_{\xi} \partial^{\beta}_{x} r(t,x,\xi)} \leq C_{\alpha, \beta} \,\bra{\xi}^{m_k(p)-(1/2-\eps)}.$ Now since $a_H b(t\cdot)\in S^{m_{k}(p)}_{0,0}(\Rl^n)$ uniformly in $t\in (0,1]$,  Proposition \ref{Prop: hp-Lp} yields the $h^p - L^p$ boundedness of $ b (tD)T_{a_H}^\varphi$ with a norm that is independent of $t$, and the proof for the oscillatory integral operators is concluded.\\  

For Schr\"odinger integral operators (Theorems \ref{thm:main5} and \ref{thm:main5glob} $iii)$--$iv)$) and the smooth version of Theorem \ref{thm:main4glob}, there is no need to divide the amplitude different frequency portions, and we once again note that $T_a^\varphi= e^{i\varphi(x,0)} T_a^{\tilde{\varphi}}$ where $\tilde{\varphi}(x,0)= 0.$ Therefore, using the condition $|\nabla_x\varphi(x,0)|\in L^{\infty}(\Rl^n)$ and condition \eqref{Schrodinger phase}, we see by \eqref{poitwise multiiplier} and the definition of the local Hardy space as a Triebel-Lizorkin space that 
{$\Vert T_a^\varphi f\Vert_{h^p}\lesssim \Vert T_a^{\tilde{\varphi}}f\Vert_{h^p}.$} Now using Lemma \ref{Lem:saviouroftoday}, Theorem \ref{thm:left composition with pseudo}, and Theorem \ref{thm:schrodinger}, we can proceed as above to show the $h^p - L^p$ boundedness of $ b (tD)T_{a}^{\tilde{\varphi}}$ (for $0<p<\infty$) with a norm that is independent of $t$, and the proof for the Schr\"odinger integral operators is also concluded.\\

\textbf{Step 3 -- Boosting $\mathbf{F^0_{p,2}}$-boundedness to arbitrary regularity} \\
Once again for oscillatory integral operators, we decompose into low, middle and high frequency portions.
For the low and middle frequency parts, we apply Lemma \ref{Lem:lowfreq} and Lemma \ref{Lem:midfreq}. For the high frequency parts, we proceed as follows.
Write $a_H(x,\xi)= \sigma(x,\xi) \jap{\xi}^{m-m_k(p)}$, with $\sigma(x,\xi) \in S^{m_k(p)}_{0,0}(\Rl^n),$ and use Theorem \ref{thm:left composition with pseudo} to conclude that $(1-\Delta)^{s/2}T_\sigma^\varphi (1-\Delta)^{-s/2}$ is the same kind of oscillatory integral operator as $T_\sigma^\varphi$. Therefore by Step 2 above
\eq{
\norm{T_{a_H}^\varphi f}_{F^{s}_{p,2}(\Rl^n)} &= \norm{(1-\Delta)^{s/2}
{T_\sigma^\varphi} (1-\Delta)^{-s/2}(1-\Delta)^{(m-m_k(p)+s)/2}f}_{F^{0}_{p,2}(\Rl^n)}\\
&\lesssim \norm{(1-\Delta)^{(m-m_k(p)+s)/2}f}_{F^{0}_{p,2}(\Rl^n)} = \norm{f}_{F^{s+m-m_k(p)}_{p,2}(\Rl^n)}.
}
Now for the Schr\"odinger integral operator case and the smooth phase function case, there is no need to decompose the operator into high and low frequency cases, instead we just use \eqref{poitwise multiiplier} to once again reduce to the case of $T_{a_H}^{\tilde{\varphi}}$ for which it is true, thanks to Lemma \ref{Lem:saviouroftoday}, that 
$(1-\Delta)^{s/2}T_\sigma^{\tilde{\varphi}} (1-\Delta)^{-s/2}$ is the same kind of oscillatory integral operator as $T_\sigma^{\tilde{\varphi}}.$ Therefore we can once again run the same argument as above and achieve the desired result.\\

\textbf{Step 4 -- Interpolation}\\
By interpolation in $q$ we get the desired result. Note that one cannot interpolate between Triebel-Lizorkin spaces when $p=\infty$.
\end{proof}

\begin{figure}[ht!]
\centering\includegraphics[scale=.4]{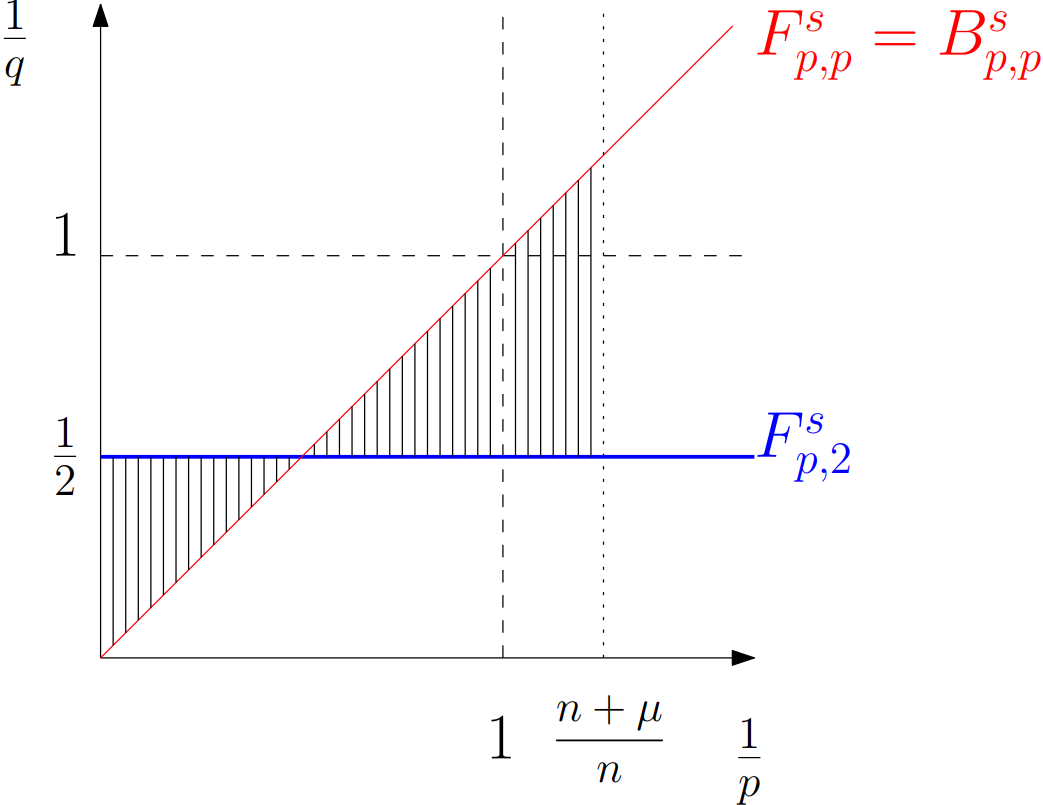}
\caption{Boundedness and interpolation scheme in Triebel-Lizorkin scale.}
\end{figure}

\begin{proof}[\emph{\textbf{Proof of Theorem \ref{Th:bonus}, part \emph{i}) 
}}]
Write $a(x,\xi)= \sigma (x,\xi) \jap{\xi}^{m_2(p)}$ for 
$\sigma \in S^0_{0,0}(\Rl^n)$, and let $b(D)$ be any pseudodifferential operator of order zero. Then \linebreak\cite[Theorem 6.1]{AF} asserts that $[b(tD),T_{\sigma}^\varphi]= T^{\varphi}_{r}$,  for some $r(t,x, \xi) \in S^0_{0,0}(\Rl^n)$ uniformly in $t$ if $t\in (0,1]$. Therefore, we have that
\begin{align*}
b(tD)T_a^\varphi f(x)
& =b(tD)T_{\sigma}^\varphi (1-\Delta)^{{m_2(p)}/2}f(x)\\
& = T_{\sigma}^\varphi b(tD)(1-\Delta)^{{m_2(p)}/2}f(x)+ T_{r}^\varphi (1-\Delta)^{{m_2(p)}/2}f(x),
\end{align*}
with $r(x, \xi,t) \in S^0_{0,0}(\Rl^n)$ uniformly in $t\in (0,1]$.
Then since $b\in S^{0}_{1,0}(\Rl^n)$ we have that $\sigma(x,\xi)\,b(t\xi) \jap{\xi}^{m_2(p)}\in S^{{m_2(p)}}_{0,0}(\Rl^n)$  uniformly in $t\in (0,1]$ and also $r(t,x,\xi)\jap{\xi}^{m_2(p)}\linebreak\in S^{{m_2(p)}}_{0,0}(\Rl^n)$ uniformly in $t\in (0,1]$. Therefore, we can apply Theorem \ref{thm:schrodinger} to conclude the desired result.
\end{proof}
\section{Sharpness of the results}
\label{sect:Sharp}

Let us start from a naive approach to the regularity problem of oscillatory integral operators, by considering a concrete case of an oscillatory integral operator, namely
$$Tf(x)
:= \int_{\mathbb{R}^n} |\xi|^{m} \, (1-\psi_0(\xi))\, e^{ix\cdot\xi+ i|\xi|^k}\,\widehat{f}(\xi) \, \dd \xi,$$
with $\psi_0$ as in Definition \ref{def:LP}. \\

Now, if we look upon $T$ as a $\Psi$DO with symbol $$a_{k,m}(\xi):=e^{i|\xi|^k }\,  (1-\psi_0(\xi))\, |\xi|^{m},$$ then we see that this symbol does not belong to any H\"ormander class $S^{m}_{\rho,\delta}(\Rl^n)$ for any $\rho \in [0,1]$, since $|\partial^{\alpha} a_{k,m}(\xi)|\lesssim \langle \xi\rangle^{m +(k-1)|\alpha|}$. Therefore, the appeal to the boundedness theory of pseudodifferential operators fails in a rather drastic way.\\

To understand the significance of the order $m_k (p)= -kn\left |1/p-1/2\right |,$ let 
\begin{equation*}
\begin{split}
K_{k,m}(x):=\int_{\mathbb{R}^n} (1-\psi_0(\xi))\,|\xi|^{m}\, e^{ix\cdot\xi+|\xi|^k}\, \ddd \xi.
\end{split}
\end{equation*}
Let $1<p<\infty$ and 
$$f_\lambda(x):= \int_{\mathbb{R}^n} (1-\psi_0(\xi))\,|\xi|^{-\lambda}\,e^{ix\cdot\xi}\, \ddd \xi.$$
It was shown in \cite[p. 302]{Miyachi1} that $f_\lambda \in L^p(\mathbb{R}^n)$ iff $-\lambda < n/p-n$. Now, if $m>m_k(p)$ and if $\lambda$ is such that $-\lambda <n/p-n$ and $-m +\lambda -n +nk/2<n(k-1)/p$, then $f_{\lambda}\in L^p(\Rl^n)$, but $Tf_\lambda (x)=(K_{k,m}\ast f_{\lambda})(x)\notin L^p(\Rl^n)$, see \cite[p. 301, (I-ii)]{Miyachi1}.\\

This shows that, if we regard the operator $T$ above as an oscillatory integral operator with the amplitude  $(1-\psi_0(\xi))\,|\xi|^{m}\in S^m_{1,0}(\Rl^n),$  and the phase function $x\cdot\xi+|\xi|^k$, then one can not in general expect any $L^p$-boundedness, unless $m\leq m_k(p)$ and thus this order of the amplitude is sharp for the $L^p$-regularity of $T$.

\bibliographystyle{siam}
\bibliography{references}

\end{document}